\documentclass{article}
\usepackage[utf8]{inputenc}
\usepackage{bbm}
\usepackage{multirow}
\usepackage{lineno}

\usepackage[overload]{empheq}
\usepackage{mathrsfs}
\usepackage{amsfonts}
\usepackage{amsmath}
\usepackage{amssymb}
\usepackage{dsfont}
\usepackage{diagbox}
\usepackage{graphicx}

\usepackage{tikz}
\usepackage{caption}
\usepackage{subcaption}
\usepackage{placeins, microtype}
\usepackage{latexsym,amsthm,xcolor,graphicx, longtable, tabu}
\usepackage{hyperref}
\usepackage{natbib}
\usepackage{mathtools}
\mathtoolsset{showonlyrefs}

\makeatletter
\def\theorem@space@setup{
  \theorem@preskip=0.4cm 
  \theorem@postskip=\theorem@preskip 
}
\makeatother

\newtheorem{theorem}{Theorem}[section]
\newtheorem{definition}[theorem]{Definition}

\newtheorem{assumption}[theorem]{Assumption}
\newtheorem{lemma}[theorem]{Lemma}

\newtheorem{remark}[theorem]{Remark}
\newtheorem{proposition}[theorem]{Proposition}
\newtheorem{corollary}[theorem]{Corollary}

\newcommand\Z{\mathbb{Z}}
\newcommand\N{\mathbb{N}}
\newcommand\R{\mathbb{R}}
\newcommand\PP{\mathbb{P}}
\newcommand{\En}{\mathcal{E}_n}

\newcommand{\Sn}{\mathcal{S}_n}

\newcommand{\toL}{\,{\buildrel {d} \over \longrightarrow}\,}

\newcommand{\lp}[1]{\left(#1\right)}
\newcommand{\lb}[1]{\left[#1\right]}
\newcommand{\la}[1]{\left|#1\right|}
\newcommand{\lcb}[1]{\left\{#1\right\}}

\newcommand{\norm}[1]{\left|\left|#1\right|\right| }

\long\def\metanote#1#2{{\color{#1}\
\ifmmode\hbox\fi{\sffamily\mdseries\upshape [#2]}\ }}

\numberwithin{equation}{section}
\newcounter{keepeqno}
\usepackage{fancyhdr}
\title{The quenched structured coalescent for diploid population models on finite graphs with large migrations and uneven offspring distributions}

 \author{ 
 Maximillian Newman\footnote{Department of Genetic Medicine, University of Chicago, Chicago, IL, USA. Email: \texttt{mnewman98@uchicago.edu}}
} 

\date{\today}

\begin{document}

\captionsetup{width=0.85\textwidth}

\maketitle

\abstract{ 
    In this work we describe a new model for the evolution of a diploid structured population backwards in time that allows for large migrations and uneven offspring distributions, and that generalizes both the mean-field model of Birkner et al. [\textit{Electron. J. Probab.} 23: 1-44 (2018)] and the haploid structured model of M\"{o}hle [\textit{Theor. Popul. Biol.} 2024 Apr:156:103-116]. We show convergence, with mild conditions on the joint distribution of offspring frequencies and migrations, of gene genealogies conditional on the pedigree to a time-inhomogeneous coalescent process driven by a Poisson point process $\Psi$ that records the timing and scale of large migrations and uneven offspring distributions. This quenched scaling limit demonstrates a significant difference in the predictions of the classical annealed theory of structured coalescent processes. In particular, the annealed and quenched scaling limits coincide if and only if these large migrations and uneven offspring distributions are absent. The proof proceeds by the method of moments and utilizes coupling techniques from the theory of random walks in random environments. Several examples are given and their quenched scaling limits established.
}
\section{Introduction}
    
    Consider a sample of $n$ copies of a gene from $n$ distinct individuals in a population. These samples, the result of generations of genetic inheritance backwards in time, trace their roots back to a common ancestor. The structure of how these genes trace their roots backwards in time is called the gene genealogy and is the central object of interest in coalescent theory. The seminal work of Kingman \cite{kingman78, kingman1982} established, marginally, at each place in the genome of a well-mixed population of constant, large size undergoing neutral evolution, that gene genealogies take the structure of a binary tree, where the binary mergers occur uniformly at random among each possible pair of extant lineages and each with the same rate. This tree structure, now known as the Kingman coalescent, is the fundamental universality class of coalescent theory, arising in one way or another in a slew of population models. Following Kingman's work, several now-classical works in coalescent theory \cite{pitman1999, sagitov1999general, schweinsberg} have described all possible scaling limits of genealogies in well-mixed population models as $\Xi$-coalescent processes, a class of trees that admit simultaneous multiple mergers.

    While extremely powerful as a tool in population genetics \cite{wakeley2016coalescent, Rosenberg2002}, and with rich mathematical ties to statistical physics, interacting particle systems, diffusion processes, and beyond \cite{bolthausenschnitsman, berestyckirecentprogress, bertoinlegall04}, the scaling limits of the classical theory implicitly describe the \textit{average} genealogy at a locus. In reality, every locus in the genome shares the same history of reproductive relationships, modeled by a random graph called the pedigree. Genealogies at every place in the genome are correlated through the random structure of this pedigree. The correct starting point for the coalescent, if we wish to understand the \textit{distribution} of gene genealogies across the genome, is to understand the law of the coalescent conditional on the pedigree. We refer to this conditional law as the law of the \emph{quenched coalescent}. Beginning with the non-rigorous works \cite{wakeleyetal2012, WakeleyEtAl2016, WiltonEtAl2017}, pedigree effects were observed in the recent past and in structured populations. This includes significant differences in the site-frequency spectrum and distribution of pairwise coalescence times (see \cite[Figure 3]{WiltonEtAl2017}, for example.) Recently, mathematically rigorous scaling limits have begun to emerge from quenched coalescent theory that are quite distinct from those of the classical theory.

    \cite{tyukin15} established the first result for quenched coalescents, showing that in a well-mixed population in the absence of uneven offspring distributions, gene genealogies of unlinked loci really are independent in the large population limit. However, the works \cite{dfbw24, abfw25} then showed the effects of uneven offspring distributions in the absence of selfing on gene genealogies. The pedigree leaves a persistent imprint on the dynamics of these gene genealogies through prodigious progenitors, individuals whose genetic contributions are shared across the genome. The limiting quenched genealogy becomes time inhomogeneous, evolving with an ordinary Kingman-like background punctuated by rare generations in which prodigious individuals contribute a non-vanishing proportion of the genetic material of the population even in the infinite population limit. Under the same assumptions of \cite{birkner2018} that yield an annealed multiple merger coalescent, gene genealogies conditional on the pedigree converge to a coalescent process driven by a Poisson point process recording the timing and structure of exceptional reproductive generations of the pedigree. Between exceptional generations, lineages merge only in binary mergers at a constant rate. At exceptional generations, many lineages may merge simultaneously, with the merger pattern determined by the ordered offspring frequencies. This identifies the precise sense in which widely separated loci are coupled; they share the same record of macrosopic pedigree events, these prodigious progenitors. This Poisson driven viewpoint is closely related to the limiting object in the present paper.

    While well-mixed populations are a natural starting point for coalescent analysis, many real population exhibit population structure. They are subdivided into demes connected by migration. In such settings, ancestral lineages carry locations and move between demes backward in time according to some migration mechanism. Coalescence can occur only when lineages occupy the same deme, with rates governed by deme sizes. This yields the structured coalescent\cite{herbots97, notohara90}, which has become the baseline approximation for inference in spatially distributed data and metapopulations\cite{muller17, guo22}. Recent literature has characterized multi-type $\Lambda$- and $\Xi$-coalescents as the most general universality classes of structured populations. These are models in which genealogies are subject to migrations that move them between types and where, within each type, the genealogies look like $\Xi$-coalescents. \cite{eldon09} proved convergence of a Moran model with large offspring numbers to a multi-type $\Lambda$-coalescent. \cite{johnston23} classifies exchangeable, consistent, and asynchronous multi-type coalescent processes as multi-type $\Lambda$-coalescents where mergers and migrations do not occur simultaneously. \cite{mohle24} analyzes multi-type Cannings models with conservative migration, demonstrating convergence to more general multi-type exchangeable coalescents that can admit simultaneous multiple mergers. In \cite{daipra25}, they study a structured population undergoing bottlenecks, proving that the limiting annealed genealogies are multi-type $\Xi$-coalescents characterized by simultaneous multiple mergers and migrations. 
    
    Much like the well-mixed case, the structured coalescent of \cite{notohara90, herbots97} and multi-type $\Xi$ coalescents \cite{eldon09, mohle24, daipra25} describe an averaged single locus genealogy. They do not condition on the single realized history of the pedigree. In structured populations this distinction becomes especially natural. Large migration waves can occur on the same time scale as coalescence and can affect many lineages at once. Such events are recorded in the pedigree and are shared across loci. They therefore create an additional source of time inhomogeneity and cross locus dependence that is invisible to purely annealed descriptions. \cite{WiltonEtAl2017} provided simulation results that demonstrated the importance of the pedigree in structured populations, where deviations from standard coalescent predictions can persist longer than in well-mixed populations due to the specific history of migration events captured by the pedigree.

    \subsection{The contributions of this paper}

        This work develops a quenched coalescent theory for a diploid structured population modeled on an arbitrary finite directed graph $G = \lp{V,E}$. Each vertex $v$ in $V$ represents a deme of the population and it is kept a constant size of $\lfloor s(v) N \rfloor$ backwards in time, where $N$ is a population size parameter we will send to infinity. Each generation draws a random proportion $m_{(v,w)}$ of the population at deme $v$ to travel to deme $w$. The total population at each deme $v$ then coalesces according to some exchangeable, within each deme, offspring distribution $\mathcal{V}^v$ in a manner similar to \cite{birkner2018, abfw25}. This model differs from those of \cite{mohle24, daipra25} by allowing migrations to be arbitrarily large and non-conservative in the sense of \cite{notohara16}, encoding diploidy, and allowing generic offspring frequencies, including not-necessarily independent offspring frequencies across demes, simultaneously. This model includes the models of \cite{birkner2018, abfw25} as special subcases, and generalizes the model of \cite{mohle24} to diploid populations and random migration sizes at each time-step. We denote the ordered offspring frequencies in each deme $v$ by $\widetilde{\mathcal{V}}^v$, and the collection of these over all demes by $\widetilde{\mathcal{V}} = \lp{\widetilde{\mathcal{V}}^v}_{v \in V}$.

        We show that as long as large migrations and uneven offspring distributions are suitably rare, described by the limiting joint distribution of $\lp{\widetilde{\mathcal{V}}, m}$ satisfying some local integrability condition, the law of the coalescent conditional on the pedigree converges in distribution to a random measure governing a time inhomogeneous structured coalescent. More precisely, there is some ``neutral" migration rate $\mu_e$ along every edge $e$ of $E$, and some ``neutral" coalescence rate $\kappa_v$ within each deme $v$ so that any block of the coalescent in deme $v$ moves independently of all the other blocks to deme $w$ at rate $\mu_{v,w}$, and any pair of blocks in deme $v$ coalesce independently of any of the other pairs of blocks at rate $\kappa_v$. At exceptional generations, governed by a Poisson point process $\Psi$, the coalescent jumps according to a paintbox construction, generalizing the Kingman $x$-paintbox described, for example, in \cite{berestycki04}. The limiting conditional gene genealogy is described by the law of a coalescent following the neutral migration and coalescence rates, jumping according to a paintbox construction at atoms of $\Psi$, all conditional on $\Psi$. This model is a natural extension of the limiting model of \cite{abfw25}, where there are no migration rates and where our $\kappa_v$ for the single deme in their model corresponds to their $c_{\text{pair}}$. 
    
        The proof proceeds by the method of moments. By a theorem of \cite{nfw25_2}, it is enough to prove that, as random variables, the finite-dimensional distributions of the coalescent \emph{conditional on the pedigree} converge in distribution to those of the limiting quenched process. Since these conditional finite-dimensional distributions are uniformly bounded random variables, we can show convergence by showing all moments of these conditional finite-dimensional distributions converge. It will suffice to show that for $l$ realizations of the coalescent conditionally independent with respect to the pedigree that the joint law converges to the same joint law as for $l$ realizations of the time-inhomogeneous coalescent conditionally independent with respect to $\Psi$. The biological interpretation of the method is then as follows: the distribution of gene genealogies across unlinked loci converges in distribution to a given (random) law if and only if the \textit{annealed} distribution of gene genealogies across $l$ unlinked loci converges to that of the annealed distribution of that given law across $l$ unlinked loci for any finite number of $l$ unlinked loci.
    
        We control these joint laws by constructing a coupling of the $l$ conditionally independent copies on the same pedigree, in the same spirit as \cite{abfw25}, who built on earlier work of \cite{bs02, birkner13}. The coupling makes transparent the two mechanisms that drive the limit. Firstly, genuinely \emph{large} pedigree events that change the coalescent state, such as macroscopic migration waves or unusually large families, are rare on the coalescent time scale. However, they affect all copies in nearly the same way. In particular, the different copies jump to corresponding new states with asymptotically matching probabilities. Second, during the overwhelming majority of \emph{small} generations, the copies experience the ``neutral" migration patterns described above, and their jump decisions become asymptotically independent across copies. This approximate independence follows from straightforward moment bounds. These two observations form the basis for a ``separation of scales" argument in which we bridge these two regimes to demonstrate the desired convergence of the joint law. We note here that there is an implicit separation of scales argument in the $\varepsilon$-naive-coalescent of \cite{abfw25}, though the total form of the argument is distinct. A strength of this method is that it only relies on some crude moment bounds and law of large numbers type arguments, which makes the analysis particularly simple and general.

    \subsection{Organization of the paper}

        The paper is organized into sections as follows. We begin in Section~\ref{S: model_and_coalescent} by introducing a diploid bi-parental population model for a finite structured population with size scaling with a parameter $N$. This population model gives rise to a random graph $\mathcal{G}_N$, called the pedigree, that tracks the total history of reproductive relationships in the population. We proceed then to describe how the ancestral process, which tracks how a sample of size $n$ traces its genetic inheritance in the population backwards in time, is determined by the structure of a family of coalescing random walks on $\mathcal{G}_N \times \{0,1\}$. The law of this ancestral process, conditional on the pedigree, is the central object of focus in this work. In Section~\ref{S: annealed_result} we describe the annealed limit of the ancestral process. It is here that we outline the essential assumptions necessary for our main theorem to hold. This annealed result follows as a corollary of the main result and its proof is contained in Section~\ref{App: annealed_convergence}.

        In Section~\ref{S: main_result} we describe the limit of the law of the ancestral process conditional on the pedigree. There is a Poisson point process $\Psi$ that governs the joint distribution of offspring frequencies and migration proportions between generations and a coalescent process $\chi^n$ driven by $\Psi$ as follows: Between atoms of $\Psi$ the ancestral process evolves as in the classical structured coalescent. That is, lineages move independently along an edge at a given rate $\mu_e$ and pairs of lineages in the same deme $v$ coalesce at a given rate $\kappa_v$. At the atoms of $\Psi$, the transition matrix for the ancestral process is described by means of a paintbox construction, akin to that of the Kingman paintbox \cite{berestycki04}. This paintbox corresponds, in the finite population model, to the existence of migrations on the order of the size of the total population or to prodigious progenitors who contribute a positive proportion of the total populations genetic information. The limiting \textit{quenched} law is then the law of $\chi^n$ conditional on $\Psi$. Because $\Psi$ may in general have infinitely many atoms, we rigorously construct a pathwise description of the time-inhomogeneous coalescent built on $\Psi$ in Section~\ref{SS: driven_coupling}.

        In Section~\ref{S: proof_main} we outline the proof of the main result. This consists of a series of lemmas making the method of moments and separation of scales argument rigorous. The proofs of these lemmas are largely postponed to Section~\ref{S: coupling}. Therein we rigorously describe how to couple $l$ realizations of the discrete-time ancestral process on the same pedigree, why the small-scale behavior of our model looks like the classical structured coalescent picture, and why the large-scale behavior of our model affects all of the $l$ realizations in, asymptotically, the same manner. It also contains a separation of timescales argument that allows us to simplify the state space of our coupling.

        In Section~\ref{S: examples_and_apps} we demonstrate convergence of the conditional gene genealogies for several examples discrete-time population models of interest. The examples include a two deme Wright-Fisher model with neutral migration analogous to the model of \cite{WiltonEtAl2017}, a two deme model analogous to that of \cite{dfbw24}, a model converging to a beta coalescent with beta distributed migrations, and what we call a discrete approximation of a spatial $\Xi$ Fleming-Viot genealogies on the $2$-torus.

        In Section~\ref{S: discussion}, we contrast our results with the work in \cite{WiltonEtAl2017} and provide a possible way to harmonize our paper and theirs. We then quickly describe some opportunities for future work.

\section{A diploid structured population model with migration and its coalescent}\label{S: model_and_coalescent}

    In Section~\ref{SS: model} we construct a diploid structured population with large migrations and uneven offspring distributions on a generic finite directed graph. We proceed to describe the pedigree and the corresponding ancestral process generated by this model in Section~\ref{SS: pedigree_coalescent}.

    \subsection{A diploid structured population with large migrations and uneven offspring distributions}\label{SS: model}
    
        Let $G =\lp{V,E}$ denote a finite directed graph. The vertices shall correspond to population demes and the directed edges $E$ will correspond to migration routes between demes. To each vertex $v$ in $V$ we assign a relative size $s(v) \in (0,\infty)$. Each deme $v$ will be kept a constant size $N(v) = \lfloor s(v) N \rfloor$. Let $V_N$ denote the vertex set
        \begin{equation*}
            \{(v,i): v \in V, i \in [N(v)]\},
        \end{equation*}
        where $[k] = \{1,2,\ldots, k\}$ for any natural number $k$. The population model will be governed by the joint distribution of a sample of a $V$-tuple of random symmetric matrices $\mathcal{V} = \lp{\mathcal{V}^v}_{v \in V} = \lp{\lp{\mathcal{V}_{i,j}^v}_{i,j = 1}^{N(v)}}_{v \in V}$ and a random variable $m = \lp{m_{e}}$ taking values in $[0,1]^{E}$. $m$ denotes the migration sizes between demes and $\mathcal{V}$ denotes the \textit{offspring numbers} of couples in each deme. We assume the following:
        \begin{itemize}
            \item Both $V$ and $E$ are finite.
            \item There are no edges from a deme to itself, i.e. for all $v$ in $V$ the edge $(v,v)$ is not an element of $E$.
            \item The entries of $\mathcal{V}$ are all non-negative integers.
            \item The total number of migrants from each deme cannot exceed the population of that deme, i.e. $\sum_{e = (v,w) \in E} m_e \leq 1$ for every deme $v$.
            \item We exclude selfing, i.e. $\mathcal{V}_{i,i}^v = 0$ for all $v \in V$ and $i$ in $[N(v)]$.
            \item The total number of offspring in each deme $v$ satisfies $\sum_{i < j}^{N(v)} \mathcal{V}_{i,j}^v = N(v) + \sum_{e=(w,v) \in E} \lfloor m_{e} N(w)\rfloor - \sum_{e=(v,w)\in E} \lfloor m_{e} N(v)\rfloor$.
            \item The matrix $\mathcal{V}$ is exchangeable under any permutations that fix demes conditional on $m$. That is,let $S_r$ denote the group of permutations on $r$ elements and $\sigma = \lp{\sigma(v)}_{v \in V} \in \prod_{v \in V} S_{N(v)}$. Then
            \begin{equation*}
                \lp{\lp{\lp{\mathcal{V}^v_{i,j}}_{i,j \in [N(v)]}}_{v \in V}, m} \stackrel{d}{=} \lp{\lp{\lp{\mathcal{V}^v_{\sigma(v)(i),\sigma(v)(j)}}_{i,j \in [N(v)]}}_{v \in V}, m}.
            \end{equation*}
        \end{itemize}
    
        We consider discrete generations indexed by $k \in \Z_+ = \{0,1,2,\ldots\}$, where $k = 0$ denotes the present generation and $k = 1$ is the previous generation, and so on backwards in time. Let $\{(\mathcal{V}(k), m(k))\}_{k \in \Z_+}$ denote a sequence of independent and identically distributed random variables that have the same distribution as $(\mathcal{V},m)$. Reproduction dynamics at each time-step are independent. They involve one stage of migration and one stage of reproduction as follows:
    
        For each $k \in \Z_+$ and each edge $e = (v,w)$ in $E$, $\lfloor N(v) m_{e}(k) \rfloor$ individuals from deme $v$ are chosen uniformly at random without replacement to migrate to deme $w$. Each deme $v$, after migration, therefore contains
        \begin{equation*}
            N^*(v)(k) = N(v) + \sum_{e=(w,v) \in E} \lfloor m_{e}(k) N(w)\rfloor - \sum_{e=(v,w)\in E} \lfloor m_{e}(k) N(v)\rfloor
        \end{equation*}
        individuals. These $N^*(v)$ individuals are the children of parental couples from the previous time-step $k+1$ following the random matrix $\mathcal{V}^{v}(k)$. This may be viewed, for instance, as having each of the $N^*(v)$ children be a ball thrown into the $\binom{N(v)}{2}$ parental couples, viewed as boxes, where the total count in the $(i,j)$ box is conditioned to be $\mathcal{V}_{i,j}^v$. An example of the reproductive dynamics this model may generate is shown in Figure~\ref{F: pedigree_example}.
        
        \begin{figure}[ht!]
            \centering
            \includegraphics[width=0.8\linewidth]{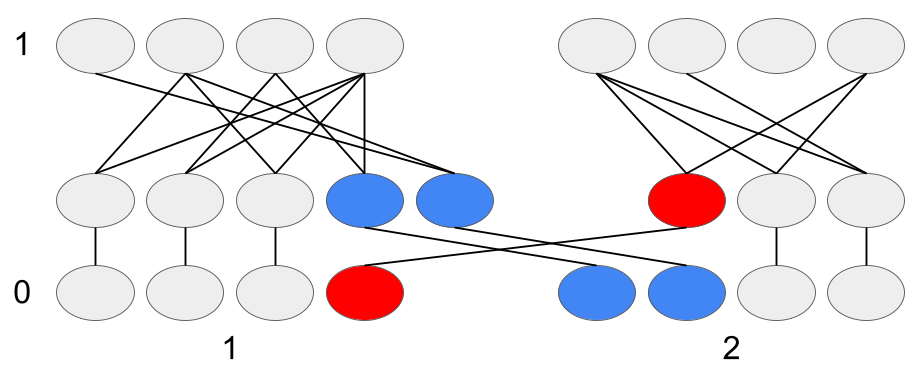}
            \caption{Here we see two demes containing $4$ individuals each so that $N(1) = N(2) = 4$. One individual from deme $1$, marked in red, migrates to deme $2$, and two individuals from deme $2$, marked in blue, migrate to deme $1$, so that $N^*(1)(0) = 5$ and $N^*(2)(0) = 3$. The two demes then experience some reproductive event, with the parental relationships tracked by two edges emanating from a child to their two parents. For example, the couple $(1,2)$ have a single child, the second individual in blue, and so $\mathcal{V}_{1,2}^1 =1$. At the same time the couple $(2,4)$ have two children together, the third and fifth individuals after migration. Therefore $\mathcal{V}_{2,4}^1 = 2$. The third individual in deme $2$ has no offspring, so $\mathcal{V}^2_{3,j} = 0$ for all $j$.}
            \label{F: pedigree_example}
        \end{figure}
    
        The population is diploid, and so each individual carries two copies of each chromosome. Following Mendel's law of random segregation, each gene copy in the offspring chooses independently from the two gene copies in the parent from which the gene copy is inherited. One gene comes from one parent and the other gene comes from the other parent.

        In each deme $v$, the total number of offspring $\mathcal{V}_i^v$ of the $i$th individual is
        \begin{equation*}
            \mathcal{V}_i^v :=\sum_{j \neq i} \mathcal{V}_{i,j}^v.
        \end{equation*}
        Following \cite[p. 3]{birkner2018}, we write $\mathcal{V}^v_{(1)} \geq \mathcal{V}^v_{(2)} \geq \ldots \geq \mathcal{V}_{(N(v))}^v$ for the ranked version of $\lp{\mathcal{V}_i^v}_{i=1}^{N(v)}$. As $\sum_{i} \mathcal{V}_i^v = 2N^*(v)$, we can define
        \begin{equation*}
            \widetilde{\mathcal{V}}^v_N := \lp{\frac{\mathcal{V}^v_{(1)}}{2N^*(v)}, \frac{\mathcal{V}^v_{(2)}}{2N^*(v)}, \ldots, \frac{\mathcal{V}^v_{(N(v))}}{2N^*(v)}, 0, 0, \ldots}
        \end{equation*}
        as the ranked total offspring frequencies, viewed as an element in the infinite-dimensional simplex
        \begin{equation*}
            \Delta := \{x \in [0,1]^\infty: x_1 \geq x_2 \geq \ldots \geq 0, \sum_i x_i \leq 1\}.
        \end{equation*}
        For an element $x$ in $\Delta$ we denote by $\la{x} := \sum_i x_i$ and $\left\langle x,x \right\rangle = \sum_i x_i^2$. We put $\mathbf{0}$ to mean the element $(0,0,\ldots)$ in $\Delta$. $\Delta$ is a compact metric space with the topology given by the metric
        \begin{equation*}
            d_{\Delta}(x,y) := \sum_i 2^{-i} |x_i-y_i|.
        \end{equation*}

    \subsection{The pedigree and its ancestral process}\label{SS: pedigree_coalescent}

        The population dynamics described in Section~\ref{SS: model} induce a random directed graph $\mathcal{G}_N$, called the \emph{pedigree}, which encodes the complete history of reproductive relationships in the population. Following \cite{nfw25_2}, we now give a formal definition. Let
        \begin{equation*}
            V_N = \{\{v\} \times [N(v)] : v \in V\}
        \end{equation*}
        denote the population at any fixed time-step.
        
        \begin{definition}[The pedigree]\label{D: pedigree}
            The pedigree is an undirected multigraph $\mathcal{G}_N$ with vertex set $V_N \times \Z_+$. A vertex $(v,i,k)$ corresponds to individual $i$ in deme $v$ at time-step $k$. Edges of $\mathcal{G}_N$ connect vertices in consecutive time-steps. In particular, there is an edge between $(v,i,k)$ and $(w,j,k+1)$ if the $i$th individual in deme $v$ at time-step $k$ (migrating to deme $w$ if $w \neq v$) is the child of the $j$th individual in deme $w$ at time-step $k+1$.
        \end{definition}
        
        The pedigree captures the total population history shared across the genome, coupling genealogies at loci that may be arbitrarily far apart or lie on distinct chromosomes. Classical coalescent theory \cite{kingman1982, sagitov1999general, pitman1999, schweinsberg} implicitly averages over the pedigree, describing the marginal genealogy at a single locus. Our goal is instead to describe the law of gene genealogies conditional on the pedigree, thereby capturing the distribution of genealogies across unlinked loci. For more theoretical discussion on the significance of the pedigree we refer to the section ``Previous Work on Pedigrees" in \cite{dfbw24}.
        
        Let $\{X_i^N(0)\}_{i=1}^n$ be a sample of $n$ gene copies drawn from $V_N \times \{0,1\}$ at time-step $0$. For simplicity, we assume the samples are drawn from $n$ distinct individuals
        \begin{equation*}
            \{\widehat{X}_i^N(0)\}_{i=1}^n \subset V_N ,
        \end{equation*}
        with one gene copy per individual.
        
        For each $k \in \Z_+$, let $X_i^N(k)$ denote the ancestral gene copy of $X_i^N(0)$ at time-step $k$. Under the model of Section~\ref{SS: model}, we write
        \begin{equation*}
            X_i^N(k)
            =
            \lp{\widehat{X}_i^N(k), M_i^N(k)}
            \in V_N \times \{0,1\},
            \qquad i \in [n],
        \end{equation*}
        where $\widehat{X}_i^N(k)$ denotes the individual carrying the gene copy and $M_i^N(k)$ is an independent Bernoulli$\lp{\tfrac{1}{2}}$ random variable specifying which of the two gene copies is inherited. The individual $\widehat{X}_i^N(k+1)$ is chosen uniformly at random among those $y$ such that $(y,k+1)$ is adjacent to $(\widehat{X}_i^N(k),k)$ in the pedigree.
        
        \begin{definition}[Ancestral line]
        For each $i \in [n]$, the $V_N \times \{0,1\}$-valued process
        \begin{equation*}
            X_i^N = \lp{X_i^N(k)}_{k \in \Z_+}
        \end{equation*}
        is called the ancestral line of the $i$th sampled gene copy.
        \end{definition}
        
        The ancestral lines $\lp{X_i^N}_{i \in [n]}$ form a family of correlated Markov processes on $V_N \times \{0,1\}$. Their coalescent structure is encoded by a process $\chi^{N,n}$, called the ancestral process, taking values in a space of partitions with type.
        
        We first introduce partition spaces without type, as in \cite{abfw25} and others. Let $\mathcal{E}_n$ denote the set of partitions $\xi = \{C_1,\ldots,C_b\}$ of $[n]$, where the blocks are ordered by least element. We write $|\xi| = b$ for the number of blocks. To encode diploidy, we define a refinement $\mathcal{S}_n$ of $\mathcal{E}_n$ as follows. An element $\xi \in \mathcal{S}_n$ consists of a partition
        \begin{equation*}
            \xi = \{C_1,\ldots,C_b\}
        \end{equation*}
        together with a pairing of $2x$ of its blocks, where
        \begin{equation*}
            1 \leq x \leq \left\lfloor \frac{b}{2} \right\rfloor .
        \end{equation*}
        We write such a configuration as
        \begin{equation*}
            \xi
            =
            \lp{
            (C_1,C_2),\ldots,(C_{2x-1},C_{2x}),
            C_{2x+1},\ldots,C_b
            } .
        \end{equation*}
        The paired blocks correspond to two ancestral gene copies carried by the same diploid individual, while the unpaired blocks correspond to individuals carrying a single ancestral gene copy. We write $\norm{\xi} := x$ for the number of paired individuals. We identify $\mathcal{E}_n$ as the subset of $\mathcal{S}_n$ consisting of configurations with no paired blocks. There is a natural projection \cite{birkner2018, bbe13}, the complete dispersion map $cd$, from $\Sn$ to $\En$ defined by
        \begin{equation*}
            cd(\{(C_1, C_2), \ldots, (C_{2x-1}, C_{2x}), C_{2x+1}, \ldots, C_b\}) = \{C_i\}_{i=1}^b.
        \end{equation*}
        
        A partition with type in $F$, where $F$ is an arbitrary topological space, is a pair $(\xi,f)$, where $\xi \in \mathcal{E}_n$ and
        \begin{equation*}
            f = \lp{f_r}_{r=1}^{|\xi|} \in F^{|\xi|}.
        \end{equation*}
        We denote the space of partitions with type in $F$ by $\mathcal{E}_n(F)$. Similarly, a diploid partition with type in $F$ is a pair $(\xi,f)$, where $\xi \in \mathcal{S}_n$ and
        \begin{equation*}
            f = \lp{f_r}_{r=1}^b \in F^b ,
            \qquad b = |\xi|.
        \end{equation*}
        We denote the corresponding space by $\mathcal{S}_n(F)$. We equip $\mathcal{E}_n(F)$ and $\mathcal{S}_n(F)$ with the disjoint-union topology
        \begin{equation*}
            \bigsqcup_{\xi \in \mathcal{E}_n} \{\xi\} \times F^{|\xi|}
            \quad \text{and} \quad
            \bigsqcup_{\xi \in \mathcal{S}_n} \{\xi\} \times F^{|\xi|},
        \end{equation*}
        respectively.

        As the evolution of gene genealogies backwards in time is informed by the location of the lineages, we track this extra information by viewing genealogies as $\Sn(V)$-valued processes as follows.
        
        \begin{definition}[Ancestral process]
            We define the ancestral process $\chi^{N,n} = \lp{\chi^{N,n}(k)}_{k \in \Z_+}$ as an $\mathcal{S}_n(V)$-valued stochastic process
            as follows: For each $k \in \Z_+$, $\chi^{N,n}(k) = (\xi,v)$, with $v = \lp{v_r}_{r=1}^b \in V^b$, if
            \begin{itemize}
                \item $i \sim_\xi j$ if and only if $X_i^N(k) = X_j^N(k)$,
                \item the blocks containing $i$ and $j$ are paired in $\xi$ if and only if $\widehat{X}_i^N(k) = \widehat{X}_j^N(k)$ and $M_i^N(k) \neq M_j^N(k)$,
                \item if $i \in C_r$, then $\widehat{X}_i^N(k)$ lies in deme $v_r$.
            \end{itemize}
        \end{definition}
        
        Our main results concern suitable scaling limits of $\chi^{N,n}$ under the complete dispersion map
        \begin{equation*}
            cd_V : \mathcal{S}_n(V) \to \mathcal{E}_n(V),
            \qquad
            (\xi,v) \mapsto \lp{cd(\xi),v}.
        \end{equation*}

\section{An annealed scaling limit}\label{S: annealed_result}

    To connect to existing results on the structured coalescent and multi-type $\Xi$-coalescents, including \cite{notohara90, herbots97, eldon09, notohara16, mohle24, daipra25}, we consider the \emph{annealed law}, obtained by averaging over the random pedigree, of the ancestral process in this section. To this end, for each deme $v$ we define the \emph{coalescence timescale} $c_N^v$ as the probability that two genes sampled uniformly at random from two distinct individuals in deme $v$ coalesce in a single time-step, conditional on neither migrating. Then
    \begin{equation*}
        c_N^v = \frac{1}{8}\mathbb{E}\lb{\frac{\mathcal{V}_1^v(\mathcal{V}_1^v-1)}{N(v)-1}}.
    \end{equation*}
    This is readily compared with \cite[Equation 1.4]{birkner2018} for the single deme case.

    Fix an arbitrary reference deme $v_0 \in V$. Convergence will be established for the time-rescaled ancestral process $cd_V(\overline{\chi}^{N,n})$, where
    \begin{equation}\label{E: ancestral_process_defnition}
        \overline{\chi}^{N,n}
        =
        \lp{
            \chi^{N,n}(\lfloor t \lp{c_N^{v_0}}^{-1}\rfloor)
        }_{t \in \R_+}.
    \end{equation}
    We view $cd_V(\overline{\chi}^{N,n})$ as a random element of the Skorokhod space $\mathcal{D}\lp{\R_+, \mathcal{E}_n(V)}$
    equipped with the $J_1$ topology. See \cite{ek09} for a reference for the Skorokhod space with the $J_1$ topology.
    
    We impose the following assumptions on the asymptotic behavior of the model as $N \to \infty$.

    \begin{assumption}[Convergence of the initial condition]\label{A: IC}
        Suppose that $\chi^{N,n}(0)$ converges in law to $\chi^n(0)$ in $\En(V)$, and that the $n$ samples are all from $n$ distinct individuals.
    \end{assumption}
    
    \begin{assumption}[Continuous timescale]\label{A: continuous}
        For every deme $v \in V$,
        \begin{equation*}
            \lim_{N \to \infty} c_N^v = 0.
        \end{equation*}
    \end{assumption}
    
    \begin{assumption}[Comparable timescales]\label{A: comparable}
        For every deme $v \in V$,
        \begin{equation*}
            \lim_{N \to \infty} \frac{c_N^v}{c_N^{v_0}} = c(v) \in (0,\infty).
        \end{equation*}
    \end{assumption}

    We define the \textit{joint} law of the ordered offspring frequencies and the migration frequencies $\Phi_N$ by
    \begin{equation*}
        \Phi_N := \PP\lp{(\widetilde{\mathcal{V}}_N, m_N) \in \cdot}
    \end{equation*}
    as a measure on $\Delta^V \times [0,1]^E$, where $\widetilde{\mathcal{V}}_N := \lp{\widetilde{\mathcal{V}}_N^v}_{v \in V}$. We define a metric $d$ on $\Delta^V \times [0,1]^E$ by
    \begin{equation*}
        d((x,m), (y,n)) := \sup_{v \in V}\lp{ \sum_{i \in \N} (x_i^v - y_i^v)^2}^{\frac{1}{2}} + \norm{m - n}_{\infty},
    \end{equation*}
    where $\norm{\cdot}_{\infty}$ denotes the infinity norm on $[0,1]^E$.
    Under this metric $\Delta^V \times [0,1]^E$ is a compact metric space. We let $\mathbf{0}_V = \lp{\mathbf{0}}_{v \in V}$ and $\mathbf{0}_E = \lp{0,0,\ldots, 0} \in [0,1]^E$, and define $\mathbf{0}_{V,E} = (\mathbf{0}_V, \mathbf{0}_E) \in \Delta^V \times [0,1]^E$. By $B(r)$ we will denote a ball of radius $r$ centered at $\mathbf{0}_{V,E}$ in $\Delta^V \times [0,1]^E$ under the metric $d$, and its complement $\Delta^V \times [0,1]^E \setminus B(r)$ is denoted by $B(r)^c$. We assume that $\Phi_N$ is sharply concentrated around $\mathbf{0}_{V,E}$ in the following sense.

    \begin{assumption}[Rarity of large migrations and reproduction]\label{A: rarity}
        We assume that
        \begin{equation*}
            \frac{1}{c_N^{v_0}}\,\Phi_N
            \;\to\;
            \Phi
        \end{equation*}
        vaguely on $\Delta^V \times [0,1]^E \setminus \{\mathbf{0}_{V,E}\}$, where $\Phi$ is a $\sigma$-finite measure.
    \end{assumption}

    \begin{remark}
        We provide necessary and sufficient criteria for the convergence of Assumption~\ref{A: rarity} in Section~\ref{App: weak_convergence}. It is the exact analogue of \cite[Equation 1.6]{birkner2018} for our structured model, which gives simple moment conditions for the desired convergence.
    \end{remark}

    \begin{remark}
        Unlike in \cite{birkner2018, abfw25} we do not enforce that the total mass of $\Phi$ marginally at each copy of $\Delta$ to be one. Indeed, in general, the coalescence rates at different demes will be different, requiring that marginally the total mass at each copy of $\Delta$ corresponds to the difference in scale of the coalescence rates.
    \end{remark}

    \begin{assumption}[Migration rate convergence]\label{A: migration_convergence}
        For each edge $e = (v,w) \in E$,
        \begin{equation*}
            \lim_{N \to \infty}
            \frac{1}{c_N^{v_0}}
            \PP\lp{ \widehat{X}_1(1) \in \{w\} \times \Z_+ \mid \widehat{X}_1(0)\in \{v\} \times \Z_+ }
            =
            \bar{\mu}_e .
        \end{equation*}
    \end{assumption}
    \begin{remark}
        Note that this condition is equivalent to
        \begin{equation*}
            \lim_{N \to \infty} \frac{1}{c_N^{v_0}}\frac{\mathbb{E}\lb{\lfloor N(v) m_{(v,w)} \rfloor}}{N(v)} = \bar{\mu}_{(v,w)}
        \end{equation*}
        for every $(v,w)$ in $E$. In particular, when $m_{(v,w)}$ is supported on $\{0,\frac{1}{N}, \ldots, \frac{N-1}{N}, 1\}$ it is equivalent to 
        \begin{equation*}
            \lim_{N \to \infty}\frac{1}{c_N^{v_0}}\mathbb{E}\lb{m_{(v,w)}} = \bar{\mu}_{(v,w)}.
        \end{equation*}
    \end{remark}

    We now describe the $(x,m,\xi)$-paintbox associated with $(x,m) \sim \Phi$ and $\xi \in \mathcal{E}_n(V)$.  
    Write $x = (x^v)_{v \in V}$ with $x^v = (x_1^v, x_2^v, \ldots) \in \Delta$.  
    Suppose that $\xi$ contains $b_v$ blocks in deme $v$.
    
    Define the interval space
    \begin{equation*}
        I_V := \bigsqcup_{v \in V} [0,1]
    \end{equation*}
    to be the disjoint union of $V$ copies of $[0,1]$.
    To each block of $\xi$ in deme $v$, assign independently a point in $I_V$ with intensity
    \begin{equation*}
        \delta_v \lp{ 1 - \sum_{e=(v,w)\in E} m_e } \otimes dx
        \;+\;
        \sum_{w \in V} \delta_w \, m_{(v,w)} \otimes dx.
    \end{equation*}
    This intensity is such that any block begining in deme $v$ is thrown uniformly at random on the interval for deme $w$ with probability $m_{(v,w)}$ for $w \neq v$, and probability $1-\sum_{w \neq v} m_{(v,w)}$ for being thrown uniformly at random on the interval for deme $v$.
    Blocks are identified if their points lie in the same deme and in the same interval determined by $x^v$. Denote by $q_n(x,m)(\xi,\eta)$ the probability that an element $\xi$ coalesces into the element $\eta$, conditional on $(x,m)$.
    \begin{remark}
        Note that when $V$ is a singleton and $E$ is empty that this is exactly the Kingman $x$-paintbox construction \cite{berestycki04}.
    \end{remark}
    The quantity $q_n$ may be calculated explicitly as follows: Suppose that, among the $b_v$ blocks in deme $v$ that $t_{(v,w)}$ of them travel from deme $v$ to deme $w$, and that a total of $t_{(v,v)}$ of them remain in deme $v$. After these migrations, we write the number of blocks in deme $v$ by $b_{v}'$ and temporarily view them each as an element $\xi_v$ in $\mathcal{E}_{b'_v}$. The restriction $\eta_v$ of $\eta$ to deme $v$, viewed as an element of $\mathcal{E}_{b'_v}$ may be viewed as keeping $n_v$ of the blocks of $\xi_v$ the same, while coalescing the remaining $b'_v - s_v$ blocks of $\eta_v$ into $r$ blocks made up of $y_1,y_2,\ldots, y_r$ of the blocks of $\eta_v$. The probability that this coalesce occurs in deme $v$ after the described migration (compare with Equation 1.6 in \cite{birkner2018}) is
    \begin{equation*}
        \phi_{b'_v}(y_1,\ldots, y_r) = \sum_{i_1,i_2,\ldots, i_r = 1 \text{ distinct}}^{\infty} \lp{1-\sum_{j=1}^r x_{i_j}^v}\prod_{j=1}^r (x_{i_j}^v)^{y_j}.
    \end{equation*}
    Therefore we may write $q_n(x,m)(\xi,\eta)$, by writing $m_{(v,v)} = 1- \sum_{w \neq v}m_{(v,w)}$, as 
    \begin{equation*}
        \prod_{v \in V} \phi_{b'_v}(y_1,\ldots, y_r) m_{(v,v)}^{b_v - t_{(v,v)}} \prod_{w \in V} m_{(v,w)}^{t_{(v,w)}}.
    \end{equation*}

    For $\xi \neq \eta$ and $\Phi$ a measure on $\Delta^V \times [0,1]^E \setminus \{\mathbf{0}_{V,E}\}$, define $q_n(\Phi)(\xi,\eta)$ to be rate at which, averaged over $\Phi$, an $(x,m)$-paintbox sends $\xi$ to $\eta$, i.e.
    \begin{equation*}
        q_n(\Phi)(\xi,\eta) = \int q_n(x,m)(\xi,\eta) d\Phi(x,m).
    \end{equation*}
    Further, we set
    \begin{equation*}
        q_n(\Phi)(\xi,\xi)
        =
        - \sum_{\eta \neq \xi} q_n(\Phi)(\xi,\eta).
    \end{equation*}
    We define then a transition rate matrix $Q_n(\Phi) := \lp{q_n(\Phi)(\xi,\eta)}_{\xi,\eta \in \En(V)}$.

    Finally we assume some local integrability assumption for the $\Phi$ of Assumption~\ref{A: rarity}. 
    \begin{assumption}\label{A: integrability}
        We assume the following local integrability condition near $\mathbf{0}_{V,E}$:
        \begin{equation*}
            \sup_{\xi \in \En(V)}\int_{\Delta^V \times [0,1]^E}
            \lp{1-q_n(x,m)(\xi,\xi)}
            d\Phi(x,m)
            \;<\;\infty.
        \end{equation*}
    \end{assumption}
    \begin{remark}\label{R: integrability_condition}
        For $v$ in $V$ and $e$ in $E$ let $\Phi_v$ and $\Phi_e$ denote the marginal distribution of $\Phi$ at $v$ and $e$, respectively. Assumption~\ref{A: integrability} follows so long as the marginal laws of $\Phi$ are well-behaved, i.e. so long as 
        \begin{align*}
            \sup_{v \in V} \int_{\Delta} \left\langle x^v, x^v \right\rangle d\Phi_v(x^v) &< \infty \text{ and }\\
            \sup_{e \in E} \int_{[0,1]} m_e d\Phi_e(m_e) &< \infty.
        \end{align*}
    \end{remark}

    Finally, we have a technical assumption on the migration sizes that we assume can be relaxed, but which we have not been able to remove from the proof.
    \begin{assumption}[No total migrations] \label{A: no_total_migrations}
        We assume, for any deme $v$ in $V$, that $\Phi\lp{\sum_{e=(v,w)\in E} m_e \in \cdot}$ as a law on $[0,1]$ has no atom at $1$.
    \end{assumption}

    Combining Assumptions~\ref{A: migration_convergence} and \ref{A: rarity}, the neutral migration rate along edge $e$ is
    \begin{equation*}
        \mu_e
        =
        \bar{\mu}_e
        -
        \int_{[0,1]} m_e \, d\Phi_e(x,m) .
    \end{equation*}
    
    Similarly, combining Assumptions~\ref{A: comparable} and \ref{A: rarity}, the total neutral Kingman coalescence rate in deme $v$ is
    \begin{equation*}
        \kappa_v
        =
        c(v)
        -
        \frac{1}{2}\int_{\Delta} \langle x^v, x^v \rangle \, d\Phi_v(x^v).
    \end{equation*}
    
    For $\kappa = \lp{\kappa_v}_{v \in V} \in [0,\infty)^V$, we define the $V$-structured Kingman $n$-generator
    $K_n(\kappa) = (k_n(\kappa)(\xi,\eta))_{\xi,\eta \in \mathcal{E}_n(V)}$ by
    \begin{equation*}
        k_n(\kappa)(\xi,\eta)
        =
        \begin{cases}
            -\displaystyle\sum_{v \in V} \binom{b_v}{2} \kappa_v,
            & \eta = \xi,\\
            \kappa_v&,
             \eta \text{ obtained from } \xi \text{ by coalescing two blocks in deme } v, \\
            0, & \text{otherwise.}
        \end{cases}
    \end{equation*}
    
    Similarly, for any $\mu = \lp{\mu_e}_{e \in E} \in [0,\infty)^E$, we define the migration generator $M_n(\mu) = (m_n(\mu)(\xi,\eta))_{\xi,\eta \in \mathcal{E}_n(V)}$ by
    \begin{equation*}
        m_n(\mu)(\xi,\eta)
        =
        \begin{cases}
            \mu_{(v,w)},
            & \text{ if }\eta \underset{v\to w}{\prec}\xi,\\
            -\displaystyle\sum_{(v,w)\in E} b_v\,\mu_{(v,w)},
            & \text{ if } \eta = \xi, \\
            0, & \text{otherwise.}
        \end{cases}
    \end{equation*}
    Here $\eta \underset{v\to w}{\prec}\xi$ indicates that $\eta$ may be obtained by moving one block of $\xi$ in the deme $v$ to the deme $w$.

    Finally, we introduce the halving map $h: \Delta \to \Delta$ (compare with $\varphi$ in the statement of Theorem 1.1 in \cite{birkner2018})  that splits each interval in twain as follows:
    \begin{equation*}
        h(x) := \lp{\frac{1}{2}x_1, \frac{1}{2}x_1, \frac{1}{2}x_2, \frac{1}{2}x_2, \ldots}.
    \end{equation*}
    We define an extension $h_V: \Delta^V \times [0,1]^E\to \Delta^V \times [0,1]^E$ of $h$ by
    \begin{equation*}
        h_V(x,m) = \lp{(h(x^v))_{v \in V}, m}.
    \end{equation*}

    \begin{definition}\label{D: structured_coalescent}
        A $(\Phi, \kappa, \mu)$-$n$-coalescent $\chi^n$ is a $\mathcal{E}_n(V)$-valued process with transition rate matrix $K_n(\kappa) + M_n(\mu) + Q_n(\Phi)$.
    \end{definition}

    \begin{theorem}\label{T: annealed}
        Suppose that Assumptions~\ref{A: IC}, \ref{A: continuous}, \ref{A: comparable}, \ref{A: rarity}, \ref{A: migration_convergence}, \ref{A: integrability}, and \ref{A: no_total_migrations} hold. Then $cd_V\lp{\overline{\chi}^{N,n}}$ converges in distribution in $\mathcal{D}\lp{\R_+, \En(V)}$ to a $((h_V)_*\Phi, \kappa, \mu)$-$n$-coalescent.
    \end{theorem}
    \begin{remark}
        Note that when $V$ is a singleton and $E$ is empty, we have the model of \cite{birkner2018}. In this case we have that $\Phi$ is a $\sigma$-finite measure on $\Delta$ for which $\int_{\Delta} \left\langle x,x\right\rangle d\Phi(x) < \infty$. In particular, therefore, there exists a finite measure $\Xi$ on $\Delta$ so that $\Phi(dx) = \frac{1}{\left\langle x, x\right\rangle}\Xi(dx)$, which is precisely the limiting object of \cite[Equation 1.5]{birkner2018}. In particular, \citep[Theorem 1.1]{birkner2018} is implied by Theorem~\ref{T: annealed}.
    \end{remark}
    \begin{remark}
        When $\Phi = 0$, we recover the classical structured coalescent of \cite[Equation 3]{herbots97} for finite graphs. That is, there exists some rates $\mu_e$ along which which lineages migrate independently across each edge $e$ and within each deme $v$ there exist some rate $\kappa_v$ at which lineages coalesce.
    \end{remark}

    Theorem~\ref{T: annealed} will follow as a corollary of the main result of this paper, described in Section~\ref{S: main_result}. It is proven in Section~\ref{App: annealed_convergence}.

\section{Main results}\label{S: main_result}

    To describe the main result, we define a $\Psi$-driven $(\kappa,\mu)$-$n$-coalescent.

    \begin{definition}[$\Psi$-driven $(\kappa,\mu)$-$n$-coalescent]\label{D: Psi_driven}
        Let $\Psi$ be a Poisson point process on $[0,\infty) \times (\Delta^V \times [0,1]^E)$ with intensity measure $dt \otimes d\Phi$ for some measure $\Phi$ on $\Delta^V \times [0,1]^E$ satisfying Assumption~\ref{A: integrability}. A $\Psi$-driven $(\kappa,\mu)$-$n$-coalescent $\chi^n$ is a time-inhomogeneous Markov process taking values in $\En(V)$. Between atom times of $\Psi$, $\chi^n$ evolves independently as a $n$-$(0, \kappa, \mu)$-coalescent. Conditional on $\Psi$ and for each atom $(t,x,m) \in \Psi$, $\chi^n$ performs at time $t$ an $(x,m,\chi^n(t-))$-merger. That is, for all $\xi,\eta \in \En(V)$ we have that
        \begin{equation*}
            \PP\lp{\chi^n(t) = \eta \,\mid\, \Psi, \chi^n(t-) = \xi} = q_n(x,m)(\xi,\eta).
        \end{equation*}
    \end{definition}
    While this is well-defined when $\Phi$ is a finite measure, we want it to be defined for a larger class of measures. To this end, a rigorous pathwise construction showing that this is well-defined for any $\Phi$ satisfying the regularity condition of Assumption~\ref{A: integrability} is given in Section~\ref{SS: driven_coupling}.

    Finally, we describe the mode of convergence of the main theorem. We are interested in the convergence of
    the law of $\overline{\chi}^{N,n}$ conditional on $\mathcal{A}_N$,
    \begin{equation*}
        \PP\lp{cd_V\lp{\overline{\chi}^{N,n}} \in \cdot \,\mid\, \mathcal{A}_N},
    \end{equation*}
    where
    \begin{equation*}
        \mathcal{A}_N := \sigma(\mathcal{G}_N, \widehat{X}_i^N(0): 1 \leq i \leq n)
    \end{equation*}
    is the $\sigma$-algebra generated by the pedigree and the labels of the individuals from whom our sample is taken. The conditional law is a random measure taking values in the space $\mathcal{M}_1(\mathcal{D}\lp{\R_+, \En(V)})$. As $\mathcal{D}\lp{\R_+, \En(V)}$ is a Polish space under the $J_1$ topology, the convergence in distribution of random measures thereon has a straightforward characterization, which we provide here.
    \begin{definition}[Weak convergence in distribution for random measures]
        Let $S$ denote a Polish space. Let $\mu_N, \mu$ denote a sequence of $\mathcal{M}_1(S)$-valued random measures, where $\mathcal{M}_1(S)$ is given the weak topology. $\mu_N$ converges in distribution to $\mu$ if
        \begin{equation*}
            \int_S f(x) d\mu_N(x) \toL \int_S f(x) d\mu(x)
        \end{equation*}
        for any continuous and bounded $f$ of bounded support. This follows \cite[Theorem 4.19]{kallenberg17}.
    \end{definition}
    For a more detailed discussion about weak convergence in distribution on the Skorokhod space with the $J_1$ topology we refer the reader to \cite[Section 6]{nfw25_2}.

    We now state the main result of this paper. 
     \begin{theorem}\label{T: quenched}
        Suppose Assumptions~\ref{A: IC}, \ref{A: continuous}, \ref{A: comparable},\ref{A: rarity}, \ref{A: migration_convergence},\ref{A: integrability}, and \ref{A: no_total_migrations} hold. Let $\chi^n$ be a $\Psi$-driven $n$-$(\kappa,\mu)$-coalescent where $\Psi$ has intensity measure $dt \otimes (h_V)_* \Phi$. Then the law of the quenched coalescent on the pedigree converges weakly in distribution to the law of $\chi^n$ conditional on $\Psi$; viz. 
        \begin{equation*}
            \PP\lp{cd_V\lp{\overline{\chi}^{N,n}} \in \cdot \,\mid\, \mathcal{A}_N} \toL \PP\lp{\chi^n \in \cdot \,\mid\, \Psi}.
        \end{equation*}
    \end{theorem}
    
    \begin{remark}
        We can see that, so long as $\Phi$ is not the zero measure, $\Psi$ will have atoms. In particular, the limiting object of Theorem~\ref{T: quenched} is a bona fide random measure. In particular, the annealed descriptions of \cite{herbots97, notohara16} do not capture the cross-locus dependencies that are preserved here in our theorem. In fact, the annealed and quenched limits coincide, and unlinked loci really are independent in the infinite population limit, if and only if $\Phi$ is the zero measure.
    \end{remark}
    \begin{remark}
        This theorem implies the main theorem of \cite{abfw25}, where $\Phi = \frac{1}{\left\langle x,x\right\rangle} \Xi$ for $\Xi$ a finite measure on $\Delta \setminus \{\mathbf{0}\}$.
    \end{remark}
    \begin{remark}\label{R: infinite_graph}
        While the model up to this point has only been described for a finite directed graph, it is straightforward to see how the limit would extend to the infinite graph limit. Indeed, so long as an arbitrary lineage tracing itself backwards in time satisfies a compact containment condition, the limiting quenched law for the gene genealogies will follow from this theorem so long as the infinite graph model, restricted to a large enough compact set, satisfies the assumptions of the main theorem.
    \end{remark}

    The proof of Theorem~\ref{T: quenched} is contained in Section~\ref{S: proof_main}. The central idea of the proof is to use the method of moments to compare the quenched finite-dimensional distributions of $\overline{\chi}^{N,n}$ and $\chi^n$. This is possible because the $l$th moment of the quenched finite-dimensional distribution of $\overline{\chi}^{N,n}$ ($\chi^n$) is an annealed finite-dimensional distribution of $l$ conditionally independent realizations of $\overline{\chi}^{N,n}$ ($\chi^n$). We describe a coupling, in a manner following \cite{abfw25}, which allows us to characterize this annealed joint law in Section~\ref{S: coupling}. 

\section{The proof of Theorem~\ref{T: quenched}}\label{S: proof_main}

    The proof essentially proceeds by the method of moments. This is possible, due to a corollary of \cite{nfw25_2}, which guarantees that quenched convergence in finite-dimensional distribution of c\`{a}dl\`{a}g paths is enough to guarantee weak convergence in distribution of quenched laws when the space the paths take values in is Polish and locally compact. We make this reasoning precise in the following lemma.

    \begin{lemma}\label{L: method_of_moments}
        For each $l \in \N$, let $\{\overline{\chi}_i\}_{i = 1}^l$ and $\{\chi_i\}_{i=1}^m$ denote realizations of $\overline{\chi}^{N,n}$ and $\chi^n$, respectively, that are conditionally independent with respect to $\mathcal{A}_N$ and $\Psi$, respectively. If, for each $l$ in $\N$, $\lp{\overline{\chi}_i^{N,n}}_{i=1}^l$ converges in finite-dimensional distribution to $\lp{\chi_i^n}_{i=1}^l$, then $\PP\lp{cd_V\lp{\overline{\chi}^{N,n}} \in \cdot \,\mid\, \mathcal{A}_N}$ converges weakly in distribution to $\PP\lp{\chi^n \in \cdot \,\mid\, \Psi}$.
    \end{lemma}
    \begin{proof}
        As $\En(V)$ is a compact metric space, by Corollary 6.9 of \cite{nfw25_2} it suffices to show that, for any vector $\vec{t} = \lp{t_i}_{i=1}^r \in \R_+^r$ of ascending times and any vector of states $\vec{\xi} = \lp{\xi_i}_{i=1}^r$, that
        \begin{equation}\label{E: cylinder_convergence}
            \PP\lp{cd_V\lp{\overline{\chi}^{N,n}} \in C(\vec{t}, \vec{\xi}) \,\mid\, \mathcal{A}_N} \toL \PP\lp{\chi^n \in C(\vec{t}, \vec{\xi}) \,\mid\, \Psi},
        \end{equation}
        where $C(\vec{t}, \vec{\xi})$ is the cylinder set
        \begin{equation}\label{E: cylinder_set}
            C(\vec{t}, \vec{\xi}):= \{x \in \mathcal{D}\lp{\R_+, \En(V)}: x(t_i) = \xi_i \text{ for all } 1 \leq i \leq r\}.
        \end{equation}

        As $\PP\lp{cd_V\lp{\overline{\chi}^{N,n}} \in C(\vec{t}, \vec{\xi}) \,\mid\, \mathcal{A}_N}$ is a sequence uniformly bounded random variables, the convergence in \eqref{E: cylinder_convergence} follows if
        \begin{equation*}
            \mathbb{E}\lb{\PP\lp{cd_V(\overline{\chi}^{N,n}) \in C(\vec{t}, \vec{\xi}) \,\mid\, \mathcal{A}_N}^l} = \PP\lp{\lp{cd_V\lp{\overline{\chi}_i^{N,n}}}_{i=1}^l \in C(\vec{t}, \vec{\xi})^l}
        \end{equation*}
        converges as $N$ goes to infinity to
        \begin{equation*}
            \mathbb{E}\lb{\PP\lp{\chi^n \in C(\vec{t}, \vec{\xi}) \,\mid\, \Psi }^l} = \PP\lp{\lp{\chi_i^n}_{i=1}^l \in C(\vec{t}, \vec{\xi})^l}
        \end{equation*}
        for any finite $l$.
        This follows immediately from convergence in finite-dimensional distribution of the $\lp{\overline{\chi}_i^{N,n}}_{i=1}^l$ to $\lp{\chi_i^n}_{i=1}^l$.
    \end{proof}

    By Lemma~\ref{L: method_of_moments}, it suffices therefore to show that $\lp{cd_V\lp{\overline{\chi}_{i}^{N,n}}}_{i=1}^l$ converges in finite-dimensional distribution to $\lp{\chi_i^n}_{i = 1}^l$ for all $l \in \N$. To this end, we describe a precise coupling of the $\overline{\chi}_i^{N,n}$ in Section~\ref{S: coupling}. The coupling is such that the conditionally independent realizations are given by random walks on the \textit{same} pedigree, but where the choice of edges traversed by the blocks in one realization are independent to those of another. In biological terms, the $l$ realizations are the gene genealogies of a sample taken from the same individuals but at $l$ unlinked loci. This coupling makes explicit how it is that the lineages backwards in time may overlap on the pedigree, belong to the same individual or not. As the time-scale on which there may be an overlap of lineages is $O_{\PP}(1)$ while the time-scale that any of two lineages of among the $l$ realizations may overlap is $O_{\PP}((c_N^{v_0})^{-1})$, there is a separation of timescales so that we only need to consider jumps of $\lp{\overline{\chi}_i^{N,n}}_{i=1}^l$ in $\En(V)^l$, where none of the lineages in among all of the $l$ realizations belong to the same individual, jump to another state in $\En(V)^l$ after projecting via the complete dispersion map $cd_V$. This is made precise in Section~\ref{SS: negligible_outside}.

    To demonstrate joint convergence of the realizations, we utilize a separation of scales argument. To this end, we fix $\varepsilon > 0$. When we restrict ourselves to small enough offspring distributions and migrations, i.e. when $(\widetilde{\mathcal{V}}(k), m(k)) \in B(\varepsilon)$, the jumps of each of the $\overline{\chi}_i^{N,n}$ are approximately independent. This is made precise by the following lemma. 
    \begin{lemma}\label{L: small_jumps}
        Suppose, as $N \to \infty$, that Assumptions~\ref{A: comparable}, \ref{A: migration_convergence}, and \ref{A: no_total_migrations} hold. Let $R$ be the infinitesimal generator of $l$ joint, independent $(0,\kappa,\mu)$-$n$-coalescents, viewed as an $\En(V)^l \times \En(V)^l$ rate matrix. Let $A_N := e^{c_N^{v_0} R}$. Then there is a continuous function $f:\R_+ \to \R_+$ such that $f(0) = 0$ and
        \begin{equation*}
            \norm{P_{N,n,\varepsilon,l} - A_N}_{\infty} \leq f(\varepsilon) c_N^{v_0} \text{\quad for any $\epsilon > 0$},
        \end{equation*}
        where $P_{N,n,\varepsilon,l}$ is the one-step transition matrix
        \begin{equation*}
            \lp{\PP\lp{\lp{cd_V\lp{\chi_i^{N,n}(1)}}_{i=1}^l = \vec{\eta} \,\mid\, \lp{\chi_i^{N,n}(0)}_{i=1}^l = \vec{\xi}, (\widetilde{\mathcal{V}}(0), m(0)) \in B(\varepsilon)}}_{\vec{\xi}, \vec{\eta} \in \En(V)^l}.
        \end{equation*}
    \end{lemma}
    Lemma~\ref{L: small_jumps} is proven in Section~\ref{SS: small_jumps}.

    At the same time, when we restrict ourselves to the rare jumps where $(\widetilde{\mathcal{V}}(k), m(k)) \in B(\varepsilon)^c$ we have that the transitions of the conditionally independent realizations are really asymptotically identical. 
    \begin{lemma}\label{L: big_jumps}
        Suppose, as $N \to \infty$, that Assumptions~\ref{A: continuous}, \ref{A: comparable}, and \ref{A: rarity} hold. Define the transition probability matrix $H_{N,l}: \Delta^V \times [0,1]^E \to [0,1]^{\En(V)^l \times \En(V)^l}$ by
        \begin{equation*}
            H_{N,l}(x,m) := \lp{\PP\lp{\lp{cd_V\lp{\overline{\chi}_{i}^{N,n}(k+1)}}_{i=1}^l = \vec{\eta} \,\mid\, \lp{\overline{\chi}_{i}^{N,n}(k)}_{i=1}^l = \vec{\xi}, (\widetilde{\mathcal{V}}(k), m(k)) = (x,m)}}_{\vec{\xi}, \vec{\eta} \in \En(V)^l}.
        \end{equation*}
        Define the transition matrix $H_l: \Delta^V \times [0,1]^E \to [0,1]^{\En(V)^l \times \En(V)^l}$ by
        \begin{equation*}
            H_l(x,m) := \lp{\prod_{i=1}^l q_n(x,m)(\xi_i, \eta_i)}_{\vec{\xi}, \vec{\eta} \in \En(V)^l}.
        \end{equation*}
        Then for every $\varepsilon>0$,
        \begin{equation*}
            \frac{1}{c_N^{v_0}}
            \int_{B(\varepsilon)^c}
                \la{H_{N,l}(x,m)-H_l\circ h_V(x,m)}_\infty \, d\Phi_N(x,m)
            \to 0
        \end{equation*}
        as $N$ goes to infinity.
    \end{lemma}
    Lemma~\ref{L: big_jumps} is proven in Section~\ref{SS: big_jumps}.
    \begin{remark}
        Here we use a truncation argument to bound the total proportion of migrants from any deme away from $1$ using Assumption~\ref{A: no_total_migrations}. We believe this truncation argument may be improced to remove this assumption, though this is not undertaken in the present work.
    \end{remark}

    Combining the two scales of Lemma~\ref{L: small_jumps} and Lemma~\ref{L: big_jumps} allows us to conclude that the transition kernel of the joint process $\lp{\overline{\chi}_i^{N,n}}_{i=1}^m$ is arbitrarily close to that of a Markov process with infinitesimal generator
    \begin{equation}\label{E: L_defn}
        \mathcal{L} := R + \int (H_l\circ h_V(x,m)-I)d\Phi(x,m),    
    \end{equation}
    which is the meaning of the following lemma.
    \begin{lemma}\label{L: combining_scales}
        Suppose, as $N \to \infty$, that Assumptions~\ref{A: continuous}, \ref{A: comparable}, \ref{A: rarity}, \ref{A: migration_convergence}, and \ref{A: no_total_migrations} hold. Let 
        \begin{equation*}
            P_{N,l}(t) := \lp{\PP\lp{\lp{cd_V\lp{\overline{\chi}_i^{N,n}}}_{i=1}^l = \vec{\eta} \,\mid\, \lp{cd_V\lp{\overline{\chi}_i^{N,n}}(0)}_{i=1}^l = \vec{\xi}}}_{\vec{\xi}, \vec{\eta} \in \En(V)^l}
        \end{equation*}
        denote the transition kernel of $\lp{cd_V\lp{\overline{\chi}_i^{N,n}}}_{i=1}^l$ at time $t$. Then, for any fixed $T > 0$, there is a continuous function $g_T:\R_+\to\R_+$ with $g(0) = 0$ such that
        \begin{equation}\label{E: kernel_approximation}
            \sup_{0 \leq t \leq T} \norm{P_{N,l}(t) - \exp\lp{t \mathcal{L}}}_{\infty} \leq g_T(\varepsilon) \text{ for all } \epsilon > 0
        \end{equation}
    \end{lemma}
    Lemma~\ref{L: combining_scales} is proven in Section~\ref{SS: combining_scales}.

    Finally, we characterize the annealed joint law of $l$ realizations $\lp{\chi_i^n}_{i=1}^l$ of a $\Psi$-driven $n$-$(\kappa,\mu)$-coalescent $\chi^n$ that are conditionally independent with respect to $\Psi$. We show that $\lp{\chi_i^n}_{i=1}^l$ has infinitesimal generator $\mathcal{L}$.
    \begin{lemma}\label{L: driven_coupling}
        Suppose $\Phi$ satisfies the regularity condition of Assumption~\ref{A: integrability}. Then $\lp{\chi_i^n}_{i=1}^l$, as a $\En(V)^l$-valued Markov process, has infinitesimal generator $\mathcal{L}$.
    \end{lemma}
    Lemma~\ref{L: driven_coupling} is proven in Section~\ref{SS: driven_coupling}.

    With the above lemmas, we are now able to present the proof of Theorem~\ref{T: quenched}.
    \begin{proof}
        By Lemma~\ref{L: method_of_moments} it suffices to show that, for any $l$ in $\N$ that
        \begin{equation*}
            \lim_{N \to \infty}\PP\lp{\lp{cd_V\lp{\overline{\chi}_i^{N,n}}}_{i=1}^l \in C(\vec{t}, \vec{\xi})^l}
            =
            \PP\lp{\lp{\chi_i^n}_{i=1}^l \in C(\vec{t}, \vec{\xi})^l},
        \end{equation*}
        where $C(\vec{t}, \vec{\xi})$ is the cylinder set of Equation~\eqref{E: cylinder_set} and $\chi_i^N$ and $\overline{\chi}_i^{N,n}$ are as in the statement of the lemma.

        By Lemma~\ref{L: driven_coupling}, the annealed joint process $\lp{\chi_i^n}_{i=1}^l$ has infinitesimal generator $\mathcal{L}$. As $\lp{\chi_i^n}_{i=1}^l$ has infinitesimal generator $\mathcal{L}$ we know it has transition kernel $P_l(t) := \exp(t\mathcal{L})$. Under Assumptions~\ref{A: continuous}, \ref{A: comparable}, \ref{A: rarity}, and \ref{A: migration_convergence}, we have that Lemma~\ref{L: combining_scales} holds. As the transition kernel $P_{N,l}(t)$ of $\lp{cd_V\lp{\overline{\chi}_i^{N,n}}}_{i=1}^l$ satisfies Equation~\eqref{E: kernel_approximation} it follows that $P_{N,l}$ converges uniformly on any compact interval to $P_l$ by taking $N$ to infinity and then $\varepsilon$ to zero. This, and convergence of the initial condition given by Assumption~\ref{A: IC} is enough to demonstrate convergence in finite-dimensional distribution, as was needed.
    \end{proof}

\section{Examples}\label{S: examples_and_apps}

    In this section, we provide various examples of models of interest. These include a two-deme Wright-Fisher model with neutral migration in Section~\ref{SS: WF_two_deme_convergence}, a two-deme modified Wright-Fisher model with large offspring numbers in Section~\ref{SS: diamantidis_model_convergence}, a two-deme model with random individual fitness and beta-distributed migrations in Section~\ref{SS: beta_migration}, and what we call a discrete approximation of a $\Xi$ Fleming-Viot genealogy in Section~\ref{SS: xi_model}. The single deme case is described in some detail in \cite[Section 7]{abfw25}.

    \subsection{A two-deme Wright-Fisher model with neutral migration}\label{SS: WF_two_deme_convergence}

        We describe here the discrete-time model of \cite{WiltonEtAl2017}, which is a two deme Wright-Fisher model with conservative neutral migration. Concretely, we take
        \begin{equation*}
            V=\{1,2\}, \qquad E=\{(1,2),(2,1)\}, \qquad s(1)=s(2)=1,
        \end{equation*}
        and we sample the conservative migration sizes
        \begin{equation*}
            m_{(1,2)}=m_{(2,1)}\sim \frac{1}{N}\text{\emph{Binomial}}\lp{N; \frac{1}{N}\frac{\sigma}{2}} 
        \end{equation*}
        for some fixed $\sigma>0$. Since migration is conservative, at each time-step we have exactly $N$ diploid individuals in each deme reproduce. Conditional on $m$, reproduction within each deme is Wright-Fisher with uniform weights and no selfing, and reproduction across demes is independent. Equivalently,
        \begin{equation*}
            (\mathcal{V}^v)_{v\in V}
            \stackrel{d}{=}
            \text{\emph{Multinomial}}\lp{N;\frac{2}{N(N-1)},\ldots,\frac{2}{N(N-1)}}^{\otimes 2}.
        \end{equation*}
        Convergence of this model is established by the following proposition.

        \begin{proposition}\label{P: two_deme_WF}
            Let $\overline{\chi}^{N,n}$ and $\mathcal{G}_N$ denote the ancestral process and pedigree associated to a two-deme Wright-Fisher model with neutral migration. Let $\kappa = \lp{1,1}$ and $\mu_{(1,2)} = \mu_{(2,1)} = \mu$. Suppose that, as $N$ goes to infinity, that Assumption~\ref{A: IC} holds. Then $\PP\lp{\overline{\chi}^{N,n} \in \cdot \,\mid\, \mathcal{A}_N}$ converges weakly in distribution to the law of a $(0, \kappa, \mu)$-$n$-coalescent.
        \end{proposition}
        \begin{proof}
            Under this model the pair-coalescence scale is the usual Wright-Fisher scale: for every $v\in V$ and every $N$,
            \begin{equation*}
                c_N^v=\frac{1}{2N}.
            \end{equation*}
            In particular, Assumptions~\ref{A: continuous} and \ref{A: comparable} hold. Moreover, since $Nm_e$ is binomially distributed with $N$ trials and success probability $\frac{\sigma}{2}$ and $c_N^v=1/(2N)$, a direct calculation yields $\bar{\mu}_e=\sigma$ for every $e\in E$, so Assumption~\ref{A: migration_convergence} holds as well.
    
            Finally, because reproduction within each deme is Wright-Fisher and is independent across demes, \cite[Proposition 2.1]{birkner2018} and Corollary~\ref{C: rarity_product_form} imply that $\frac{1}{c_N^1}\Phi_N$ converges vaguely on $\Delta^V\times[0,1]^E$ to
            \begin{equation*}
                \Phi \;=\; 0^{\otimes 2}\otimes \delta_0^{\otimes 2}.
            \end{equation*}
            Therefore the main assumptions necessary for Theorem~\ref{T: quenched} hold. Since $\Phi$ charges no sets bounded away from $\mathbf{0}_{V,E}$, the limiting driving point process $\Psi$ has no atoms. Consequently the quenched and annealed laws coincide in this example.
        \end{proof}
        In Section~\ref{S: discussion} we discuss implications of Proposition~\ref{P: two_deme_WF} to the work of \cite{WiltonEtAl2017}.
        
    \subsection{A two-deme modified Wright-Fisher model with large offspring numbers}\label{SS: diamantidis_model_convergence}
        
        Here we discuss a two-deme analogue of the model of \cite{dfbw24}. We take
        \begin{equation*}
            V=\{1,2\}, \qquad E=\{(1,2),(2,1)\}, \qquad s(1)=s(2)=1,
        \end{equation*}
        so that each deme has size $N(1) = N(2) = N$. We assume that reproduction dynamics in the two
        demes are independent conditional on $m$.
        
        Fix parameters $\sigma>0$, $\phi>0$, $\psi\in(0,1)$, and $\gamma>0$. Write $N_\gamma := N \wedge N^\gamma$ and sample the migrations by
        \begin{equation*}
            m_{(1,2)}=m_{(2,1)}\sim \frac{1}{N} \text{\emph{Binomial}}\lp{N; \frac{\sigma}{2}\frac{1}{N_\gamma}}
        \end{equation*}
        This migration is conservative, so $N^*(1)=N^*(2)=N$ for every time-step.
        
        We now describe the offspring matrices in a single deme $v\in\{1,2\}$.
        Let $\varepsilon_N:=2\phi/N^\gamma$. Conditional on $m$, we first sample a Bernoulli variable
        $B^v$ with
        \begin{equation*}
            \PP\{B^v=1\}=\varepsilon_N.
        \end{equation*}
        If $B^v=0$ (a ``small'' generation), the deme reproduces by Wright-Fisher dynamics, i.e. we have
        \begin{equation*}
            (\mathcal{V}^v)_{v\in V}
            \stackrel{d}{=}
            \text{\emph{Multinomial}}\lp{N;\frac{2}{N(N-1)},\ldots,\frac{2}{N(N-1)}}^{\otimes 2}
        \end{equation*}
        
        If $B^v=1$ (a ``large'' generation), we sample an unordered pair
        $\{I^v,J^v\}\subset[N]$ uniformly at random and set
        \begin{equation*}
            \mathcal{V}^v_{I^v,J^v}=\lfloor \psi N\rfloor.
        \end{equation*}
        The remaining $N - \lfloor \psi N \rfloor$ individuals in the deme choose parental couples uniformly at random from the $\binom{N}{2}-1$ remaining possible couples, i.e. reproduce via Wright-Fisher.

        We define the necessary elements of the scaling limit of this model as follows:  Set
        \begin{equation*}
            x_\psi:=\lp{\frac{\psi}{4},\frac{\psi}{4},\frac{\psi}{4},\frac{\psi}{4},0,0,\ldots}\in\Delta,
        \end{equation*}
        and define $x^{(1)},x^{(2)}\in\Delta^V$ by
        \begin{equation*}
            x^{(1),1}=x_\psi,\quad x^{(1),2}=\mathbf{0},
            \qquad
            x^{(2),1}=\mathbf{0},\quad x^{(2),2}=x_\psi.
        \end{equation*}
        Then
        \begin{equation*}
            \Phi_\gamma
            :=
            \lambda_\gamma\lp{\delta_{\lp{x^{(1)},0_E}}+\delta_{\lp{x^{(2)},0_E}}},
            \qquad
            \lambda_\gamma
            :=
            \begin{cases}
                0, & \gamma>1,\\
                \dfrac{4\phi}{1+\phi\psi^2}, & \gamma=1,\\
                \dfrac{4}{\psi^2}, & 0<\gamma<1,
            \end{cases}
        \end{equation*}
        and the Kingman and migration parameters are
        \begin{equation*}
            \kappa_\gamma
            :=
            \begin{cases}
                \lp{1,1}, & \gamma>1,\\
                \lp{\dfrac{1}{1+\phi\psi^2},\dfrac{1}{1+\phi\psi^2}}, & \gamma=1,\\
                \lp{0,0}, & 0<\gamma<1,
            \end{cases}
            \qquad
            \mu_{\gamma,(1,2)}=\mu_{\gamma,(2,1)}
            :=
            \begin{cases}
                \sigma, & \gamma>1,\\
                \dfrac{\sigma}{1+\phi\psi^2}, & \gamma=1,\\
                0, & 0<\gamma<1.
            \end{cases}
        \end{equation*}
        Write $\mu_\gamma =\lp{\mu_{\gamma,e}}_{e \in E}$.

        Convergence of this model is established by the following proposition.

        \begin{proposition}\label{P: diamantidis_model}
            Let $\overline{\chi}^{N,n}$ and $\mathcal{G}_N$ denote the ancestral process and pedigree of a two-deme Wright-Fisher model with large offspring numbers. Let $\chi^n$ denote a $\Psi_\gamma$ driven $(\kappa_\gamma,\mu_\gamma)$-$n$ coalescent, where $\Psi_\gamma$ has intensity measure $dt \otimes (h_V)_* \Phi_\gamma$. Suppose that, as $N$ goes to infinity, that Assumption~\ref{A: IC} holds. Then 
            \begin{equation*}
                \PP\lp{cd_V\lp{\overline{\chi}^{N,n}} \in \cdot \,\mid\, \mathcal{A}_N} \toL \PP\lp{\chi^n \in \cdot \,\mid\, \Psi_\gamma}.
            \end{equation*}      
        \end{proposition}
    
        \begin{proof}
            Note by symmetry that
            $c_N^1=c_N^2=:c_N$.
            By conditioning on $B^v\in\{0,1\}$ we can compute $c_N$ from
            \begin{equation*}
                c_N
                =
                \frac{1}{8}\,\mathbb{E}\lb{\frac{\mathcal{V}_1^v\lp{\mathcal{V}_1^v-1}}{N-1}}
                =
                \frac{1}{8}\lp{ (1-\varepsilon_N)\mathbb{E}\lb{\frac{\mathcal{V}_1^v\lp{\mathcal{V}_1^v-1}}{N-1}\,\,\mid\, \,B^v=0}
                +\varepsilon_N\mathbb{E}\lb{\frac{\mathcal{V}_1^v\lp{\mathcal{V}_1^v-1}}{N-1}\,\,\mid\, \,B^v=1}}.
            \end{equation*}
        
            In a small generation $B^v=0$, each of the $N$ offspring chooses an unordered parental couple uniformly,
            so $\PP\{1\in\{I,J\}\}=2/N$ and hence $\mathcal{V}_1^v\mid\{B^v=0\}\sim\mathrm{Bin}\lp{N,2/N}$.
            Therefore
            \begin{equation*}
                \mathbb{E}\lb{\mathcal{V}_1^v\lp{\mathcal{V}_1^v-1}\,\,\mid\, \,B^v=0}
                =N\lp{N-1}\lp{\frac{2}{N}}^2
                =\frac{4\lp{N-1}}{N},
            \end{equation*}
            and consequently
            \begin{equation}\label{E: diamantidis_cN_small}
                \frac{1}{8}\mathbb{E}\lb{\frac{\mathcal{V}_1^v\lp{\mathcal{V}_1^v-1}}{N-1}\,\,\mid\, \,B^v=0}
                =
                \frac{1}{2N}.
            \end{equation}
        
            In a large generation $B^v=1$, we have $\PP\{1\in\{I^v,J^v\}\}=2/N$ and, on this event,
            individual $1$ is a parent of $\lfloor\psi N\rfloor$ offspring, so
            $\mathcal{V}_1^v=\lfloor\psi N\rfloor+O_\PP(1)$ and hence
            \begin{equation*}
                \mathbb{E}\lb{\frac{\mathcal{V}_1^v\lp{\mathcal{V}_1^v-1}}{N-1}\,\,\mid\, \,B^v=1,\;1\in\{I^v,J^v\}}
                =\psi^2 N+O(1).
            \end{equation*}
            On the complementary event $\{1\notin\{I^v,J^v\}\}$, we still have $\mathcal{V}_1^v=O_\PP(1)$, and so the same
            quantity is $O\lp{1/N}$.
            Putting this together yields
            \begin{equation}\label{E: diamantidis_cN_large}
                \frac{1}{8}\mathbb{E}\lb{\frac{\mathcal{V}_1^v\lp{\mathcal{V}_1^v-1}}{N-1}\,\,\mid\, \,B^v=1}
                =
                \frac{\psi^2}{4}+O\lp{\frac{1}{N}}.
            \end{equation}
        
            Combining equations \eqref{E: diamantidis_cN_small} and \eqref{E: diamantidis_cN_large} with $\varepsilon_N=2\phi/N^\gamma$,
            we obtain the sharper asymptotic
            \begin{equation}\label{E: diamantidis_cN_asymptotic}
                c_N
                =
                \frac{1}{2N}
                +\frac{\phi\psi^2}{2N^\gamma}
                +o\lp{\frac{1}{N^\gamma}}.
            \end{equation}
            In particular $c_N\to 0$, and Assumptions~\ref{A: continuous} and \ref{A: comparable} hold.

            We now identify the vague limit in Assumption~\ref{A: rarity}.
            In a small generation, $\widetilde{\mathcal{V}}^v$ is Wright-Fisher reproduction, so its contribution vanishes on
            $\Delta^V\times[0,1]^E\setminus\{\mathbf{0}_{V,E}\}$ by \cite[Proposition 2.1]{birkner2018}.
            In a large generation in deme $v$, the $\lfloor\psi N\rfloor$ offspring of the prolific couple inherit,
            at a fixed locus, one of the two parental copies from each of the two parents independently; thus the
            four parental copies each contribute $\mathrm{Bin}\lp{\lfloor\psi N\rfloor,1/2}$ gametes, and after
            normalisation by $2N$ this yields the limit $x_\psi$ in $\Delta$.
            Since large generations occur with probability $\varepsilon_N$ independently in the two demes, the event
            of simultaneous large generations has probability $\varepsilon_N^2=o(c_N)$, and does not contribute in
            the $1/c_N$ scaling.
            Consequently,
            \begin{equation*}
                \frac{1}{c_N}\Phi_N
                \;\to\;
                \lambda_\gamma\lp{\delta_{\lp{x^{(1)},0_E}}+\delta_{\lp{x^{(2)},0_E}}}
                \qquad\text{vaguely on }\Delta^V\times[0,1]^E\setminus\{\mathbf{0}_{V,E}\},
            \end{equation*}
            with
            \begin{equation*}
                \lambda_\gamma=\lim_{N \to \infty} \varepsilon_N/c_N
            \end{equation*}
            given by \eqref{E: diamantidis_cN_asymptotic}. This is exactly Assumption~\ref{A: rarity} with $\Phi=\Phi_\gamma$.
            By a direct computation one obtains
            $\mu_{\gamma,(1,2)}=\mu_{\gamma,(2,1)}$ as stated in the proposition and
            Assumption~\ref{A: migration_convergence} holds.

            All hypotheses of Theorem~\ref{T: quenched} are therefore satisfied, and the stated quenched convergence
            follows.
        \end{proof}

    \subsection{A two-deme model with random individual fitness and large migrations}\label{SS: beta_migration}

        We extend the model with random individual fitness from \cite{birkner2018} to a two deme model. We take $V=\{1,2\}$, $E=\{(1,2),(2,1)\}$, and $s(1)=s(2)=1$. We will take migration to be conservative, and will describe the precise migration patters of the model after we introduce the random fitness reproduction mechanism. 

        We now specify the joint law of the offspring matrices and the migration proportions in each generation.
        Fix $\alpha \in (1,2)$, and let $W$ be a nonnegative random variable with
        \begin{equation*}
            \lim_{z\to \infty}\PP\lp{W \ge z}z^\alpha =c_W
        \end{equation*}
        for some constant $c_W\in(0,\infty)$, and with $\mathbb{E}\lb{W}>0$. Let $\lp{W_i^v}_{i=1}^N$ be independent copies of $W$ that are also independent across demes. Define
        \begin{equation*}
            Z_N^v := \sum_{i,j=1}^N W_i^v W_j^v = \frac{1}{2}\lp{\sum_{i=1}^N W_i^v}^2 - \frac{1}{2}\sum_{i=1}^N (W_i^v)^2.
        \end{equation*}
        Given $\lp{W_i^v}_{i=1}^N$, we let
        \begin{equation*}
            (\mathcal{V}^v)_{v \in V} \stackrel{d}{=} \text{\emph{Multinomial}}\lp{N; \frac{W_1^vW_2^v}{Z_N^v}, \frac{W_2^vW_3^v}{Z_N^v}, \ldots \frac{W_{N-1}^vW_N^v}{Z_N^v}}^{\otimes 2}
        \end{equation*}
        
        We describe now the conservative migration patterns of the model. Fix parameters $a,b>0$. We let $m_{(1,2)}=m_{(2,1)} = m$ and sample $m$ as
        \begin{equation*}
            m \sim
            \begin{cases}
                \delta_0 &, \text{ with probability } 1-N^{1-\alpha} c_W\lp{\frac{2}{\mathbb{E}[W]}}^\alpha \alpha \frac{\Gamma(2-\alpha)\Gamma(\alpha)}{8}\\
                \text{\emph{Beta}}(a,b) &, \text{ otherwise.}
            \end{cases}
        \end{equation*}           

        We define now the relevant limiting objects of this model as follows:
        Define
        \begin{align*}
            c_\alpha
            &:=
            c_W\lp{\frac{2}{\mathbb{E}[W]}}^\alpha
            \alpha\,\frac{\Gamma(2-\alpha)\Gamma(\alpha)}{8},\\
            \kappa&:=\lp{0,0}, \text{ and }\\ \mu_{(1,2)}=\mu_{(2,1)}&:=0.
        \end{align*}
        Define, for $y\in(0,1)$,
        \begin{equation*}
            x_y:=\lp{\frac{y}{2},\frac{y}{2},0,0,\ldots}\in\Delta,
            \qquad
            x_y^{(1),1}=x_y,\;x_y^{(1),2}=\mathbf{0},
            \qquad
            x_y^{(2),1}=\mathbf{0},\;x_y^{(2),2}=x_y.
        \end{equation*}
        Let $\Lambda_\alpha$ denote the $\mathrm{Beta}(2-\alpha,\alpha)$ probability measure on $(0,1)$, i.e.
        \begin{equation*}
            d\Lambda_\alpha(y)
            =
            \frac{1}{\Gamma(2-\alpha)\Gamma(\alpha)}\,y^{1-\alpha}(1-y)^{\alpha-1}\,dy,
            \qquad
            d\nu_\alpha(y):=\frac{d\Lambda_\alpha(y)}{y^2}.
        \end{equation*}
        Then we set
        \begin{equation*}
            \Phi
            =
            \Phi^{\mathrm{fit}}+\Phi^{\mathrm{mig}},
        \end{equation*}
        where
        \begin{equation*}
            \Phi^{\mathrm{fit}}
            :=
            c_\alpha\int_{(0,1)}
            \lp{\delta_{\lp{x_y^{(1)},\mathbf{0}_E}}+\delta_{\lp{x_y^{(2)},\mathbf{0}_E}}}\,d\nu_\alpha(y),
        \end{equation*}
        and
        \begin{equation*}
            \Phi^{\mathrm{mig}}
            :=
            \int_{[0,1]} \delta_{\lp{\mathbf{0}_V,(y,y)}}\frac{\Gamma(a+b)}{\Gamma(a)\Gamma(b)} y^{1-a}(1-y)^{1-b}dy.
        \end{equation*}

        Convergence of this model is then established by the following proposition.
        
        \begin{proposition}\label{P: beta_migration}
            Let $\overline{\chi}^{N,n}$ and $\mathcal{G}_N$ denote the ancestral process and pedigree of a two-deme model with random individual fitness and large migrations. 
            Let $\chi^n$ denote a $\Psi$-driven $(\kappa,\mu)$-$n$-coalescent where $\Psi$ has intensity measure $dt \otimes d(h_V)_*\Phi$. Suppose that, as $N$ goes to infinity, that Assumption~\ref{A: IC} holds. Then
            \begin{equation*}
                \PP\lp{cd_V\lp{\overline{\chi}^{N,n}} \in\cdot\,\mid\, \mathcal{A}_N} \toL \PP\lp{\chi^n \in \cdot \,\mid\, \Psi}.
            \end{equation*}
        \end{proposition}
        
        \begin{proof}
            By symmetry we have that $c_N^1 = c_N^2 := c_N$, where $c_N$ is given by 
            \begin{equation*}
                c_N
                \sim
                c_W\lp{\frac{2}{\mathbb{E}[W]}}^\alpha
                \alpha\,\frac{\Gamma(2-\alpha)\Gamma(\alpha)}{8}\,N^{1-\alpha}
            \end{equation*}
            by \cite[Lemma 2.2]{birkner2018}.
            In particular, Assumptions~\ref{A: continuous} and
            \ref{A: comparable} hold.
        
            We now identify the vague limit in Assumption~\ref{A: rarity}.
            In the one-deme random-fitness model, the sequence of intensity measures $(\Phi_N)_v$ associated with the
            ordered offspring frequencies in each deme $v$ converges by \cite[Proposition 2.3]{birkner2018} to
            \begin{equation*}
                \int_0^1 \delta_{x_y} \text{\emph{Beta}}(\alpha,2-\alpha)(dy),
            \end{equation*}
            reflecting that a single highly fit parent contributes a
            macroscopic fraction $y$ of the parental gametes at a locus, split equally between its two copies.
            
            Finally, consider migration. By construction,
            \begin{equation*}
                \PP\{m\neq 0\}
                =
                N^{1-\alpha}c_W\lp{\frac{2}{\mathbb{E}[W]}}^\alpha
                \alpha\,\frac{\Gamma(2-\alpha)\Gamma(\alpha)}{8}
                =
                c_N,
                \qquad
                \PP\{m=0\}=1-c_N,
            \end{equation*}
            and $m_{(1,2)}=m_{(2,1)}=m$. In particular, by construction it is clear that
            \begin{equation*}
                \frac{1}{c_N}\lp{\Phi_N}_e \to \text{\emph{Beta}}(a,b)
            \end{equation*}
            on $(0,1]$.

            Since we sample the migration and reproduction events independently, it follows from Corollary~\ref{C: rarity_product_form} that $\frac{1}{c_N} \Phi_N$ converges vaguely to $\Phi$, as described.
        
            All hypotheses of Theorem~\ref{T: quenched} are satisfied, and the claimed quenched convergence follows.
        \end{proof}

    \subsection{\texorpdfstring{A discrete approximation of a toroidal spatial $\Xi$ Fleming-Viot genealogy}{}}\label{SS: xi_model}

        We describe here a discrete approximation of a $\Xi$ Fleming-Viot genealogy. The $\Lambda$ Fleming-Viot coalescent $\chi^n$ of \cite{etheridge11} on a large torus can be described as follows: Fix a measure $\nu$ on $\R_+$ so that $\int r^2 d\nu(r) < \infty$. Let $\Gamma$ denote a Poisson point process on $\R_+ \times \R^2 \times \R_+$ with intensity measure $dt \otimes dx \otimes d\nu$. We say that an atom $(t,x,r)$ ``affects" a point $y$ if $\norm{x-y} < r$. We couple $n$ $\R^2$-valued processes $\lp{\widetilde{W}_i}_{i=1}^n$ with $\Gamma$ as follows:
        \begin{itemize}
            \item between intersections of any of the $\lp{\widetilde{W}_i}_{i=1}^n$, they processes evolve as if they were independent Brownian motions;
            \item if two of these processes coalesce, they remain identified together for all time;
            \item for any atom $(t,x,r)$ of $\Gamma$, $j$ is sampled from some distribution on $\N$ and $j$ points $y_1,y_2,\ldots, y_j$ are selected uniformly at random in $B(x,r)$. All of the particles that are affected by $(t,x,r)$ instantaneously jump to one of these points uniformly at random.
        \end{itemize}
        $\chi^n = \lp{\pi^n, x^n}$ is then an $\En(\R^2)$-valued Markov process, where $\pi^n$ is the equivalence relation
        \begin{equation*}
            i \sim_{\pi^n(t)} j \text{\quad if } \widetilde{W}_i(t) = \widetilde{W}_j(t),
        \end{equation*}
        and $x_i^n$ is the position of the particles corresponding to the $i$th block of $\pi^n(t)$.

        We discretely approximate this process on the $2$-Torus as follows:
        Let $V := \Z^2 / \langle L_1,L_2 \rangle_{\Z}$ denote a discretized $2$-Torus with latitude and longitude of lengths $L_1$ and $L_2$ in $\N$, respectively. Define the edge set to be $E := V \times V \setminus \{(v,v) : v \in V\}$. For convenience we set $s(v) = 1$ for all $v$ in $V$. We fix parameters $\rho,\sigma,\gamma > 0$, $a_v,b_v > 0$ for each deme $v$ in $V$, and a finite measure $\nu$ on $\R_+$ with $\int r d\nu(r) < \infty$. $\rho$ will correspond to a scaling of the lattice and $\sigma$ to migration rates. $\gamma$ and $\nu$ will relate to the intensity of the analogue $\Gamma$ described above for the spatial $\Lambda$ Fleming-Viot process. We sample $\lp{\mathcal{V}, m}$ as follows:
        \begin{itemize}
            \item With probability $1-\frac{\gamma}{2N}$ we set
            \begin{align*}
                m_{(v,w)} = m_{(w,v)} &\sim \frac{1}{N}\text{\emph{Binomial}}\lp{N; \frac{1}{N}\frac{\sigma}{2}} \text{ if $v$ is a nearest neighbor to $w$}\\
                m_{(v,w)} = m_{(w,v)} &= 0, \text{ if  $v$ and $w$ are not nearest neighbors, and}\\
                \lp{\mathcal{V}^v}_{v \in V} &\sim \text{Multinomial}\lp{N; \frac{2}{N(N-1)}, \ldots, \frac{2}{N(N-1)}}^{\otimes L_1 \times L_2}.
            \end{align*}
            That is, we have conservative migration between nearest neighbors and the reproduction within each deme is Wright-Fisher.
            \item With probability $\frac{\gamma}{2N}$ we have an extreme reproductive event. Choose a deme $w^*$ uniformly at random from $V$ and sample $r \sim \nu$. A deme $v^*$ is then selected uniformly at random from the collection of demes $w$ with $\norm{w-w^*} < r\rho$. We sample $(\mathcal{V}, m)$ as follows: 
            \begin{itemize}
                \item For all demes $w$ with $\norm{w-w^*} \geq r\rho$ we set
                    \begin{align*}
                        m_{(w,w')} &= 0 \text{ for all } w'\neq w \text{ and }\\
                        \mathcal{V}^w &\sim \text{Multinomial}\lp{N; \frac{2}{N(N-1)}, \ldots, \frac{2}{N(N-1)}}.
                    \end{align*}
                    That is, demes far enough from $w^*$ have no migrations to or from them and simply reproduce via Wright-Fisher dynamics.
                \item For a deme $w \neq v^*$ with $\norm{w-w^*} < r\rho$ we sample the offspring distribution and migrations by
                    \begin{align*}
                        m_{(w,w')} &\sim 
                        \begin{cases}
                            \delta_0, \text{ if } w'\neq v^*\\
                            \text{Beta}(a_w,b_w), \text{ if } w' = v^*
                        \end{cases}
                        \text{\quad and }\\
                        \mathcal{V}^w &\sim \text{Multinomial}\lp{N; \frac{2}{N^*(w)(N^*(w)-1)}, \ldots, \frac{2}{N^*(w)(N^*(w)-1)}}.
                    \end{align*}
                    That is, demes within distance $r\rho$ of $w^*$ migrate according to a Beta distribution to $v^*$ and those individuals that remain in the deme reproduce according to Wright-Fisher dynamics.
                \item Inside the deme $v^*$ we set
                    \begin{align*}
                        m_{(v^*,w)} = 0 &, \text{ for all } w \neq v^*, \text{ and}\\
                        \widetilde{\mathcal{V}}^{v^*} &\sim \widetilde{\Phi}_N,
                    \end{align*}
                    where we assume that $\widetilde{\Phi}_N$ converges weakly in $\Delta$ to a measure $\Xi$. 
            \end{itemize}
        \end{itemize}

        A realization of the large event dynamics of this model are shown in Figure~\ref{F: fleming_viot_example}.
        \begin{figure}
            \centering
            \includegraphics[width=0.7\linewidth]{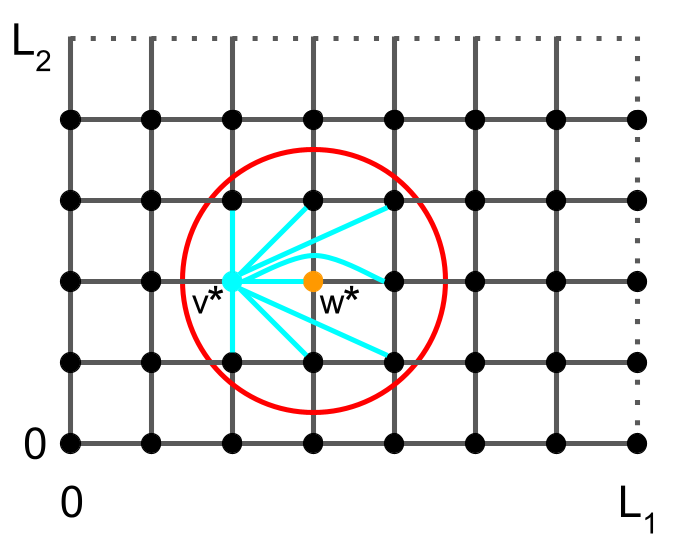}
            \caption{In the figure we see a discretized torus $\Z^2 /\langle L_1,L_2\rangle_\Z$ with $L_1 = 7$ and $L_2 = 5$. At each point on the discrete torus, drawn as a solid ball, is a deme of the population. Emanating from each deme are its nearest neighor edges. A point $w^*$, colored orange, and a radius $r$ is chosen uniformly at random. The ball of radius $r\rho$ centered at $w^*$ is drawn in red. A point $v^*$ is then selected uniformly at random from that ball and drawn in cyan. All of the demes within the ball of radius $r\rho$ from $w^*$ then send a Beta-distributed proportion of their population to $v^*$. These migrations are shown in cyan.}
            \label{F: fleming_viot_example}
        \end{figure}

        We describe now the limiting objects of this model as follows:
        For any $v \in V$ and $r \in \R_+$ we define $\widehat{\Phi}_{v,r}$ to be the unique product measure on $\Delta^V \times [0,1]^E$ such that $\widehat{\Phi}_{v,r}$ satisfies the marginal distributions
        \begin{align*}
            \lp{\widehat{\Phi}_{v,r}}_e &= 
            \begin{cases}
                \text{Beta} (a_w, b_w) &, \text{ if } e = (w,w'), w' = v, \text{ and } \norm{w-v} < r\rho\\
                0 &, \text{ otherwise } 
            \end{cases} \text{\quad \quad and }\\
            \lp{\widehat{\Phi}_{v,r}}_w &=
            \begin{cases}
                \delta_0 &, \text{ if } w \neq v\\
                \Xi &, \text{ if } w = v
            \end{cases}.
        \end{align*}
        Let
        \begin{equation*}
            \Phi_{\rm FV} := \frac{\lambda}{1+\frac{\lambda}{|V|} \int_{\Delta} \langle x,x\rangle d\Xi(x)}\sum_{v \in V} \frac{1}{|V|} \int_{\R_+} \widehat{\Phi}_{v,r\rho} d\nu(r).
        \end{equation*}
        Let $\kappa_{\rm FV} = \lp{\kappa_v}_{v \in V}$ be defined by
        \begin{equation*}
            \kappa_v = 1 - \frac{\frac{\lambda}{|V|} \int_{\Delta} \langle x,x\rangle d\Xi(x)}{1+\frac{\lambda}{|V|} \int_{\Delta} \langle x,x\rangle d\Xi(x)},
        \end{equation*}
        and $\mu_{\rm FV} = \lp{\mu_e}_{e \in E}$ by
        \begin{equation*}
            \mu_{e} = \sigma \frac{\frac{\lambda}{|V|} \int_{\Delta} \langle x,x\rangle d\Xi(x)}{1+\frac{\lambda}{|V|} \int_{\Delta} \langle x,x\rangle d\Xi(x)}.
        \end{equation*}

        The convergence of this model is then described by the following proposition.
        
        \begin{proposition}\label{P: xi_fleming_viot}
            Let $\overline{\chi}^{N,n}$, $\mathcal{G}_N$ denote the ancestral process and pedigree associated to the discrete approximation of the spatial $\Xi$ Fleming-Viot process described above. Suppose that, as $N$ goes to infinity, Assumption~\ref{A: IC} holds. 
            Let $\chi^n_{\rm FV}$ denote a $\Psi_{\rm FV}$-driven $(\kappa_{\rm FV},\mu_{\rm FV})$-$n$-coalescent with intensity measure $dt \otimes d(h_V)_*\Phi_{\rm FV}$
            Then
            \begin{equation*}
                \PP\lp{cd_V\lp{\overline{\chi}^{N,n}} \in \cdot \,\mid\, \mathcal{A}_N} \toL \PP\lp{\chi^n_{\rm FV} \in \cdot \,\mid\, \Psi_{\rm FV}}.
            \end{equation*}
        \end{proposition}
        \begin{proof}
            We verify Assumptions~\ref{A: continuous}, \ref{A: comparable}, \ref{A: rarity},
            \ref{A: migration_convergence}, \ref{A: integrability}, and \ref{A: no_total_migrations} for the discrete model with the
            choices of $\Phi,\kappa,\mu$ stated in the proposition, and then apply Theorem~\ref{T: quenched}.
            
            We now check Assumptions~\ref{A: continuous} and \ref{A: comparable}. By translation invariance of
            the torus and the fact that the extreme-event center $v^*$ is uniform on $V$, the joint law of the reproduction and
            migration mechanism is the same in every deme, and therefore $c_N^v=c_N^{v_0}=: c_N$ for all $v\in V$, so
            Assumption~\ref{A: comparable} holds with $c(v)\equiv 1$ once we show $c_N\to 0$. To see this, consider two genes
            sampled uniformly from two distinct individuals in deme $v_0$ and condition on neither migrating. In a baseline generation,
            reproduction in deme $v_0$ is Wright-Fisher, so the one-step coalescence probability is $1/(2N)+o(1/N)$. In an extreme
            generation, which occurs with probability $\lambda/(2N)$, the only additional contribution to coalescence in deme $v_0$
            beyond the Wright-Fisher $O(1/N)$ term is when $v^*=v_0$, which occurs with probability $1/|V|$, and then two sampled
            genes coalesce with probability asymptotically $\int_{\Delta}\langle x,x\rangle\,d\Xi(x)$ under the limiting offspring law.
            Combining these contributions yields
            \begin{equation*}
            c_N
            =
            \frac{1}{2N}\lp{1+\frac{\lambda}{|V|}\int_{\Delta}\langle x,x\rangle\,d\Xi(x)}+o\lp{\frac{1}{N}},
            \end{equation*}
            so in particular $c_N^{v_0}\to 0$, proving Assumptions~\ref{A: continuous} and Assumption~\ref{A: comparable} hold with
            $c(v)\equiv 1$.
            
            We now check Assumption~\ref{A: rarity} with the $\Phi$ given in the proposition. By \cite[Proposition 2.1]{birkner2018} we know the contribution of the Wright-Fisher components to $\Phi_N$ are negligible in the $c_N$ rescaling.
            In an extreme generation, conditional on $(v^*,r)$, the construction is product across coordinates, with the only
            nontrivial offspring marginal at deme $v^*$ given by $\widetilde{\Phi}_N$ and all migration marginals specified in the
            model; since $\widetilde{\Phi}_N\stackrel{w}{\to}\Xi$ weakly in $\Delta$ the assumption follows.
            
            Assumption~\ref{A: migration_convergence} follows with the described migration rates by a direct computation.
            
            To demonstration Assumption~\ref{A: integrability}, it suffices by remark~\ref{R: integrability_condition} to show that the marginals $\Phi_e$ and $\Phi_v$ satisfy
            \begin{equation}\label{E: regularity}
                \int_{[0,1]} m_e d\Phi_e(m_e) < \infty \text{\quad and \quad } \int_{\Delta} \langle x^v,x^v \rangle d\Phi_v(x^v).
            \end{equation}
            Under $\Phi$, the only nontrivial offspring marginal is $\Xi$ at the affected deme and all other offspring marginals are
            $\delta_0$, while the only nontrivial migration marginals are Beta$(a_w,b_w)$. By assumption
            $\int_{\Delta}\langle x,x\rangle\,d\Xi(x)<\infty$, and Beta laws have finite first moments Equation~\eqref{E: regularity} follows, and Assumption~\ref{A: integrability} with it.
            
            We now check Assumption~\ref{A: no_total_migrations}. Under $\Phi$, each nonzero migration
            coordinate $m_e$ is either $0$ or Beta$(a_w,b_w)$, hence takes values in $[0,1)$ almost surely and has no atom at $1$.
            Therefore, for every deme $v$, the law of $\sum_{e=(v,w)\in E} m_e$ under $\Phi$ has no atom at $1$, proving
            Assumption~\ref{A: no_total_migrations}.
            
            We have verified the assumptions of Theorem~\ref{T: quenched} for the present
            model with the $\Phi$ given in the proposition. The claim thus follows.
        \end{proof}

        Since we describe this model as a discrete approximation of a $\Xi$ Fleming-Viot process, this should be justified by taking a suitable scaling limit of the discrete torus. Since this is outside the purview of the present work, we outline roughly the scaling limit one would take in Section~\ref{S: discussion}.

\section{\texorpdfstring{Existence and coupling of $\Psi$-driven $(\kappa,\mu)$-$n$-coalescents}{}}\label{SS: driven_coupling}

    We give a pathwise construction of $\Psi$-driven $(\kappa,\mu)$-$n$-coalescents that avoids any ambiguity coming from the fact that $\Psi$ may have dense atom times. The construction proceeds by enriching $\Psi$ with i.i.d.\ auxiliary marks and then retaining only the \emph{effective} events that can actually move the chain, in the spirit of \cite[Section~5]{abfw25}.

    \subsection*{Marked driving process and state-dependent thinnings}

    Let $\Psi$ be a Poisson point process on $\R_+ \times \lp{\Delta^V \times [0,1]^E}$ with intensity $dt \otimes d\Phi$, where $\Phi$ satisfies Assumption~\ref{A: integrability}, i.e.\ for each $\xi \in \En(V)$,
    \begin{equation*}
        \int_{\Delta^V \times [0,1]^E} \lp{1-q_n(x,m)(\xi,\xi)}\,d\Phi\lp{x,m} < \infty.
    \end{equation*}
    Since the time-marginal of $dt$ is nonatomic, $\Psi$ has almost surely no simultaneous atom times.

    We now attach auxiliary marks to each atom of $\Psi$. Let $\widehat{\Psi}$ be a Poisson point process on
    \begin{equation*}
        \R_+ \times \lp{\Delta^V \times [0,1]^E} \times [0,1]^{\En(V)} \times [0,1]
    \end{equation*}
    with intensity measure
    \begin{equation*}
        dt \otimes d\Phi\lp{x,m} \otimes \lambda_{\En(V)}(du) \otimes \lambda(dv),
    \end{equation*}
    where $\lambda_{\En(V)}$ denotes Lebesgue measure on $[0,1]^{\En(V)}$ and $\lambda$ denotes Lebesgue measure on $[0,1]$.
    We write atoms of $\widehat{\Psi}$ as $(t,(x,m),u,v)$, where $u=\lp{u^\xi}_{\xi\in\En(V)}$ is a vector of auxiliary points uniformly sampled on the unit interval.

    For each $\xi\in\En(V)$ define the state-dependent thinning
    \begin{equation}\label{E: Psi_xi_def}
        \Psi_\xi
        :=
        \lcb{(t,(x,m)):\exists\,u,v\ \text{s.t.}\ (t,(x,m),u,v)\in\widehat{\Psi}
        \ \text{and}\ u^\xi \le 1-q_n(x,m)(\xi,\xi)}.
    \end{equation}

    \begin{lemma}[Poisson thinnings and local finiteness]\label{L: thinning_local_finite}
        For each $\xi\in\En(V)$, $\Psi_\xi$ is a Poisson point process on $\R_+\times \lp{\Delta^V \times [0,1]^E}$ with intensity measure
        \begin{equation*}
            dt \otimes \lp{1-q_n(x,m)(\xi,\xi)}\,d\Phi(x,m).
        \end{equation*}
        In particular, for every $T<\infty$,
        \begin{equation*}
            \#\lp{\Psi_\xi \cap \lp{[0,T]\times \lp{\Delta^V \times [0,1]^E}}}<\infty
            \qquad\text{a.s.}
        \end{equation*}
        Moreover, since $\En(V)$ is finite for fixed $n$, the union
        \begin{equation*}
            \Psi^\star := \bigcup_{\xi\in\En(V)} \Psi_\xi
        \end{equation*}
        has almost surely finitely many atoms in $[0,T]\times \lp{\Delta^V \times [0,1]^E}$ for each $T<\infty$.
    \end{lemma}

    \begin{proof}
        The finiteness on compacts follows from Assumption~\ref{A: integrability} and the fact that a Poisson random variable with finite mean is almost surely finite. The final claim follows because $\En(V)$ is finite, hence $\Psi^\star$ is a finite superposition of locally finite point processes.
    \end{proof}

    Fix an enumeration $\En(V)=\{\xi^{(1)},\dots,\xi^{(M)}\}$.
    For each $(x,m)\in \Delta^V \times [0,1]^E$ and each $\xi\in\En(V)$, define a measurable map
    \begin{equation*}
        J_{x,m,\xi}:[0,1]\to \En(V)\setminus\{\xi\}
    \end{equation*}
    by setting, for $v\in[0,1]$,
    \begin{equation*}
        J_{x,m,\xi}(v) = \xi^{(j)},
    \end{equation*}
    where $j$ is the least index such that $\xi^{(j)}\neq \xi$ and
    \begin{equation*}
        v \le \sum_{\substack{1\le k\le j\\ \xi^{(k)}\neq \xi}}
        \frac{q_n(x,m)(\xi,\xi^{(k)})}{1-q_n(x,m)(\xi,\xi)}.
    \end{equation*}
    \noindent
    If $q_n(x,m)(\xi,\xi)=1$, then $1-q_n(x,m)(\xi,\xi)=0$ and $J_{x,m,\xi}$ may be defined arbitrarily.

    Using $J$, define the random map $F_{x,m,u,v}:\En(V)\to\En(V)$ by
    \begin{equation}\label{E: jump_map_def}
        F_{x,m,u,v}(\xi)
        :=
        \begin{cases}
            \xi, & u^\xi > 1-q_n(x,m)(\xi,\xi),\\
            J_{x,m,\xi}(v), & u^\xi \le 1-q_n(x,m)(\xi,\xi).
        \end{cases}
    \end{equation}

    \subsection*{Construction of the $\Psi$-driven $(\kappa,\mu)$-$n$-coalescent}

    Let $K_n(\kappa)+M_n(\mu)$ denote the generator of the $(0,\kappa,\mu)$-$n$-coalescent on the finite state space $\En(V)$.
    Let $\lp{Y^{(k)}}_{k\ge 0}$ be i.i.d.\ càdlàg Markov processes with generator $K_n(\kappa)+M_n(\mu)$, independent of $\widehat{\Psi}$ and each other.

    By Lemma~\ref{L: thinning_local_finite}, we may list the atom times of $\Psi^\star$ increasingly as
    \begin{equation*}
        0<\tau_1<\tau_2<\cdots,
        \qquad \tau_k\to\infty.
    \end{equation*}
    For each $k$, denote the unique atom of $\widehat{\Psi}$ at time $\tau_k$ by
    \begin{equation*}
        \lp{\tau_k,(x_k,m_k),u_k,v_k}.
    \end{equation*}

    \begin{definition}[$\Psi$-driven $(\kappa,\mu)$-$n$-coalescent, pathwise]\label{D: Psi_driven_rigorous}
        Fix an initial state $\chi^n(0)=\xi_0\in\En(V)$.
        Define $\chi^n$ recursively as follows.

        Set $\chi^n(0)=\xi_0$ and, for $t\in[0,\tau_1)$, let
        \begin{equation*}
            \chi^n(t) := Y^{(0)}(t).
        \end{equation*}
        Suppose $\chi^n$ has been defined on $[0,\tau_k)$.
        Define the left limit $\chi^n(\tau_k-):=\lim_{t\uparrow\tau_k}\chi^n(t)$, and set
        \begin{equation*}
            \chi^n(\tau_k) := F_{x_k,m_k,u_k,v_k}\lp{\chi^n(\tau_k-)}.
        \end{equation*}
        For $t\in[\tau_k,\tau_{k+1})$, define
        \begin{equation*}
            \chi^n(t) := Y^{(k)}(t-\tau_k)
            \quad\text{with initial condition } Y^{(k)}(0)=\chi^n(\tau_k).
        \end{equation*}
    \end{definition}

    \begin{remark}\label{R: conditional_kernel}
        By construction, conditional on $(x,m)$ at an atom time and on the pre-jump state $\xi$,
        the post-jump state has distribution $q_n(x,m)(\xi,\cdot)$.
        Indeed, \eqref{E: jump_map_def} yields
        \begin{equation*}
            \PP\lp{\chi^n(\tau_k)=\xi \,\,\mid\, \, \chi^n(\tau_k-)=\xi,(x_k,m_k)}=q_n(x_k,m_k)(\xi,\xi),
        \end{equation*}
        and, for $\eta\neq \xi$,
        \begin{equation*}
            \PP\lp{\chi^n(\tau_k)=\eta \,\,\mid\, \, \chi^n(\tau_k-)=\xi,(x_k,m_k)}=q_n(x_k,m_k)(\xi,\eta).
        \end{equation*}
        The role of the thinning is solely to ensure that the set of times at which a change is possible is locally finite; it does not alter the induced transition kernel at those times.
    \end{remark}

    \subsection*{Coupling $l$ conditionally independent copies given $\Psi$}

    We now construct $l$ copies $\lp{\chi_i^n}_{i=1}^l$ on the same driving $\Psi$ such that, conditional on $\Psi$, they are independent.

    Let $\widehat{\Psi}^{(l)}$ be a  Poisson point process on
    \begin{equation*}
        \R_+ \times \lp{\Delta^V \times [0,1]^E} \times \lp{[0,1]^{\En(V)}\times[0,1]}^l
    \end{equation*}
    with intensity
    \begin{equation*}
        dt \otimes d\Phi\lp{(x,m)} \otimes \lp{\lambda_{\En(V)}(du)\otimes\lambda(dv)}^{\otimes l}.
    \end{equation*}
    Write its atoms as $\lp{t,(x,m),\lp{u_i,v_i}_{i=1}^l}$.
    For each $i$ define the coordinatewise jump map $F^{(i)}_{x,m,u_i,v_i}$ by \eqref{E: jump_map_def}.

    Let $\lp{Y^{(k)}_i}_{k\ge 0,\, i\in[l]}$ be i.i.d.\ Markov processes with generator $R$, independent of $\widehat{\Psi}^{(l)}$.
    Define each $\chi_i^n$ by the same interlacing construction as in Definition~\ref{D: Psi_driven_rigorous}, but using the marks $\lp{u_i,v_i}$ and background chains $\lp{Y^{(k)}_i}$.

    \begin{lemma}[Conditional independence given $\Psi$]\label{L: cond_indep}
        With the above construction, the processes $\lp{\chi_i^n}_{i=1}^l$ are i.i.d.\ with the law of the $\Psi$-driven $(\kappa,\mu)$-$n$-coalescent, and are conditionally independent given $\Psi$.
    \end{lemma}

    \begin{proof}
        Given $\Psi$, the marks $\lp{u_i,v_i}_{i\in[l]}$ and the background chains $\lp{Y_i^{(k)}}$ are independent across $i$ by construction, and each coordinate has the same conditional transition mechanism described in Remark~\ref{R: conditional_kernel}. This yields both the correct marginal law and conditional independence.
    \end{proof}

    With this construction we are now able to prove Lemma~\ref{L: driven_coupling}.
    \begin{proof}[Proof of Lemma~\ref{L: driven_coupling}]
        Consider the $l$-tuple process $\lp{\chi_i^n}_{i=1}^l$ taking values in $\En(V)^l$.
        Between atom times of $\widehat{\Psi}^{(l)}$ it evolves as $l$ independent $(0,\kappa,\mu)$-$n$-coalescents, hence has generator $R$ by our earlier definition of $R$ on $\En(V)^l$.

        Fix an atom time $t$ of $\Psi$ with mark $(x,m)$ and pre-jump state $\vec{\xi}\in\En(V)^l$.
        By Remark~\ref{R: conditional_kernel} and independence of the coordinate marks $\lp{u_i,v_i}$,
        conditional on $\vec{\chi}^n(t-)=\vec{\xi}$ and on $(x,m)$, the post-jump state has transition kernel
        \begin{equation*}
            \PP\lp{\lp{\chi_i^n}_{i=1}^l(t)=\vec{\eta}\,\,\mid\, \, \lp{\chi_i^n}_{i=1}^l(t-)=\vec{\xi},(x,m)}
            =
            \prod_{i=1}^l q_n(x,m)(\xi_i,\eta_i)
            =: H_l(x,m)(\vec{\xi},\vec{\eta}).
        \end{equation*}
        In particular,
        \begin{equation*}
            \PP\lp{\lp{\chi_i^n}_{i=1}^l(t)=\vec{\xi}\,\,\mid\, \, \lp{\chi_i^n}_{i=1}^l(t-)=\vec{\xi},(x,m)}
            = H_l(x,m)(\vec{\xi},\vec{\xi})
            = \prod_{i=1}^l q_n(x,m)(\xi_i,\xi_i).
        \end{equation*}

        Let $f:\En(V)^l\to\R$ be bounded. By annealing over $\Psi$ we see that the generator of $\vec{\chi}^n$ is
        \begin{equation*}
            \mathcal{L} f(\vec{\xi})
            =
            R f(\vec{\xi})
            +
            \int_{\Delta^V \times [0,1]^E}
            \lp{\sum_{\vec{\eta}\in\En(V)^l} f(\vec{\eta})\,H_l(x,m)(\vec{\xi},\vec{\eta}) - f(\vec{\xi})}\,d\Phi\lp{x,m}.
        \end{equation*}
        Equivalently,
        \begin{equation*}
            \mathcal{L}
            =
            R
            +
            \int_{\Delta^V \times [0,1]^E} \lp{H_l(x,m)-I}\,d\Phi\lp{x,m},
        \end{equation*}
        which is the claimed form.
    \end{proof}

\section{Coupling coalescents on the pedigree}\label{S: coupling}

    In Section~\ref{SS: coupling}, we rigorously introduce the coupling of the conditionally independent realizations $\lp{\overline{\chi}_i^{N,n}}_{i=1}^l$ of $\overline{\chi}^{N,n}$ with respect to $\mathcal{A}_N$. This coupling is then used to rigorously establish the unproved lemmas of Section~\ref{S: proof_main}. The generic behavior of small jumps described in Lemma~\ref{L: small_jumps} is proven in Section~\ref{SS: small_jumps}. The behavior of large jumps described in Lemma~\ref{L: big_jumps} is proved in Section~\ref{SS: big_jumps}. Lemma~\ref{L: combining_scales}, which combines the behavior of the two scales is the subject of Section~\ref{SS: combining_scales}.

    \subsection{The coupling}\label{SS: coupling}

        In order to prove Theorem~\ref{T: quenched}, it is convenient to work with $l$ conditionally
        independent realizations of the time-rescaled ancestral process. To this end, we couple
        $\chi_1^{N,n}=\lp{\chi_1^{N,n}(k)}_{k\in\Z_+},\ldots,\chi_l^{N,n}=\lp{\chi_l^{N,n}(k)}_{k\in\Z_+}$
        so that they are conditionally independent with respect to $\mathcal{A}_N$.
        
        We proceed in the coupling akin to the description in \cite[Section 6.1]{abfw25}. The point is that,
        even if each copy is naturally described on $\En(V)$, the \emph{$l$-tuple} of copies requires additional
        bookkeeping. One must remember which ancestral diploid individuals are shared by different copies,
        since this is where the common pedigree creates correlations after averaging over $\mathcal{A}_N$.
        
        We first fix notation for the state space. For any set $M$, write $\mathcal{P}(M)$ for its power set and
        \begin{equation*}
            \mathcal{P}_{1,2}(M)
            \;:=\;
            \lb{A\in\mathcal{P}(M): \#A\in\{1,2\}}.
        \end{equation*}
        Define
        \begin{equation*}
            \mathcal{C}_{n,l}
            \;:=\;
            \mathcal{P}([n])^l \setminus \lb{(\varnothing,\ldots,\varnothing)}.
        \end{equation*}
        An element $C=(C^{(1)},\ldots,C^{(l)})\in\mathcal{C}_{n,l}$ represents the contents of $l$ unlinked loci on a single ancestral chromosome that
        carries ancestral material for copy $j$ from the block $C^{(j)}$. Any given component may be empty,
        but not all components are simultaneously empty.
        
        A single ancestral \emph{diploid individual} can carry one or two such chromosomes. We therefore encode
        an ancestral individual by an element $I\in\mathcal{P}_{1,2}(\mathcal{C}_{n,l})$. We also record the deme
        in which this individual lives. This leads to the state space
        \begin{equation*}
            \mathcal{H}_{n,l}(V)
            \;:=\;
            \lcb{
                \xi \subset V\times \mathcal{P}_{1,2}(\mathcal{C}_{n,l})
                \,:\,
                \bigsqcup_{(v,I)\in\xi}\;\bigsqcup_{C\in I} C^{(j)} = [n]
                \text{ for all } j\in[l]
            },
        \end{equation*}
        where $\sqcup$ denotes disjoint union, and we interpret $C^{(j)}=\varnothing$ as contributing nothing to
        the union. The disjoint-union conditions say that, in each copy $j$, every sample index appears in
        exactly one block across all chromosomes of all ancestral individuals.
        
        Given $\xi\in\mathcal{H}_{n,l}(V)$, the induced configuration of copy $j$ is obtained by projection to the
        $j$th component. More precisely, define maps
        \begin{equation*}
            pr_j:\mathcal{H}_{n,l}(V)\to \mathcal{S}_n(V),
            \qquad j\in[l],
        \end{equation*}
        as follows. For $\xi\in\mathcal{H}_{n,l}(V)$ and a fixed $(v,I)\in\xi$, the element $pr_j(\xi)$ contains,
        in deme $v$, the same grouping into one or two chromosomes, but with each chromosome
        $C=(C^{(1)},\ldots,C^{(l)})\in I$ replaced by the single block $C^{(j)}$ and with empty blocks removed.
        In particular, $pr_j(\xi)$ records which blocks of copy $j$ are in the same diploid individual, together
        with the deme label of that individual.
        
        We then apply the complete-dispersion map $cd_V$ from the single-locus state space to forget the
        short-lived pairing information within individuals. We extend $cd_V$ to $\mathcal{H}_{n,l}(V)$ by
        \begin{equation*}
            cd_V(\xi)
            \;:=\;
            \lp{cd_V\!\circ\,pr_1(\xi),\ldots,cd_V\!\circ\,pr_l(\xi)}
            \;\in\;
            \En(V)^l.
        \end{equation*}
        We call $\xi\in\mathcal{H}_{n,l}(V)$ completely dispersed if $\xi$ carries no information beyond its
        $l$ complete dispersions, in the sense that each chromosome carries ancestral material from exactly one
        copy. Equivalently, $\xi$ lies in the image of the canonical embedding
        \begin{equation*}
            \iota:\En(V)^l\to \mathcal{H}_{n,l}(V),
        \end{equation*}
        defined by placing every block of copy $j$ on its own chromosome of type
        $(\varnothing,\ldots,\varnothing,C,\varnothing,\ldots,\varnothing)$ (with $C$ in the $j$th slot) in the
        appropriate deme.
        
        We now construct the coupled processes. We work on a probability space that carries the pedigree
        $\mathcal{G}_N$ and $l$ independent collections of Mendelian randomness, one for each copy. Conditional
        on $\mathcal{A}_N$, we build $l$ families of ancestral lines by running the $l$ copies through the \emph{same}
        pedigree realisation and using the \emph{independent} Mendelian randomness to decide gene inheritance.
        This induces $l$ ancestral processes $\chi_1^{N,n},\ldots,\chi_l^{N,n}$ with values in $\En(V)$, and hence an
        $l$-tuple process $\lp{\chi_i^{N,n}}_{i=1}^l$.
        By construction,
        \begin{equation*}
            \PP\lp{\chi_1^{N,n} \in \cdot\,\middle|\,\mathcal{A}_N}
            \;=\;\cdots\;=\;
            \PP\lp{\chi_l^{N,n} \in \cdot\,\middle|\,\mathcal{A}_N},
            \qquad
            \chi_1^{N,n},\ldots,\chi_l^{N,n}\ \text{ are i.i.d.\ given }\ \mathcal{A}_N.
        \end{equation*}
        
        Finally, we lift $\lp{\chi_i^{N,n}}_{i = 1}^l$ to a process with values in $\mathcal{H}_{n,l}(V)$ by keeping, at
        each discrete time $k$, the collection of ancestral diploid individuals that carry ancestral material in at
        least one copy, together with the allocation of the $l$ copies to the (up to) two chromosomes of each such
        individual. This is exactly the additional structure that remembers how the different copies share
        individuals on the fixed pedigree.
        
        We then define $\overline{\chi}_1^{N,n},\ldots,\overline{\chi}_l^{N,n}$ to be the time-rescaled
        $\mathcal{D}\lp{\R_+,\Sn(V)}$-valued processes given by
        \begin{equation*}
            \overline{\chi}_j^{N,n}(t)
            \;:=\;
            \chi_j^{N,n}\!\lp{\lfloor t (c_N^{v_0})^{-1}\rfloor},
            \qquad t\ge 0,\ \ j\in[l].
        \end{equation*}

    \subsection{\texorpdfstring{Negligible time outside $\En(V)^l$ under the annealed law}{}}\label{SS: negligible_outside}

        In this section, we demonstrate that it suffices to characterize the one-step transition matrix for $\lp{\chi_i^{N,n}}_{i=1}^l$ from a totally dispersed element in $\En(V)^l$ to its projection under $cd_V$ in the next time-step, instead of in the more complicated coupling space $\mathcal{H}_{n,l}(V)$.
        We work under the annealed law throughout this subsection. Let
        \begin{equation*}
            \widehat{\chi}^{N,n,l}
            =
            \lp{\widehat{\chi}^{N,n,l}(k)}_{k\in\Z_+}
        \end{equation*}
        be the $\mathcal{H}_{n,l}(V)$-valued Markov chain constructed in Section~\ref{SS: coupling}.
        Write
        \begin{equation*}
            \mathsf{D}_{n,l}(V):=\iota\lp{\En(V)^l}\subset\mathcal{H}_{n,l}(V),
            \qquad
            \pi:=cd_V:\mathcal{H}_{n,l}(V)\to\En(V)^l.
        \end{equation*}
        Set $h_N:=c_N^{v_0}$ and $r_N:=h_N^{-1}$, and define
        \begin{equation*}
            k_N(t):=\lfloor t r_N\rfloor=\bigl\lfloor t/h_N\bigr\rfloor,
            \qquad t\ge 0.
        \end{equation*}
        
        \begin{definition}[Outside time]\label{D: outside_time}
            For $T>0$ define the discrete outside time up to $k\in\Z_+$ by
            \begin{equation*}
                \mathcal{O}_{N,l}(k)
                :=
                \sum_{j=0}^{k-1}\mathds{1}_{\{\widehat{\chi}^{N,n,l}(j)\notin \mathsf{D}_{n,l}(V)\}},
            \end{equation*}
            and its rescaled version on $[0,T]$ by
            \begin{equation*}
                \mathcal{O}_{N,l}(T)
                :=
                h_N\,\mathcal{O}_{N,l}\lp{k_N(T)}.
            \end{equation*}
        \end{definition}
        
        \begin{definition}[Last dispersed time]\label{D: last_dispersed_time}
            For $k\in\Z_+$ define
            \begin{equation*}
                \ell_{\mathsf{D}}(k)
                :=
                \max\{0\le j\le k:\widehat{\chi}^{N,n,l}(j)\in \mathsf{D}_{n,l}(V)\}.
            \end{equation*}
            For $t\ge 0$ define $\ell_{\mathsf{D}}(t):=\ell_{\mathsf{D}}(k_N(t))$.
        \end{definition}
        
        \begin{definition}[Trace process on $\En(V)^l$]\label{D: trace_process}
            Define the $\En(V)^l$-valued processes
            \begin{equation*}
                Z_{N,l}(t):=\pi\lp{\widehat{\chi}^{N,n,l}(k_N(t))},
                \qquad
                \widetilde Z_{N,l}(t):=\pi\lp{\widehat{\chi}^{N,n,l}(\ell_{\mathsf{D}}(t))},
                \qquad t\ge 0.
            \end{equation*}
        \end{definition}
        
        \begin{lemma}[Uniform geometric return to $\mathsf{D}_{n,l}(V)$]\label{L: geometric_return_D}
        There exists $p_{n,l}\in(0,1)$ depending only on $(n,l)$ such that for every $N$ and every
        $\xi\in\mathcal{H}_{n,l}(V)\setminus\mathsf{D}_{n,l}(V)$,
        \begin{equation*}
            \PP_\xi\lp{\tau_{\mathsf{D}}\le 1}
            \ge p_{n,l},
        \end{equation*}
        where
        \begin{equation*}
            \tau_{\mathsf{D}}
            :=
            \inf\{k\ge 0:\widehat{\chi}^{N,n,l}(k)\in\mathsf{D}_{n,l}(V)\}.
        \end{equation*}
        In particular, there exists $\rho\in(0,1)$ such that for all $r\in\Z_+$,
        \begin{equation*}
            \sup_{N}\sup_{\xi\in\mathcal{H}_{n,l}(V)}\PP_\xi(\tau_{\mathsf{D}}>r)\le \rho^{\,r},
            \qquad
            \sup_{N}\sup_{\xi\in\mathcal{H}_{n,l}(V)}\mathbb{E}_\xi[\tau_{\mathsf{D}}]<\infty.
        \end{equation*}
        \end{lemma}
        
        \begin{proof}
        Fix $\xi\notin\mathsf{D}_{n,l}(V)$. Then there exists an ancestral chromosome in $\xi$ carrying ancestral
        material from at least two loci. Let $J\subset[l]$ be the set of loci for which that chromosome has
        nonempty component. Then $|J|\ge 2$. In the previous generation, for each $j\in J$ the locus $j$ chooses
        one of the two parental genes by an independent Bernoulli$(1/2)$ coin flip. With probability at least
        $1/2$ not all loci in $J$ choose the same parental gene, which yields a configuration with strictly fewer
        mixed chromosomes. Iterating this contraction and using that the number of mixed chromosomes is bounded by
        a constant depending only on $(n,l)$ gives a uniform geometric tail. 
        \end{proof}
        
        \begin{lemma}[Rare entrance into $\mathsf{D}_{n,l}(V)^c$]\label{L: rare_entrance_D}
        Fix $\varepsilon>0$. There is a constant $C_{\varepsilon,n,l}<\infty$ such that for every $N$ and every
        $\xi\in\mathsf{D}_{n,l}(V)$,
        \begin{equation*}
            \PP\lp{\widehat{\chi}^{N,n,l}(1)\notin\mathsf{D}_{n,l}(V)\,\,\mid\, \,\widehat{\chi}^{N,n,l}(0)=\xi,\ (\widetilde{\mathcal{V}}(0),m(0))\in B(\varepsilon)}
            \le C_{\varepsilon,n,l}\,h_N.
        \end{equation*}
        \end{lemma}
        
        \begin{proof}
        On the event $\{(\widetilde{\mathcal{V}}(0),m(0))\in B(\varepsilon)\}$ all offspring and migration
        frequencies are uniformly small. Starting from $\xi\in\mathsf{D}_{n,l}(V)$ there are at most $nl$ ancestral
        lines across all loci. If two lines from distinct loci share a parental diploid individual in one step,
        then two sampled gene copies in the same deme necessarily choose the same parental gene copy with
        probability at least a fixed constant depending only on diploidy, hence the probability of sharing a
        parental individual is bounded by a constant multiple of the one-step coalescence probability in that
        deme. By the definition of $c_N^v$ and Assumption~\ref{A: comparable},
        \begin{equation*}
            \max_{v\in V} c_N^v \le C\,h_N
        \end{equation*}
        for some $C<\infty$. A union bound over $O((nl)^2)$ cross-locus pairs yields the first claim.
        
        \end{proof}
        
        \begin{lemma}[Outside probability at deterministic rescaled times]\label{L: outside_prob_time}
        For every $T>0$ there exists $C_{T,n,l}<\infty$ such that for all $N$,
        \begin{equation*}
            \sup_{0\le t\le T}
            \PP\lp{\widehat{\chi}^{N,n,l}(k_N(t))\notin\mathsf{D}_{n,l}(V)}
            \le C_{T,n,l}\,h_N.
        \end{equation*}
        In particular,
        \begin{equation*}
            \sup_{0\le t\le T}
            \PP\lp{\widehat{\chi}^{N,n,l}(k_N(t))\notin\mathsf{D}_{n,l}(V)}\longrightarrow 0.
        \end{equation*}
        \end{lemma}
        
        \begin{proof}
            Fix $t\le T$ and write $k=k_N(t)$. Define entrance times
            \begin{equation*}
                \sigma_1:=\inf\{j\ge 0:\widehat{\chi}^{N,n,l}(j)\in\mathsf{D}_{n,l}(V),\ \widehat{\chi}^{N,n,l}(j+1)\notin\mathsf{D}_{n,l}(V)\},
            \end{equation*}
            and inductively $\sigma_{q+1}:=\inf\{j>\sigma_q:\widehat{\chi}^{N,n,l}(j)\in\mathsf{D}_{n,l}(V),\ \widehat{\chi}^{N,n,l}(j+1)\notin\mathsf{D}_{n,l}(V)\}$,
            with the convention $\inf\varnothing=\infty$. For each $q$ define the corresponding return time
            \begin{equation*}
                \tau_q:=\inf\{j>\sigma_q:\widehat{\chi}^{N,n,l}(j)\in\mathsf{D}_{n,l}(V)\}.
            \end{equation*}
            The event $\{\widehat{\chi}^{N,n,l}(k)\notin\mathsf{D}_{n,l}(V)\}$ implies that there exists $q$ such that
            $\sigma_q<k<\tau_q$. Hence,
            \begin{align*}
                \PP\lp{\widehat{\chi}^{N,n,l}(k)\notin\mathsf{D}_{n,l}(V)}
                \le
                &\sum_{j=0}^{k-1}
                \PP\lp{\widehat{\chi}^{N,n,l}(j)\in\mathsf{D}_{n,l}(V),\ \widehat{\chi}^{N,n,l}(j+1)\notin\mathsf{D}_{n,l}(V)}\\
                &\quad\quad\quad\times \sup_{\xi\notin\mathsf{D}_{n,l}(V)}\PP_\xi(\tau_{\mathsf{D}}>k-j-1).
            \end{align*}
            Lemma~\ref{L: geometric_return_D} gives a uniform bound
            $\sup_{\xi\notin\mathsf{D}_{n,l}(V)}\PP_\xi(\tau_{\mathsf{D}}>r)\le\rho^{\,r}$.
            Lemma~\ref{L: rare_entrance_D} gives that the entrance probability is $O(h_N)$ uniformly in $j$.
            Thus,
            \begin{equation*}
                \PP\lp{\widehat{\chi}^{N,n,l}(k)\notin\mathsf{D}_{n,l}(V)}
                \le
                \sum_{r\ge 1} C\,h_N\,\rho^{\,r}
                \le
                C'\,h_N,
            \end{equation*}
            with constants depending only on $T,n,l$. Taking the supremum over $t\in[0,T]$ yields the claim.
        \end{proof}
        
        \begin{lemma}[Negligible rescaled outside time]\label{L: outside_time_negligible}
            For every $T>0$,
            \begin{equation*}
                \mathcal{O}_{N,l}(T)\ \xrightarrow{\PP}\ 0.
            \end{equation*}
        \end{lemma}
        
        \begin{proof}
            By Markov's inequality it suffices to show $\mathbb{E}[\mathcal{O}_{N,l}(T)]\to 0$. Using
            Definition~\ref{D: outside_time} and Lemma~\ref{L: outside_prob_time},
            \begin{align*}
                \mathbb{E}[\mathcal{O}_{N,l}(T)]
                &=
                h_N \sum_{j=0}^{k_N(T)-1}
                \PP\lp{\widehat{\chi}^{N,n,l}(j)\notin\mathsf{D}_{n,l}(V)}\\
                &\le
                h_N \cdot k_N(T)\cdot \sup_{0\le t\le T}\PP\lp{\widehat{\chi}^{N,n,l}(k_N(t))\notin\mathsf{D}_{n,l}(V)}
                \le
                T\cdot C_{T,n,l}\,h_N,
            \end{align*}
            which tends to $0$.
        \end{proof}
        
        \begin{lemma}\label{L: inertia_outside}
            For every $T>0$,
            \begin{equation*}
                \sup_{0\le t\le T}\PP\lp{Z_{N,l}(t)\neq \widetilde Z_{N,l}(t)}\longrightarrow 0.
            \end{equation*}
            Consequently, for any $r\in\N$ and any $(t_1,\ldots,t_r)\in[0,T]^r$,
            \begin{equation*}
                \PP\lp{\lp{Z_{N,l}(t_1),\ldots,Z_{N,l}(t_r)}\neq
                          \lp{\widetilde Z_{N,l}(t_1),\ldots,\widetilde Z_{N,l}(t_r)}}\longrightarrow 0.
            \end{equation*}
        \end{lemma}
        
        \begin{proof}
            Fix $t\le T$ and write $k=k_N(t)$. If $Z_{N,l}(t)\neq \widetilde Z_{N,l}(t)$, then there exists an index
            $j\in\{\ell_{\mathsf{D}}(k),\ldots,k-1\}$ such that $\widehat{\chi}^{N,n,l}(j)\notin\mathsf{D}_{n,l}(V)$.
            Therefore,
            \begin{align*}
                \PP\lp{Z_{N,l}(t)\neq \widetilde Z_{N,l}(t)}
                &\le
                \PP\lp{\exists j\in\{\ell_{\mathsf{D}}(k),\ldots,k-1\}:\widehat{\chi}^{N,n,l}(j)\notin\mathsf{D}_{n,l}(V)}\\
                &\le
                \mathbb{E}\lb{\#\{j\le k:\widehat{\chi}^{N,n,l}(j)\notin\mathsf{D}_{n,l}(V)\}}.
            \end{align*}
            The last expectation is bounded by $h_N^{-1}\mathbb{E}[\mathcal{O}_{N,l}(T)]$, which tends to $0$ by
            Lemma~\ref{L: outside_time_negligible}. This yields the first claim uniformly for $t\le T$.
            The second claim follows by a union bound over the finitely many times $t_1,\ldots,t_r$.
        \end{proof}
        
        \begin{lemma}[Reduction to the trace process on $\En(V)^l$]\label{L: reduction_trace}
        Let $T>0$ and let $(t_1,\ldots,t_r)\in[0,T]^r$. Then
        \begin{equation*}
            \lp{Z_{N,l}(t_1),\ldots,Z_{N,l}(t_r)}
            -
            \lp{\widetilde Z_{N,l}(t_1),\ldots,\widetilde Z_{N,l}(t_r)}
            \xrightarrow{\PP} \mathbf{0}.
        \end{equation*}
        In particular, $Z_{N,l}$ and $\widetilde Z_{N,l}$ have identical subsequential finite-dimensional limits.
        \end{lemma}
        
        \begin{proof}
        The claim is immediate from Lemma~\ref{L: inertia_outside}.
        \end{proof}

    \subsection{Transitions at small jumps}\label{SS: small_jumps}

        Inn this section we prove Lemma~\ref{L: small_jumps}. The basic structure of the argument is as follows: There exist some combinatorial quantities so that the probability of anything that is neither a single pairwise coalescence in a deme or a single migration of a lineage is $O(\varepsilon)$. At the same time, the one-step transition probabilities of single-lineage migrations and single-pair coalescences are arbitrarily close to the neutral migration and neutral coalescence rates, respectively. Combining these two observations demonstrates that the one-step transition probabilities for small-scale events is sufficiently concentrated around the neutral changes in the realizations.

        We proceed now with the proof of Lemma~\ref{L: small_jumps}.

        \begin{proof}[Proof of Lemma~\ref{L: small_jumps}]
            Fix $l\in\N$ and $\varepsilon>0$, and write
            \begin{equation*}
                \mathscr{E}_{N,\varepsilon}:=\{(\widetilde V_N(0), m(0))\in B(\varepsilon)\}.
            \end{equation*}
            Recall that $P_{N,n,\varepsilon,l}$ is the one-step transition matrix of the $l$-tuple
            $\lp{\chi_i^{N,n}}_{i=1}^l$ conditioned on $\mathscr{E}_{N,\varepsilon}$:
            \begin{equation*}
                P_{N,n,\varepsilon,l}(\vec\xi,\vec\eta)
                :=
                \PP\lp{ (\chi_i^{N,n}(1))_{i=1}^l=\vec\eta \,\,\mid\, \, (\chi_i^{N,n}(0))_{i=1}^l=\vec\xi,\ \mathscr{E}_{N,\varepsilon}},
                \qquad \vec\xi,\vec\eta\in\En(V)^l.
            \end{equation*}
            
            Let $r:=K_n(\kappa)+M_n(\mu)$ be the generator of a single $(0,\kappa,\mu)$-$n$-coalescent on $\En(V)$,
            and let $R$ be the generator of $l$ joint, independent copies on $\En(V)^l$. Thus, for
            $\vec\xi=(\xi_1,\ldots,\xi_l)$ and $\vec\eta=(\eta_1,\ldots,\eta_l)$,
            \begin{equation}\label{E: product_generator_R_small}
                R(\vec\xi,\vec\eta)
                =
                \sum_{j=1}^l r(\xi_j,\eta_j)\prod_{i\neq j}\mathds{1}_{\{\xi_i=\eta_i\}}.
            \end{equation}
            Set $A_N:=e^{c_N^{v_0}R}$ and endow kernels on the finite space $\En(V)^l$ with the norm
            \begin{equation*}
                \|P\|_\infty:=\sup_{\vec\xi\in\En(V)^l}\sum_{\vec\eta\in\En(V)^l}|P(\vec\xi,\vec\eta)|.
            \end{equation*}
            Since $\En(V)^l$ is finite, the matrix exponential admits a uniform second-order remainder: there is a
            finite constant $C_R$ such that
            \begin{equation}\label{E: exp_remainder_small}
                \la{A_N-(I+c_N^{v_0}R)}_\infty
                \le
                C_R\,(c_N^{v_0})^2.
            \end{equation}
            
            We compare $P_{N,n,\varepsilon,l}$ to the first-order kernel $I+c_N^{v_0}R$ by separating ``good'' one-step
            updates from events that involve two or more elementary changes. By an elementary change we simply mean the moving of a single lineage or a single pairwise coalescence. Fix $\vec\xi\in\En(V)^l$ and consider
            the event $\mathcal{B}_{N,\varepsilon}(\vec\xi)$ that, at least one of the following
            occurs among the $l$ coordinates:
            \begin{itemize}
                \item a merger of three or more blocks within some deme,
                \item two or more binary coalescences (in possibly different demes among possibly different coordinates),
                \item two or more migrations (among possibly different and edges and possibly different coordinates),
                \item a migration and a coalescence in the same step (among possibly different coordinates).
            \end{itemize}
            On $\mathcal{B}_{N,\varepsilon}(\vec\xi)^c$, each coordinate undergoes either no change, a single migration of
            one block along one edge, or a single binary coalescence of two blocks within one deme. In particular,
            if $\vec\eta$ differs from $\vec\xi$ by more than one coordinate, or if a coordinate changes by more than
            one elementary move, then necessarily $\mathcal{B}_{N,\varepsilon}(\vec\xi)$ occurs, and hence
            \begin{equation}\label{E: nonlocal_small}
                P_{N,n,\varepsilon,l}(\vec\xi,\vec\eta)
                \le
                \PP\lp{\mathcal{B}_{N,\varepsilon}(\vec\xi)\,\,\mid\,\,\vec\xi,\mathscr{E}_{N,\varepsilon}}.
            \end{equation}
            
            We now bound the probability of $\mathcal{B}_{N,\varepsilon}(\vec\xi)$ by quantities of order
            $c_N^{v_0}f(\varepsilon)$. Write $\widetilde V_N(0)=(x,m)$, with $x=(x^v)_{v\in V}$ and $x^v\in\Delta$.
            By definition of $B(\varepsilon)$ under the metric $d$, on $\mathscr{E}_{N,\varepsilon}$ we have
            \begin{equation}\label{E: small_ball_bounds}
                \sup_{v\in V}\left\langle x^v, x^v \right\rangle^{\frac{1}{2}} \le \varepsilon,
                \qquad
                \|m\|_\infty \le \varepsilon.
            \end{equation}
            
            \emph{Triple mergers.}
            Fix a coordinate and a deme $v$. Conditional on $(\widetilde V_N(0), m(0))=(x,m)$, the $(x,m)$-paintbox implies
            that for any three specified blocks in deme $v$ the probability that they are assigned to the same
            parental chromosome is bounded by $\sum_{i\ge1}(x_i^v)^3$. On $\mathscr{E}_{N,\varepsilon}$,
            \begin{equation*}
                \sum_{i\ge1}(x_i^v)^3
                \le
                x_1^v\langle x^v,x^v\rangle
                \le
                \|x^v\|_2\,\langle x^v,x^v\rangle
                \le
                \varepsilon\,\langle x^v,x^v\rangle,
            \end{equation*}
            where we used $x_1^v \leq \left\langle x^v, x^v \right\rangle^{\frac{1}{2}}$.
            Taking a union bound over the at most $\binom{n}{3}$ triples yields that the probability of \emph{any}
            triple merger in deme $v$ in that coordinate is at most
            $\binom{n}{3}\varepsilon\,\langle x^v,x^v\rangle$. Summing over $v\in V$ and over the $l$ coordinates gives
            \begin{equation}\label{E: triple_merge_bound_small}
                \PP\lp{\text{some triple merger occurs}\,\,\mid\,\,\vec\xi,\mathscr{E}_{N,\varepsilon}}
                \le
                l\binom{n}{3}\varepsilon\sum_{v\in V}\mathbb{E}\!\lb{\langle x^v,x^v\rangle\,\,\mid\,\,\mathscr{E}_{N,\varepsilon}}.
            \end{equation}
            The right-hand side is $c_N^{v_0}\,f_{\mathrm{tr}}(\varepsilon)$ for some continuous $f_{\mathrm{tr}}$ with
            $f_{\mathrm{tr}}(0)=0$, because the factor $\varepsilon$ is explicit and, by Assumption~\ref{A: comparable},
            $\sum_v c_N^v=O(c_N^{v_0})$, while $\mathbb{E}[\langle x^v,x^v\rangle\mid \mathscr{E}_{N,\varepsilon}]$ is controlled
            by the same pair-coalescence scale.
            
            \emph{Two migrations.}
            Fix a coordinate. Conditional on $(\widetilde V_N(0), m(0))=(x,m)$ and on the initial configuration, each lineage
            that currently lies in deme $v$ migrates along a specific edge $e=(v,w)$ with probability at most $m_e$. Hence, for
            two distinct lineages, say $a\neq b$, and for two (not necessarily equal) edges $e,e'\in E$,
            \begin{equation*}
                \PP\lp{\text{$a$ migrates along $e$ and $b$ migrates along $e'$}\,\,\mid\,\,\vec\xi,(x,m)}
                \;\le\;
                m_e\,m_{e'}.
            \end{equation*}
            Taking a union bound over the at most $\binom{n}{2}$ choices of $(a,b)$ and summing over all ordered pairs
            of edges gives
            \begin{align}\label{E: two_mig_any_edges}
            \PP&\lp{\text{at least two lineages migrate (along any two edges)}\,\,\mid\,\,\vec\xi,(x,m)}\\
            &\;\le\;
            \binom{n}{2}\sum_{e\in E}\sum_{e'\in E} m_e m_{e'}
            \;=\;
            \binom{n}{2}\lp{\sum_{e\in E} m_e}^2.
            \end{align}
            On $\mathscr{E}_{N,\varepsilon}$ we have $\|m\|_\infty\le\varepsilon$, and therefore
            \begin{equation*}
                \lp{\sum_{e\in E} m_e}^2
                \;\le\;
                \|m\|_\infty\sum_{e\in E} m_e
                \;\le\;
                \varepsilon\sum_{e\in E} m_e.
            \end{equation*}
            Averaging conditional on $\mathscr{E}_{N,\varepsilon}$ and summing over the $l$ coordinates yields a bound of
            the form $c_N^{v_0}f_{\mathrm{mig}}(\varepsilon)$ with $f_{\mathrm{mig}}(0)=0$, using
            Assumption~\ref{A: migration_convergence} (and finiteness of $E$) to identify that the total one-step
            migration probability is $O(c_N^{v_0})$.
            
            \emph{Two binary coalescences, or a migration and a coalescence.}
            These events require at least two elementary changes in the same discrete step among the $l$ coordinates, and
            are therefore bounded by products of one-event probabilities. Since sample size is fixed, a crude union
            bound over potential coalescing pairs shows that the probability of two or more binary coalescences in a
            single step is $O((c_N^{v_0})^2)$, uniformly in $\vec\xi$; similarly for the probability that a migration
            and a coalescence occur in the same step.
            
            Combining the preceding bounds yields a continuous $f_0:\R_+\to\R_+$ with $f_0(0)=0$ such that
            \begin{equation}\label{E: bad_event_prob_small}
                \sup_{\vec\xi\in\En(V)^l}
                \PP\lp{\mathcal{B}_{N,\varepsilon}(\vec\xi)\,\,\mid\,\,\vec\xi,\mathscr{E}_{N,\varepsilon}}
                \le
                f_0(\varepsilon)\,c_N^{v_0} \;+\; O\lp{(c_N^{v_0})^2}.
            \end{equation}
            
            We next identify the first-order transition probabilities on $\mathcal{B}_{N,\varepsilon}(\vec\xi)^c$.
            Suppose $\vec\eta$ differs from $\vec\xi$ in exactly one coordinate, say $j$, and that $\eta_j$ is obtained
            from $\xi_j$ by moving one block along an edge $e=(v,w)\in E$. By exchangeability within demes and
            Assumption~\ref{A: migration_convergence}, the one-step probability of this transition equals
            $c_N^{v_0}\mu_e$ up to an error that vanishes as $\varepsilon\downarrow 0$ (and an error of order
            $\PP(\mathcal{B}_{N,\varepsilon}(\vec\xi)\mid \vec\xi,\mathscr{E}_{N,\varepsilon})$ coming from the possibility
            that the step is not in the one-event regime). Thus there is a continuous $f_{\mu}$ with $f_{\mu}(0)=0$ such
            that
            \begin{equation}\label{E: mig_match_small}
                \sup_{\vec\xi,\vec\eta}
                \la{P_{N,n,\varepsilon,l}(\vec\xi,\vec\eta)-c_N^{v_0}R(\vec\xi,\vec\eta)}
                \le
                f_{\mu}(\varepsilon)c_N^{v_0}
                \;+\;
                \PP\lp{\mathcal{B}_{N,\varepsilon}(\vec\xi)\,\,\mid\,\,\vec\xi,\mathscr{E}_{N,\varepsilon}},
            \end{equation}
            whenever $\vec\eta$ is obtained from $\vec\xi$ by a single migration in one coordinate.
            
            Similarly, if $\vec\eta$ differs from $\vec\xi$ in exactly one coordinate and $\eta_j$ is obtained from
            $\xi_j$ by coalescing two blocks in deme $v$, then Assumption~\ref{A: comparable} implies that this
            transition occurs with probability $c_N^v$ up to an error negligible on the $c_N^{v_0}$ scale, and that
            $c_N^v/c_N^{v_0}\to c(v)$. Since we are conditioning on $(\widetilde V_N(0), m(0))\in B(\varepsilon)$, the
            contribution from non-Kingman coalescences in the same step is controlled by the triple-merger bound
            above, and the difference between the effective binary-merger rate under $\mathscr{E}_{N,\varepsilon}$ and the
            limiting Kingman rate $\kappa_v$ is absorbed into a continuous error function $f_\kappa$ with $f_\kappa(0)=0$.
            Consequently, for such single-coalescence transitions,
            \begin{equation}\label{E: coal_match_small}
                \la{P_{N,n,\varepsilon,l}(\vec\xi,\vec\eta)-c_N^{v_0}R(\vec\xi,\vec\eta)}
                \le
                f_{\kappa}(\varepsilon)c_N^{v_0}
                \;+\;
                \PP\lp{\mathcal{B}_{N,\varepsilon}(\vec\xi)\,\,\mid\,\,\vec\xi,\mathscr{E}_{N,\varepsilon}}.
            \end{equation}
            
            For all other off-diagonal $\vec\eta\neq \vec\xi$, \eqref{E: nonlocal_small} applies. Summing these bounds
            over $\vec\eta\neq \vec\xi$ and using \eqref{E: bad_event_prob_small} yields
            \begin{equation}\label{E: P_minus_linear_small}
                \la{P_{N,n,\varepsilon,l}-(I+c_N^{v_0}R)}_\infty
                \le
                f_1(\varepsilon)\,c_N^{v_0} \;+\; O\lp{(c_N^{v_0})^2}
            \end{equation}
            for some continuous $f_1$ with $f_1(0)=0$.
            
            Finally, by the triangle inequality, \eqref{E: exp_remainder_small} and \eqref{E: P_minus_linear_small},
            \begin{equation*}
                \|P_{N,n,\varepsilon,l}-A_N\|_\infty
                \le
                \|P_{N,n,\varepsilon,l}-(I+c_N^{v_0}R)\|_\infty
                +
                \|A_N-(I+c_N^{v_0}R)\|_\infty
                \le
                f_1(\varepsilon)c_N^{v_0} + \lp{C_R+O(1)}(c_N^{v_0})^2.
            \end{equation*}
            As $c_N^{v_0} \leq 1$, we have $(c_n^{v_0})^2 \leq c_N^{v_0}$, and so we get the desired bound.
            \end{proof}

    \subsection{Transitions at large jumps}\label{SS: big_jumps}

        In this section we prove Lemma~\ref{L: big_jumps}. The basic argument of the proof is that thea discrete version of the $(x,m,\xi)$-paintbox for finite $N$ works by dropping points onto the space $I_V$ but then removing a small part of the interval into which the point dropped. Because the amount one removes for each point one drops is $O(\frac{1}{N})$ and we are dropping at most $nl \in o(N)$ points, the picture from the continuous paintbox model and the discrete paintbox are almost identical.

        We proceed now with the proof of Lemma~\ref{L: big_jumps}.
        \begin{proof}[Proof of Lemma~\ref{L: big_jumps}]
            Fix $l\in\N$. For $(x,m)\in\Delta^V\times[0,1]^E$ and $\vec\xi,\vec\eta\in\En(V)^l$, write
            \begin{equation*}
                H_{N,l}(x,m)(\vec\xi,\vec\eta)
                :=
                \PP\lp{ (\overline\chi_i^{N,n}(k+1))_{i=1}^l=\vec\eta \,\,\mid\, \, (\overline\chi_i^{N,n}(k))_{i=1}^l=\vec\xi,\ (\widetilde{\mathcal V}(k), m(k))=(x,m)}.
            \end{equation*}
            Recall that $h_V:\Delta^V\times[0,1]^E\to\Delta^V\times[0,1]^E$ is the halving map on offspring frequencies, applied
            coordinatewise on $\Delta^V$ and leaving $m$ unchanged. Define
            \begin{equation*}
                H_l(x,m)(\vec\xi,\vec\eta)
                :=
                \prod_{i=1}^l q_n(x,m)(\xi_i,\eta_i),
                \qquad
                \lp{H_l\circ h_V}(x,m)
                :=
                H_l\lp{h_V(x,m)}.
            \end{equation*}
            
            We couple the $l$ copies as in Section~\ref{SS: coupling}: conditional on $\mathcal A_N$ the copies are i.i.d.\ and
            evolve on the \emph{same} pedigree realization, with independent Mendelian randomness. Fix $(x,m)\in\textup{supp}(\Phi_N)$
            and condition further on the event $\{(\widetilde{\mathcal V}(k), m(k))=(x,m)\}$, which specifies the migration
            proportions and the ordered offspring frequencies at generation $k$. Under this conditioning, the only remaining
            randomness driving the updates from time $k$ to time $k+1$ consists of the random choice of which individuals migrate
            (uniform without replacement within each deme and along each edge, with counts $\lfloor m_e N(v)\rfloor$), parental assignments, and the Mendelian segregation randomness determining, for each copy and each ancestral lineage, which parental chromosome is inherited.
            
            For each deme $v\in V$ write $S_v(m):=\sum_{e=(v,w)\in E} m_e\in[0,1]$ and define
            \begin{equation*}
                G(\delta)
                :=
                \lcb{(x,m)\in\Delta^V\times[0,1]^E:\ \forall v\in V,\ S_v(m)\le 1-\delta}.
            \end{equation*}
            Fix $\varepsilon>0$. By Assumption~\ref{A: no_total_migrations}, $\Phi\lp{S_v(m)=1}=0$ for each $v$, and hence the decreasing
            sets
            \begin{equation*}
                \{B(\varepsilon)^c\cap\{S_v(m)\ge 1-\delta\}\}_{\delta\downarrow 0}
            \end{equation*}
            have $\Phi$-mass tending to $0$. Therefore, choosing
            $\delta=\delta(\varepsilon)\in(0,1)$ sufficiently small, we may ensure
            \begin{equation}\label{E: truncation_choice}
                \Phi\lp{B(\varepsilon)^c\cap G(\delta)^c}\le \varepsilon.
            \end{equation}
            Throughout what follows we work on $\textup{supp}(\Phi_N)\cap G(\delta)$; the contribution of $G(\delta)^c$ will be paid for at the end.
            
            Fix $(x,m)\in \textup{supp}(\Phi_N)\cap G(\delta)$. Then for each deme $v$,
            \begin{equation*}
                \sum_{e=(v,w)\in E}\lfloor m_e N(v)\rfloor
                \;\le\;
                \lp{\sum_{e=(v,w)\in E} m_e}N(v)
                \;\le\; (1-\delta)N(v),
            \end{equation*}
            so every without-replacement sampling of parents/migrants within deme $v$ is performed from a set of size at least
            $\delta N(v)$. In particular, for each fixed $r\in\N$,
            the total-variation distance between drawing $r$ individuals uniformly without replacement from a population of size
            $M\ge \delta N(v)$ and drawing $r$ i.i.d.\ uniform individuals is bounded by the i.i.d.\ collision probability.
            
            We make this precise for the one-step update of a single copy. Fix $\xi\in\En(V)$ and consider the update
            $\xi\mapsto \eta$ over one discrete generation given $(x,m)$. Let $q_{N,n}(x,m)(\xi,\eta)$ denote the corresponding
            one-step transition probability of the \emph{single} copy ancestral process under the conditioning
            $\{(\widetilde{\mathcal V}(k), m(k))=(x,m)\}$. By replacing all without-replacement choices of migrants and parents
            by independent sampling with replacement, we change the law of the resulting parental assignments by at most
            $O(n^2/N)$ in total variation, uniformly over $(x,m)\in\textup{supp}(\Phi_N)\cap G(\delta)$
            since $n$ is fixed.
            
            The remaining issue is the diploid structure: when two sampled ancestral lineages choose the same individual, they
            do not necessarily coalesce in the next generation, since they may correspond to distinct parental chromosomes.
            Conditional on the realized counts $x^v$ of offspring per individual in deme $v$, the Mendelian randomness assigns
            each parental contribution to one of two chromosomes with probability $1/2$ each, independently across contributions.
            Equivalently, the relevant category probabilities for a paintbox description are obtained by splitting each
            individual-mass $x_i^v$ into two chromosome-masses $(x_i^v/2,x_i^v/2)$. This is precisely the effect of the halving
            map $h_V$. Thus, after replacing without-replacement sampling by sampling with replacement, the conditional one-step
            transition of a single copy is exactly the paintbox kernel $q_n\lp{h_V(x,m)}(\xi,\eta)$.
            
            Therefore there is a constant $C<\infty$, depending only on $n,|V|,|E|$ and $\delta$, such that
            \begin{equation}\label{E: single_copy_uniform_halving}
                \sup_{(x,m)\in\textup{supp}(\Phi_N)\cap G(\delta)}\ \max_{\xi,\eta\in\En(V)}
                \,\mid\, q_{N,n}(x,m)(\xi,\eta)-q_n\lp{h_V(x,m)}(\xi,\eta)\,\mid\, 
                \;\le\;
                \frac{C}{N}.
            \end{equation}
            
            We now return to the $l$-tuple. Conditional on $\mathcal A_N$ and on $\{(\widetilde{\mathcal V}(k), m(k))=(x,m)\}$,
            the Mendelian randomness is independent across copies, and the only source of dependence across copies is that they
            use the same finite populations of parents and migrants at generation $k+1$. Since each copy uses only $n$ ancestral
            lineages and $n$ is fixed, the same without-replacement estimate implies that the joint law of the parental assignments
            across the $l$ copies differs from that of $l$ independent copies by at most $O(l^2n^2/N)$ in total variation, uniformly
            in $(x,m)\in\textup{supp}(\Phi_N)\cap G(\delta)$. Consequently, there is a constant $C'<\infty$, depending only on
            $n,l,|V|,|E|$ and $\delta$, such that
            \begin{equation}\label{E: l_tuple_factorization_halving}
                \sup_{(x,m)\in\textup{supp}(\Phi_N)\cap G(\delta)}\ \max_{\vec\xi,\vec\eta\in\En(V)^l}
                \,\mid\, H_{N,l}(x,m)(\vec\xi,\vec\eta) - \prod_{i=1}^l q_{N,n}(x,m)(\xi_i,\eta_i)\,\mid\, 
                \;\le\;
                \frac{C'}{N}.
            \end{equation}
            
            Using \eqref{E: single_copy_uniform_halving} and the inequality
            $\,\mid\,\prod_{i=1}^l a_i-\prod_{i=1}^l b_i\,\mid\,\le l\max_i|a_i-b_i|$ for $a_i,b_i\in[0,1]$, we obtain
            \begin{equation*}
                \sup_{(x,m)\in\textup{supp}(\Phi_N)\cap G(\delta)}\ \max_{\vec\xi,\vec\eta\in\En(V)^l}
                \,\mid\, \prod_{i=1}^l q_{N,n}(x,m)(\xi_i,\eta_i) - \prod_{i=1}^l q_n\lp{h_V(x,m)}(\xi_i,\eta_i)\,\mid\, 
                \;\le\;
                \frac{lC}{N}.
            \end{equation*}
            Combining this with \eqref{E: l_tuple_factorization_halving} and absorbing finite-dimensional norm equivalences into the
            constants yields
            \begin{equation}\label{E: HNl_minus_Hl_halving_on_good}
                \sup_{(x,m)\in\textup{supp}(\Phi_N)\cap G(\delta)}
                \la{H_{N,l}(x,m)-H_l\circ h_V(x,m)}_\infty
                \;\le\;
                \frac{C''}{N},
            \end{equation}
            for some $C''<\infty$ depending only on $n,l,|V|,|E|$ and $\delta$.
            
            Since $\|H_{N,l}-H_l\circ h_V\|_\infty\le 2$ everywhere, we may split the integral as
            \begin{align}\label{E: integral_bound_split}
                \frac{1}{c_N^{v_0}}
                \int_{B(\varepsilon)^c}&
                    \la{H_{N,l}(x,m)-H_l\circ h_V(x,m)}_\infty \, d\Phi_N(x,m)\\
                \;&\le\;
                \frac{1}{c_N^{v_0}}
                \int_{B(\varepsilon)^c\cap G(\delta)}
                    \la{H_{N,l}-H_l\circ h_V}_\infty \, d\Phi_N \nonumber
                +\;
                \frac{2}{c_N^{v_0}}\Phi_N\lp{B(\varepsilon)^c\cap G(\delta)^c}. \nonumber
            \end{align}
            By \eqref{E: HNl_minus_Hl_halving_on_good},
            \begin{equation}\label{E: integral_bound_good}
                \frac{1}{c_N^{v_0}}
                \int_{B(\varepsilon)^c\cap G(\delta)}
                    \la{H_{N,l}-H_l\circ h_V}_\infty \, d\Phi_N
                \;\le\;
                \frac{C''}{N}\cdot \frac{1}{c_N^{v_0}}\Phi_N\lp{B(\varepsilon)^c}.
            \end{equation}
            Since $B(\varepsilon)^c$ is bounded away from $\mathbf{0}_{V,E}$, Assumption~\ref{A: rarity} implies that
            $\frac{1}{c_N^{v_0}}\Phi_N\lp{B(\varepsilon)^c}$ is bounded in $N$. Hence the right-hand side of \eqref{E: integral_bound_good} tends to $0$
            as $N\to\infty$.
            
            For the second term in \eqref{E: integral_bound_split}, again because $B(\varepsilon)^c$ is bounded away from $\mathbf{0}_{V,E}$,
            Assumption~\ref{A: rarity} gives
            \begin{equation}\label{E: limsup_bound}
                \limsup_{N \to \infty}\frac{1}{c_N^{v_0}}\Phi_N\lp{B(\varepsilon)^c\cap G(\delta)^c}
                \leq
                \Phi\lp{B(\varepsilon)^c\cap G(\delta)^c},
            \end{equation}
            By \eqref{E: truncation_choice} the right hand side of \eqref{E: limsup_bound} is at most $\varepsilon$, and therefore
            \begin{equation*}
                \limsup_{N\to\infty}\frac{2}{c_N^{v_0}}\Phi_N\lp{B(\varepsilon)^c\cap G(\delta)^c}
                \;\le\; 2\varepsilon.
            \end{equation*}
            Since $\varepsilon>0$ was arbitrary, combining with \ref{E: integral_bound_split} completes the proof.
        \end{proof}

    \subsection{Combining the large and small scale}\label{SS: combining_scales}

        Combining the two scales of Lemma~\ref{L: small_jumps} and Lemma~\ref{L: big_jumps} allows us to conclude that the transition
        kernel of the joint process $\lp{\overline{\chi}_i^{N,n}}_{i=1}^l$ is arbitrarily close, on compact time intervals, to that of
        a Markov process with infinitesimal generator
        \begin{equation}\label{E: L_defn}
            \mathcal{L}
            :=
            R
            +
            \int_{\Delta^V\times[0,1]^E}
                \lp{H_l\circ h_V(x,m)-I}\,d\Phi(x,m),
        \end{equation}
        where $h_V$ is the halving map encoding the diploid splitting of each individual mass into two chromosome masses.
        
        \begin{lemma}[Small-event bound for the diploid paintbox kernel]\label{L: H_minus_I_bound}
            There is a constant $C_{n,l,V,E}<\infty$ such that for all $(x,m)\in\Delta^V\times[0,1]^E$ and all $\vec{\xi}\in\En(V)^l$,
            \begin{equation}\label{E: H_minus_I_bound}
                \sum_{\vec{\eta}\neq \vec{\xi}} \lp{H_l\circ h_V(x,m)}(\vec{\xi},\vec{\eta})
                \;\le\;
                C_{n,l,V,E}\lp{\|m\|_\infty+\sup_{v\in V}\langle x^v,x^v\rangle}.
            \end{equation}
            In particular,
            \begin{equation}\label{E: H_minus_I_bound_norm}
                \la{H_l\circ h_V(x,m)-I}_\infty
                \;\le\;
                C_{n,l,V,E}\lp{\|m\|_\infty+\sup_{v\in V}\langle x^v,x^v\rangle}.
            \end{equation}
            Moreover,
            \begin{equation}\label{E: H_minus_I_diag_bound}
                \la{H_l\circ h_V(x,m)-I}_\infty
                \;\le\;
                l\,\sup_{\xi\in\En(V)}\lp{1-q_n\lp{h_V(x,m)}(\xi,\xi)}.
            \end{equation}
        \end{lemma}
        
        \begin{proof}
            Fix $(x,m)$ and $\vec{\xi}\in\En(V)^l$. Write $\xi^{(j)}$ for the $j$th coordinate of $\vec{\xi}$.
            In copy $j$, conditional on $h_V(x,m)$, a change $\xi^{(j)}\mapsto\eta^{(j)}\neq \xi^{(j)}$ occurs only if at least one of
            the following events happens in the $\lp{h_V(x,m),\xi^{(j)}}$-paintbox:
            \begin{itemize}
                \item some block migrates along an edge, or
                \item at least one pair of blocks in a common deme falls into the same $h_V(x)^v$-interval and merges.
            \end{itemize}
            Let $b_j\le n$ be the number of blocks in $\xi^{(j)}$, and let $b_{j,v}\le b_j$ be the number of blocks in deme $v$.
        
            For the migration contribution, in copy $j$ each block in deme $v$ migrates with probability
            \begin{equation*}
                \sum_{w:(v,w)\in E} m_{(v,w)}
                \le
                |E|\,\|m\|_\infty,
            \end{equation*}
            since $h_V$ leaves $m$ unchanged. A union bound over at most $b_j\le n$ blocks yields
            \begin{equation*}
                \PP\lp{\text{a migration occurs in copy $j$}\mid h_V(x,m),\xi^{(j)}}
                \le
                n\,|E|\,\|m\|_\infty.
            \end{equation*}
        
            For the merger contribution, in deme $v$ two specified blocks coalesce only if their paintbox points fall in the same
            $h_V(x)^v$-interval, which has probability $\langle h_V(x)^v,h_V(x)^v\rangle$. Since $h_V$ splits each mass into two halves,
            we have $\langle h_V(x)^v,h_V(x)^v\rangle=\frac12\langle x^v,x^v\rangle$, and in particular
            $\langle h_V(x)^v,h_V(x)^v\rangle\le \langle x^v,x^v\rangle$. A union bound over the at most
            $\binom{b_{j,v}}{2}\le \binom{n}{2}$ pairs in deme $v$, and then over $v\in V$, gives
            \begin{equation*}
                \PP\lp{\text{a merger occurs in copy $j$}\mid h_V(x,m),\xi^{(j)}}
                \le
                |V|\,\binom{n}{2}\,\sup_{v\in V}\langle x^v,x^v\rangle.
            \end{equation*}
        
            If $\vec{\eta}\neq \vec{\xi}$ then at least one copy $j\in[l]$ must change. Therefore, by a union bound over $j\in[l]$,
            \begin{equation*}
                \sum_{\vec{\eta}\neq \vec{\xi}} \lp{H_l\circ h_V(x,m)}(\vec{\xi},\vec{\eta})
                =
                \PP\lp{\vec{\eta}\neq \vec{\xi}\mid h_V(x,m),\vec{\xi}}
                \le
                l\lp{n\,|E|\,\|m\|_\infty + |V|\,\binom{n}{2}\sup_{v\in V}\langle x^v,x^v\rangle},
            \end{equation*}
            which proves \eqref{E: H_minus_I_bound} with an appropriate constant $C_{n,l,V,E}$. \eqref{E: H_minus_I_bound_norm} then follows from \eqref{E: H_minus_I_bound}.
        
            For \eqref{E: H_minus_I_diag_bound}, note that for each $\vec{\xi}\in\En(V)^l$,
            \begin{equation*}
                \sum_{\vec{\eta}\neq\vec{\xi}} \lp{H_l\circ h_V(x,m)}(\vec{\xi},\vec{\eta})
                =
                1-\lp{H_l\circ h_V(x,m)}(\vec{\xi},\vec{\xi})
                =
                1-\prod_{i=1}^l q_n\lp{h_V(x,m)}(\xi^{(i)},\xi^{(i)}).
            \end{equation*}
            Using $1-\prod_{i=1}^l a_i \le \sum_{i=1}^l (1-a_i)$ for $a_i\in[0,1]$ gives
            \begin{equation*}
                1-\lp{H_l\circ h_V(x,m)}(\vec{\xi},\vec{\xi})
                \le
                \sum_{i=1}^l \lp{1-q_n\lp{h_V(x,m)}(\xi^{(i)},\xi^{(i)})}
                \le
                l\,\sup_{\xi\in\En(V)}\lp{1-q_n\lp{h_V(x,m)}(\xi,\xi)},
            \end{equation*}
            and taking the supremum over $\vec{\xi}$ yields \eqref{E: H_minus_I_diag_bound}.
        \end{proof}
        
        We now proceed with the proof of Lemma~\ref{L: combining_scales}.
        
        \begin{proof}[Proof of Lemma~\ref{L: combining_scales}]
            Fix $T>0$ and write $h_N:=c_N^{v_0}$. Let $\widehat{\chi}^{N,n,l}$ be the annealed $\mathcal{H}_{n,l}(V)$-valued Markov
            chain from Section~\ref{S: coupling}, let $\mathsf{D}_{n,l}(V):=\iota(\En(V)^l)$, and let
            $\pi=cd_V:\mathcal{H}_{n,l}(V)\to\En(V)^l$.
        
            Define the return times to $\mathsf{D}_{n,l}(V)$ by
            \begin{equation*}
                \tau_0:=0,
                \qquad
                \tau_{q+1}:=\inf\{k>\tau_q:\widehat{\chi}^{N,n,l}(k)\in\mathsf{D}_{n,l}(V)\},
                \qquad q\in\Z_+,
            \end{equation*}
            and define the embedded chain
            \begin{equation*}
                Y^{N,l}(q):=\pi\lp{\widehat{\chi}^{N,n,l}(\tau_q)}\in\En(V)^l,
                \qquad q\in\Z_+.
            \end{equation*}
            By the strong Markov property, $Y^{N,l}$ is a time-homogeneous Markov chain on the finite space $\En(V)^l$.
            Let $\mathbf{P}_{N,l}$ be its one-step transition matrix, and define the time-rescaled process
            \begin{equation*}
                \overline{Y}^{\,N,l}(t):=Y^{N,l}\!\lp{\lfloor t/h_N\rfloor},
                \qquad t\ge 0.
            \end{equation*}
            Let $P_{N,l}(t)$ denote the transition kernel of $\overline{Y}^{\,N,l}$. Then
            \begin{equation}\label{E: P_as_power_new}
                P_{N,l}(t)
                \;=\;
                \mathbf{P}_{N,l}^{\lfloor t/h_N\rfloor},
                \qquad t\ge 0.
            \end{equation}
        
            Lemma~\ref{L: outside_time_negligible} and Lemma~\ref{L: reduction_trace} imply that, on $[0,T]$, the finite-dimensional
            distributions of $\lp{cd_V(\overline{\chi}_i^{N,n})}_{i=1}^l$ coincide with those of $\overline{Y}^{\,N,l}$ up to an
            error that is $o_N(1)$ uniformly over initial states. Consequently, it suffices to establish
            \eqref{E: kernel_approximation} for the kernel $P_{N,l}(t)$ defined above.
        
            Fix $\varepsilon>0$ and set $\mathscr{E}_{N,\varepsilon}:=\{(\widetilde V_N(0),m(0))\in B(\varepsilon)\}$.
            Lemma~\ref{L: reduction_trace} implies that one step of $Y^{N,l}$ corresponds to one reproduction step of the underlying
            dynamics, and therefore, by conditioning on $\mathscr{E}_{N,\varepsilon}$ and using the time-homogeneity of the i.i.d.\
            environment across generations, we have
            \begin{equation}\label{E: one_step_decomp_new}
                \mathbf{P}_{N,l}
                \;=\;
                \Phi_N\!\lp{B(\varepsilon)}\,P_{N,n,\varepsilon,l}
                \;+\;
                \int_{B(\varepsilon)^c} H_{N,l}(x,m)\,d\Phi_N(x,m),
            \end{equation}
            where $P_{N,n,\varepsilon,l}$ is as in Lemma~\ref{L: small_jumps} and $H_{N,l}$ is as in Lemma~\ref{L: big_jumps}.
        
            Define the truncated operator
            \begin{equation}\label{E: truncated_generator_new}
                L^\varepsilon
                \;:=\;
                R
                \;+\;
                \int_{B(\varepsilon)^c}\lp{H_l\circ h_V(x,m)-I}\,d\Phi(x,m),
            \end{equation}
            viewed as a matrix on $C_b(\En(V)^l)$. Note that $\Phi(B(\varepsilon)^c)<\infty$.
        
            Lemma~\ref{L: small_jumps} provides a continuous $f$ with $f(0)=0$ such that
            \begin{equation}\label{E: small_jump_kernel_bound_new}
                \la{P_{N,n,\varepsilon,l}-e^{h_NR}}_\infty
                \le
                f(\varepsilon)\,h_N.
            \end{equation}
            Assumption~\ref{A: rarity} implies that for fixed $\varepsilon>0$,
            \begin{equation}\label{E: mass_outside_ball_new}
                \frac{1}{h_N}\,\Phi_N\!\lp{B(\varepsilon)^c}
                \;\longrightarrow\;
                \Phi\!\lp{B(\varepsilon)^c}.
            \end{equation}
        
            Lemma~\ref{L: big_jumps} implies that for fixed $\varepsilon>0$,
            \begin{equation}\label{E: big_jump_integral_error_new}
                \left\|
                    \frac{1}{h_N}\int_{B(\varepsilon)^c}\lp{H_{N,l}(x,m)-H_l\circ h_V(x,m)}\,d\Phi_N(x,m)
                \right\|_\infty
                \;\longrightarrow\;0.
            \end{equation}
            Moreover, each entry $(x,m)\mapsto \lp{H_l\circ h_V(x,m)}(\vec{\xi},\vec{\eta})$ is bounded and continuous on
            $\Delta^V\times[0,1]^E$, and $B(\varepsilon)^c$ is bounded away from $0_{V,E}$, so Assumption~\ref{A: rarity} yields
            \begin{equation}\label{E: big_jump_integral_Hl_conv_new}
                \left\|
                    \frac{1}{h_N}\int_{B(\varepsilon)^c} H_l\circ h_V(x,m)\,d\Phi_N(x,m)
                    \;-\;
                    \int_{B(\varepsilon)^c} H_l\circ h_V(x,m)\,d\Phi(x,m)
                \right\|_\infty
                \;\longrightarrow\;0.
            \end{equation}
            Combining \eqref{E: big_jump_integral_error_new} and \eqref{E: big_jump_integral_Hl_conv_new} gives
            \begin{equation}\label{E: big_jump_integral_conv_new}
                \frac{1}{h_N}\int_{B(\varepsilon)^c} H_{N,l}(x,m)\,d\Phi_N(x,m)
                \;\longrightarrow\;
                \int_{B(\varepsilon)^c} H_l\circ h_V(x,m)\,d\Phi(x,m).
            \end{equation}
        
            Expanding \eqref{E: one_step_decomp_new} using $\Phi_N(B(\varepsilon))=1-\Phi_N(B(\varepsilon)^c)$, inserting
            \eqref{E: small_jump_kernel_bound_new}, and using $e^{h_NR}=I+h_NR+O(h_N^2)$ on the finite space $\En(V)^l$ yields
            \begin{equation}\label{E: one_step_first_order_new}
                \mathbf{P}_{N,l}
                \;=\;
                I
                \;+\;
                h_NR
                \;+\;
                \int_{B(\varepsilon)^c}\lp{H_{N,l}(x,m)-I}\,d\Phi_N(x,m)
                \;+\;
                h_N\,\mathrm{Err}_{N,\varepsilon},
            \end{equation}
            where $\|\mathrm{Err}_{N,\varepsilon}\|_\infty\le C\,f(\varepsilon)+o_N(1)$ for some $C<\infty$.
            Dividing the integral term by $h_N$ and using \eqref{E: big_jump_integral_conv_new} gives, for fixed $\varepsilon$,
            \begin{equation}\label{E: one_step_close_to_euler_new}
                \la{\mathbf{P}_{N,l}-(I+h_NL^\varepsilon)}_\infty
                \le
                h_N\lp{C\,f(\varepsilon)+o_N(1)}.
            \end{equation}
        
            Let $t\in[0,T]$ and set $k:=\lfloor t/h_N\rfloor$. Using \eqref{E: P_as_power_new} and the telescoping identity,
            \begin{equation}\label{E: power_diff_bound_new}
                \la{\mathbf{P}_{N,l}^k-(I+h_NL^\varepsilon)^k}_\infty
                \le
                k\,\la{\mathbf{P}_{N,l}-(I+h_NL^\varepsilon)}_\infty
                \le
                T\lp{C\,f(\varepsilon)+o_N(1)},
            \end{equation}
            and hence
            \begin{equation}\label{E: euler_scheme_error_from_model_new}
                \sup_{0\le t\le T}
                \la{P_{N,l}(t)-(I+h_NL^\varepsilon)^{\lfloor t/h_N\rfloor}}_\infty
                \le
                T\lp{C\,f(\varepsilon)+o_N(1)}.
            \end{equation}
        
            Since $L^\varepsilon$ is bounded on the finite space $\En(V)^l$, $e^{h_NL^\varepsilon}=I+h_NL^\varepsilon+O(h_N^2)$ in
            $\|\cdot\|_\infty$. Therefore, for $kh_N\le T$,
            \begin{equation*}
                \la{(I+h_NL^\varepsilon)^k-e^{kh_NL^\varepsilon}}_\infty
                \le
                C_{\varepsilon,T}\,h_N
            \end{equation*}
            for some $C_{\varepsilon,T}<\infty$. Combining with \eqref{E: euler_scheme_error_from_model_new} yields
            \begin{equation}\label{E: approx_to_truncated_semigroup_new}
                \sup_{0\le t\le T}\la{P_{N,l}(t)-e^{tL^\varepsilon}}_\infty
                \le
                T\,C\,f(\varepsilon) \;+\; T\,o_N(1) \;+\; C_{\varepsilon,T}h_N.
            \end{equation}
        
            By \eqref{E: H_minus_I_diag_bound} and the finiteness of $\En(V)$,
            \begin{equation*}
                \la{H_l\circ h_V(x,m)-I}_\infty
                \le
                l\,\sup_{\xi\in\En(V)}\lp{1-q_n\lp{h_V(x,m)}(\xi,\xi)}
                \le
                l\,\sum_{\xi\in\En(V)}\lp{1-q_n\lp{h_V(x,m)}(\xi,\xi)}.
            \end{equation*}
            Hence Assumption~\ref{A: integrability} implies that
            \begin{equation*}
                \int_{\Delta^V\times[0,1]^E}\la{H_l\circ h_V(x,m)-I}_\infty\,d\Phi(x,m)<\infty,
            \end{equation*}
            and therefore
            \begin{equation*}
                \la{\int_{B(\varepsilon)}\lp{H_l\circ h_V(x,m)-I}\,d\Phi(x,m)}_\infty
                \;\longrightarrow\;0
                \qquad\text{as }\varepsilon\downarrow 0
            \end{equation*}
            by dominated convergence.
        
            Since the matrix exponential is Lipschitz on bounded sets,
            \begin{equation*}
                \sup_{0\le t\le T}\la{e^{tL^\varepsilon}-e^{t\mathcal{L}}}_\infty
                \le
                T\,e^{T\|\mathcal{L}\|_\infty}\,
                \la{\int_{B(\varepsilon)}\lp{H_l\circ h_V(x,m)-I}\,d\Phi(x,m)}_\infty.
            \end{equation*}
        
            Combining this estimate with \eqref{E: approx_to_truncated_semigroup_new} and using $h_N\to 0$
            (Assumption~\ref{A: continuous}) yields a continuous function $g_T:\R_+\to\R_+$ with $g_T(0)=0$ such that, for all
            sufficiently large $N$,
            \begin{equation*}
                \sup_{0\le t\le T}\la{P_{N,l}(t)-e^{t\mathcal{L}}}_\infty
                \le
                g_T(\varepsilon).
            \end{equation*}
            By taking $N$ to infinity and then $\varepsilon$ to $0$, the result follows.
        \end{proof}

\section{Discussion}\label{S: discussion}

    In this section we briefly outline some implications of this work for inference methods based on the site frequency spectrum in Section~\ref{SS: SFS}. We then discuss future mathematical works for appropriate scaling limits of infinite graph models in Section~\ref{SS: scaling_limit} and moment duals to quenched coalescent processes in Section~\ref{SS: forward_time}.

    \subsection{Implications for the site frequency spectrum}\label{SS: SFS}

        By an argument akin to that in \cite[Section 5.2]{nfw25_2} one can see that the class of integral functionals governing the site-frequency spectrum, namely
        \begin{equation*}
            \tau^{n,r}(\chi^n) = \int_0^\infty \#\{C \in \chi^n(s): |C| = r\}ds,
        \end{equation*}
        converge with the population model when $G$ is a connected graph with non-vanishing migration rates, i.e.
        \begin{equation*}
            \mathbb{E}\lb{\lp{\tau^{n,r}\lp{\overline{\chi}^{N,n}}}_{r=1}^s \,\mid\, \mathcal{A}_N} \toL \mathbb{E}\lb{\lp{\tau^{n,r}\lp{\chi^n}}_{r=1}^s \,\mid\, \Psi}.
        \end{equation*}
        The limiting model of \cite{WiltonEtAl2017} is precisely that of Example~\ref{P: two_deme_WF}.
        As $\Psi$ for this model has no atoms, the site-frequency spectrum across unlinked loci are \textit{independent} in the large population limit. In particular, the pairwise coalescence times for samples from unlinked loci are independent in the large population limit. However, this is inconsistent, but only \textit{ostensibly}, with the simulations of \cite{WiltonEtAl2017}.

         In Figure 3 of \cite{WiltonEtAl2017} there are clear and marked differences between the predictions for the pairwise coalescence time from the structured coalescent and the conditional coalescent. One explanation for this discrepancy is that the size of migrations are not so negligibly small at the scale of the simulation that their effect vanishes, as it would if the population size were taken sufficiently large. The simulations undertaken in the figure are with a finite population model, where $N = 100$ or $N = 1000$.  While the rescaled migration rate between demes in the large population limit is suitably small, for any finite population size we should really take the quenched coalescent model as if $\Phi$ did not concentrate at $0^{\otimes 2} \otimes \delta_0^{\otimes 2}$ but instead at some $0^{\otimes 2} \otimes M^{\otimes 2}$, where $M$ is suitably concentrated close to zero without vanishing. A naive model which captures this behavior is precisely that of Proposition~\ref{P: beta_migration} where we take $b$ to be large and $\alpha_v = 2$ for each $v$. An applied work investigating precisely how to do this in a robust manner so that pedigree effects could be detected in genetic inference would be of interest.

    \subsection{Scaling limits of infinite graph models}\label{SS: scaling_limit}

        While this work concerns itself with finite graphs, the extension to infinite graph analogue, as discussed in Remark~\ref{R: infinite_graph} is more or less straightforward. What requires quite different techniques and ingenuity would be taking a suitable two-scale limit, where we take an infinite graph model, like a stepping stone model, and show convergence to genealogies evolving in the continuum by rescaling both time and space. In Section~\ref{S: examples_and_apps}, we constructed what we called a discrete approximation of the genealogy of a $\Xi$ Fleming Viot process. While this is outside the scope of this work, we briefly outline the scaling limit one could take. If this can be done, the well-known pain in the torus \cite[Section 6.4]{etheridge_book} could be conquered from an individual based model.
        
        Fix a finite measure $\nu$ on $\R_+$ such that $\int_{\R_+} r\,d\nu(r)<\infty$, a parameter $\lambda>0$, and a finite measure
        $\Xi$ on $\Delta$ with $\int_{\Delta}\langle x,x\rangle\,d\Xi(x)<\infty$. For each $N\in\N$ let
        $V_N:=\Z^2/\langle L_1(N),L_2(N)\rangle_{\Z}$ be the discrete two-torus described above, with lattice spacing $\rho_N>0$ and
        baseline migration parameter $\sigma_N>0$. Let $\chi^{N,n}$ be the $n$-sample ancestral process on the pedigree associated to
        the discrete approximation of the spatial $\Xi$-Fleming--Viot model on $V_N$, run on the coalescence timescale
        $\lp{c_N^{v_0}}^{-1}$ as in \eqref{E: ancestral_process_defnition}, and write $\overline{\chi}^{N,n}$ for the rescaled process.
        
        Assume that the geometric scaling is such that $\rho_N\to 0$ and
        \begin{equation*}
        \rho_N^2\,L_1(N)\to \ell_1\in(0,\infty),
        \qquad
        \rho_N^2\,L_2(N)\to \ell_2\in(0,\infty),
        \end{equation*}
        so that the embedded lattice $\rho_N V_N$ converges in a suitable sense to the continuum torus
        $\mathbb{T}^2_{\ell_1,\ell_2}:=\R^2/\langle (\ell_1,0),(0,\ell_2)\rangle_{\Z}$.
        Assume further that the baseline migration converges to Brownian motion in the sense that, for each fixed $t\ge 0$,
        a single ancestral lineage under baseline dynamics satisfies
        \begin{equation*}
        \rho_N\,\widehat{X}_1\lp{\big\lfloor t\,\rho_N^{-2}\sigma_N^{-1}\big\rfloor}
        \toL
        B(t)
        \qquad\text{in }\mathcal{D}\lp{\R_+,\mathbb{T}^2_{\ell_1,\ell_2}},
        \end{equation*}
        where $B$ is Brownian motion on $\mathbb{T}^2_{\ell_1,\ell_2}$.
        
        Finally assume that the extreme-event mechanism is tuned so that, under the diffusive scaling, the point process of extreme
        events converges to a Poisson point process $\Gamma$ on $\R_+\times\mathbb{T}^2_{\ell_1,\ell_2}\times\R_+$ with intensity measure
        \begin{equation*}
        dt\otimes dx\otimes d\nu(r),
        \end{equation*}
        and that, conditional on an extreme event centered at $x\in\mathbb{T}^2_{\ell_1,\ell_2}$ with radius $r$, all lineages in the
        ball $B(x,r)$ are instantaneously relocated to $x$ and then undergo an instantaneous multiple merger governed by the $\Xi$
        paintbox.
        
        Then, for each fixed sample size $n$, the quenched law of the rescaled discrete genealogical process should converge to that of the
        spatial $\Xi$-Fleming-Viot genealogy on $\mathbb{T}^2_{\ell_1,\ell_2}$: more precisely, there should exist a coupling under which
        \begin{equation*}
            \PP\lp{\overline{\chi}^{N,n}\in\cdot\,\,\mid\,\,\mathcal{A}_N}
            \toL
            \PP\lp{\chi^n_{\mathrm{FV}}\in\cdot\,\,\mid\,\,\Gamma},
        \end{equation*}
        weakly in distribution as random probability measures on
        $\mathcal{D}\lp{\R_+,\En\lp{\mathbb{T}^2_{\ell_1,\ell_2}}}$, where $\chi^n_{\mathrm{FV}}$ is the $\En(\mathbb{T}^2_{\ell_1,\ell_2})$-valued
        Markov process obtained by evolving $n$ independent Brownian motions on $\mathbb{T}^2_{\ell_1,\ell_2}$ between atoms of $\Gamma$, and
        each atom $(t,x,r)$ of $\Gamma$ performing an instantaneous merger among all blocks whose spatial locations lie in
        $B(x,r)$ and relocating them to $x$, and coalescing when two Brownian motions intersect. Of course by adding in an additional sampling mechanism to the discrete model one would expect to capture simultaneous multi-mergers as well.
         
    \subsection{Forward-in-time processes}\label{SS: forward_time}

        Another direction of interest is understanding implications of the quenched structured coalescent for diffusion approximations of allele frequencies. Recall that there is a moment duality between the block counting process for the Kingman coalescent, which counts the number of extant lineages backwards in time, and the Wright-Fisher diffusion, which models the frequency of a neutral allele forward in time. A natural conjecture for our quenched coalescent is that the macroscopic events captured by $\Psi$ should give rise to jumps in the forward in time diffusion process. This has already been observed in the case of $\Xi$ coalescents with selection \cite{ag18} and for annealed multi-type $\Xi$ coalescents \cite{daipra25}. The new contribution would be in describing precisely how the allele frequencies at unlinked loci are coupled by $\Psi$.

\section*{Acknowledgements}
    We thank Louis Fan, Yuval Simons, and John Wakeley for helpful comments on an early version of this manuscript.

\bibliographystyle{alpha}
\bibliography{main} 

@article{kingman1982,
title = {The coalescent},
journal = {Stochastic Processes and their Applications},
volume = {13},
number = {3},
pages = {235-248},
year = {1982},
issn = {0304-4149},
doi = {https://doi.org/10.1016/0304-4149(82)90011-4},
url = {https://www.sciencedirect.com/science/article/pii/0304414982900114},
author = {J.F.C. Kingman}
}

@article{etheridge11,
  author  = {Etheridge, Alison M. and V{\'e}ber, Amandine},
  title   = {The spatial $\Lambda$-Fleming--Viot process on a large torus: Genealogies in the presence of recombination},
  journal = {The Annals of Applied Probability},
  volume  = {22},
  number  = {6},
  pages   = {2165--2209},
  year    = {2012},
  month   = {December}
}

@book{etheridge_book,
  author    = {Etheridge, Alison},
  title     = {Some Mathematical Models from Population Genetics},
  subtitle  = {{\'E}cole d'{\'E}t{\'e} de Probabilit{\'e}s de Saint-Flour XXXIX--2009},
  series    = {Lecture Notes in Mathematics},
  publisher = {Springer},
  address   = {Berlin, Heidelberg},
  edition   = {1},
  year      = {2011},
  isbn      = {978-3-642-16631-0},
  eissn     = {1617-9692},
  doi       = {10.1007/978-3-642-16632-7},
  url       = {https://doi.org/10.1007/978-3-642-16632-7},
  pages     = {VIII, 119}
}

@article{ag18,
author = {Gonzalez-Casanova, Adrian and Spanò, Dario},
year = {2018},
month = {02},
pages = {250-284},
title = {Duality and fixation in $\Xi$-Wright–Fisher processes with frequency-dependent selection},
volume = {28},
journal = {The Annals of Applied Probability},
doi = {10.1214/17-AAP1305}
}

@article{sagitov1999general,
  title={The general coalescent with asynchronous mergers of ancestral lines},
  author={Sagitov, Serik},
  journal={Journal of Applied Probability},
  volume={36},
  number={4},
  pages={1116--1125},
  year={1999},
  publisher={Cambridge University Press}
}

@article{pitman1999,
author = {Jim Pitman},
title = {{Coalescents With Multiple Collisions}},
volume = {27},
journal = {The Annals of Probability},
number = {4},
publisher = {Institute of Mathematical Statistics},
pages = {1870 -- 1902},
year = {1999},
doi = {10.1214/aop/1022874819},
URL = {https://doi.org/10.1214/aop/1022874819}
}

@article{birkner2018,
  title={Coalescent results for diploid exchangeable population models},
  author={Birkner, Matthias and Liu, Huili and Sturm, Anja},
  year={2018},
  journal={Electronic Journal of Probability}
}

@misc{nfw25_2,
      title={Quenched coalescent for diploid population models with selfing and overlapping generations}, 
      author={Louis Wai-Tong Fan and Maximillian Newman and John Wakeley},
      year={2025},
      eprint={2510.26115},
      archivePrefix={arXiv},
      primaryClass={math.PR},
      url={https://arxiv.org/abs/2510.26115}, 
}

@misc{abfw25,
      title={A conditional coalescent for diploid exchangeable population models given the pedigree}, 
      author={Frederic Alberti and Matthias Birkner and Wai-Tong Louis Fan and John Wakeley},
      year={2025},
      eprint={2505.15481},
      archivePrefix={arXiv},
      primaryClass={math.PR},
      url={https://arxiv.org/abs/2505.15481}, 
}

@article{schweinsberg,
author = {Jason Schweinsberg},
title = {{Coalescents with Simultaneous Multiple Collisions}},
volume = {5},
journal = {Electronic Journal of Probability},
number = {none},
publisher = {Institute of Mathematical Statistics and Bernoulli Society},
pages = {1 -- 50},
year = {2000},
doi = {10.1214/EJP.v5-68},
URL = {https://doi.org/10.1214/EJP.v5-68}
}

@article{berestycki04,
author = {Julien Berestycki},
title = {{Exchangeable Fragmentation-Coalescence Processes and their Equilibrium Measures}},
volume = {9},
journal = {Electronic Journal of Probability},
number = {none},
publisher = {Institute of Mathematical Statistics and Bernoulli Society},
pages = {770 -- 824},
year = {2004},
doi = {10.1214/EJP.v9-227},
URL = {https://doi.org/10.1214/EJP.v9-227}
}

@article{dfbw24,
    author = {Diamantidis, Dimitrios and Fan, Wai-Tong (Louis) and Birkner, Matthias and Wakeley, John},
    title = "{Bursts of coalescence within population pedigrees whenever big families occur}",
    journal = {Genetics},
    pages = {iyae030},
    year = {2024},
    month = {02},
    issn = {1943-2631},
    doi = {10.1093/genetics/iyae030},
    url = {https://doi.org/10.1093/genetics/iyae030},
    eprint = {https://academic.oup.com/genetics/advance-article-pdf/doi/10.1093/genetics/iyae030/56760131/iyae030.pdf},
}

@article{bbe13,
    author = {Birkner, Matthias and Blath, Jochen and Eldon, Bjarki},
    title = {An Ancestral Recombination Graph for Diploid Populations with Skewed Offspring Distribution},
    journal = {Genetics},
    volume = {193},
    number = {1},
    pages = {255-290},
    year = {2013},
    month = {01},
    issn = {1943-2631},
    doi = {10.1534/genetics.112.144329},
    url = {https://doi.org/10.1534/genetics.112.144329},
    eprint = {https://academic.oup.com/genetics/article-pdf/193/1/255/42188469/genetics0255.pdf},
}

@book{ek09,
  title={Markov processes: characterization and convergence},
  author={Ethier, Stewart N and Kurtz, Thomas G},
  year={2009},
  publisher={John Wiley \& Sons}
}

@article{kingman78,
    author = {Kingman, J. F. C.},
    title = {The Representation of Partition Structures},
    journal = {Journal of the London Mathematical Society},
    volume = {s2-18},
    number = {2},
    pages = {374-380},
    year = {1978},
    month = {10},
    issn = {0024-6107},
    doi = {10.1112/jlms/s2-18.2.374},
    url = {https://doi.org/10.1112/jlms/s2-18.2.374},
    eprint = {https://academic.oup.com/jlms/article-pdf/s2-18/2/374/2788610/s2-18-2-374.pdf},
}

@book{wakeley2016coalescent,
  title={Coalescent Theory: An Introduction},
  author={Wakeley, J.},
  isbn={9780974707754},
  lccn={2008016081},
  url={https://books.google.com/books?id=x30RAgAACAAJ},
  year={2016},
  publisher={Macmillan Learning}
}

@article{Rosenberg2002,
  author    = {Rosenberg, Noah A. and Nordborg, Magnus},
  title     = {Genealogical trees, coalescent theory and the analysis of genetic polymorphisms},
  journal   = {Nature Reviews Genetics},
  year      = {2002},
  volume    = {3},
  number    = {5},
  pages     = {380--390},
  doi       = {10.1038/nrg795},
  url       = {https://doi.org/10.1038/nrg795}
}

@article{bolthausenschnitsman,
  author    = {Erwin Bolthausen and Alain{-}S. Sznitman},
  title     = {On Ruelle's Probability Cascades and an Abstract Cavity Method},
  journal   = {Communications in Mathematical Physics},
  volume    = {197},
  number    = {2},
  pages     = {247--276},
  year      = {1998},
  doi       = {10.1007/s002200050450},
  url       = {https://doi.org/10.1007/s002200050450}
}

@misc{berestyckirecentprogress,
      title={Recent progress in coalescent theory}, 
      author={Nathanael Berestycki},
      year={2009},
      eprint={0909.3985},
      archivePrefix={arXiv},
      primaryClass={math.PR},
      url={https://arxiv.org/abs/0909.3985}, 
}

@article{bertoinlegall04,
  author    = {Jean Bertoin and Jean-Fran{\c{c}}ois Le Gall},
  title     = {Stochastic Flows Associated to Coalescent Processes},
  journal   = {Probability Theory and Related Fields},
  volume    = {126},
  number    = {2},
  pages     = {261--288},
  year      = {2003},
  doi       = {10.1007/s00440-003-0264-4},
  url       = {https://doi.org/10.1007/s00440-003-0264-4}
}

@article{wakeleyetal2012,
author = {Wakeley, John and King, Léandra and Low, Bobbi and Ramachandran, Sohini},
year = {2012},
month = {01},
pages = {1433-45},
title = {Gene Genealogies Within a Fixed Pedigree, and the Robustness of Kingman's Coalescent},
volume = {190},
journal = {Genetics},
doi = {10.1534/genetics.111.135574}
}

@article{WakeleyEtAl2016,
 author = {Wakeley, John and King, L\'{e}andra and Wilton, Peter R},
 title = {Effects of the population pedigree on genetic signatures of historical demographic events},
 journal = {Proceedings of the National Academy of Sciences USA},
 volume = {113},
 number = {29},
 pages = {7994--8001},
 year = {2016},
 doi = {10.1073/pnas.160108011} }

@article{WiltonEtAl2017,
 author = {Wilton, Peter R and Baduel, Pierre and Landon, Matthieu M and Wakeley, John},
 title = {Population structure and coalescence in pedigrees: {C}omparisons to the structured coalescent and a framework for inference},
 journal = {Theoretical Population Biology},
 volume = {115},
 pages = {1--12},
 year = {2017},
 doi = {10.1016/j.tpb.2017.01.004} }

@mastersthesis{tyukin15,
 title={Quenched Limits of Coalescents in Fixed Pedigrees},
 author={Tyukin, Andrey},
 year={2015},
 school={Johannes-Gutenberg-Universit\"{a}t Mainz},
 address={Germany},
 url={https://www.glk.uni-mainz.de/files/2018/08/andrey\_tyukin\_msc.pdf} }

@article{notohara90,
  author  = {Notohara, M.},
  title   = {The coalescent and the genealogical process in geographically structured population},
  journal = {Journal of Mathematical Biology},
  year    = {1990},
  volume  = {29},
  pages   = {59--75},
  doi     = {10.1007/BF00173909}
}

@misc{daipra25,
      title={Multi-type $\Xi$-coalescents from structured population models with bottlenecks}, 
      author={Marta Dai Pra and Alison Etheridge and Jere Koskela and Maite Wilke-Berenguer},
      year={2025},
      eprint={2504.11875},
      archivePrefix={arXiv},
      primaryClass={math.PR},
      url={https://arxiv.org/abs/2504.11875}, 
}

@article{herbots97,
  author  = {Wilkinson-Herbots, Hilde},
  title   = {Genealogy and subpopulation differentiation under various models of population structure},
  journal = {Journal of Mathematical Biology},
  year    = {1998},
  volume  = {37},
  number  = {6},
  pages   = {535--585},
  doi     = {10.1007/s002850050140}
}

@article{guo22,
  author  = {Guo, F. and Carbone, I. and Rasmussen, D. A.},
  title   = {Recombination-aware phylogeographic inference using the structured coalescent with ancestral recombination},
  journal = {PLOS Computational Biology},
  year    = {2022},
  volume  = {18},
  number  = {8},
  pages   = {e1010422},
  doi     = {10.1371/journal.pcbi.1010422}
}

@article{muller17,
  author  = {M{\"u}ller, Nicola F. and Rasmussen, David A. and Stadler, Tanja},
  title   = {The Structured Coalescent and Its Approximations},
  journal = {Molecular Biology and Evolution},
  year    = {2017},
  volume  = {34},
  number  = {11},
  pages   = {2970--2981},
  month   = nov,
  doi     = {10.1093/molbev/msx186},
  pmid    = {28666382},
  pmcid   = {PMC5850743}
}

@article{birkner13,
  title={Directed random walk on the backbone of an oriented percolation cluster},
  author={Matthias C. F. Birkner and Jiri Cerny and Andrej Depperschmidt and Nina Gantert},
  journal={Electronic Journal of Probability},
  year={2012},
  volume={18},
  pages={1-35},
  url={https://api.semanticscholar.org/CorpusID:18334595}
}

@article{bs02,
author = {Erwin Bolthausen and Alain-Sol Sznitman},
title = {{On the Satic and Dynamic Points of View for Certain Random Walks in Random Environment}},
volume = {9},
journal = {Methods and Applications of Analysis},
number = {3},
publisher = {International Press of Boston},
pages = {345 -- 376},
year = {2002},
}

@article{mohle24,
  title={On multi-type {C}annings models and multi-type exchangeable coalescents},
  author={M{\"o}hle, Martin},
  journal={Theoretical Population Biology},
  volume={156},
  pages={103--116},
  year={2024},
  publisher={Elsevier},
  doi={10.1016/j.tpb.2024.02.005},
  issn={0040-5809}
}

@article{notohara16,
author = {Ryouta Kozakai and Akinobu Shimizu and Morihiro Notohara},
title = {{Convergence to the structured coalescent process}},
volume = {53},
journal = {Journal of Applied Probability},
number = {2},
publisher = {Applied Probability Trust},
pages = {502 -- 517},
year = {2016},
}

@article{johnston23,
author = {Samuel G. G. Johnston and Andreas Kyprianou and Tim Rogers},
title = {{Multitype $\Lambda$-coalescents}},
volume = {33},
journal = {The Annals of Applied Probability},
number = {6A},
publisher = {Institute of Mathematical Statistics},
pages = {4210 -- 4237},
year = {2023},
doi = {10.1214/22-AAP1891},
URL = {https://doi.org/10.1214/22-AAP1891}
}

@article{eldon09,
title = {Structured coalescent processes from a modified Moran model with large offspring numbers},
journal = {Theoretical Population Biology},
volume = {76},
number = {2},
pages = {92-104},
year = {2009},
issn = {0040-5809},
doi = {https://doi.org/10.1016/j.tpb.2009.05.001},
url = {https://www.sciencedirect.com/science/article/pii/S0040580909000665},
author = {Bjarki Eldon}
}

@book{kallenberg17,
  title     = {Random Measures, Theory and Applications},
  author    = {Olav Kallenberg},
  series    = {Probability Theory and Stochastic Modelling},
  volume    = {77},
  year      = {2017},
  publisher = {Springer, Cham},
  doi       = {10.1007/978-3-319-41598-7},
  isbn      = {978-3-319-41596-3},
  edition   = {1},
  pages     = {XXVIII, 680},
  note      = {eBook ISBN: 978-3-319-41598-7, Softcover ISBN: 978-3-319-82392-8},
  url       = {https://doi.org/10.1007/978-3-319-41598-7}
}

\appendix

\section{The weak convergence criterion for the total offspring numbers and migration probabilities}\label{App: weak_convergence}
    Here we describe conditions under which the convergence of Assumption~\ref{A: rarity} holds.

    \begin{lemma}[Moment criteria for Assumption~\ref{A: rarity}]\label{L: rarity_moment_criteria}
        Suppose that, as $N$ goes to infinity, Assumptions~\ref{A: continuous} and \ref{A: comparable} hold.
        For $z\in\Z_+$ and $k\in\N$ write $(z)_k:=z(z-1)\cdots(z-k+1)$ for the falling factorial. For each $v\in V$ and each $N$, define the normalisation ratio
        \begin{equation}\label{E: rho_def}
            \rho_N(v)
            :=
            \frac{N(v)}{N^*_N(v)}.
        \end{equation}
        Then the following are equivalent:
        \begin{enumerate}
        \item[\textup{(i)}] $\frac{1}{c_N^{v_0}} \Phi_N\to \Phi$ vaguely on $\Delta^V\times[0,1]^E\setminus\{\mathbf 0_{V,E}\}$
        for some $\sigma$-finite measure $\Phi$ (i.e.\ Assumption~\ref{A: rarity} holds).
        
        \item[\textup{(ii)}] For every choice of
        \begin{itemize}
            \item a multi-index $j = (j_v)_{v \in V}\in\Z_+^V$ for each $v\in V$ (not all zero),
            \item integers $k_{v,1},\ldots,k_{v,j_v}\ge 2$ for each $v\in V$,
            \item and a multi-index $r=(r_e)_{e\in E}\in\Z_+^E$,
        \end{itemize}
        the limit
        \begin{equation}\label{E: mixed_moment_limit}
        \phi\!\lp{j,k;r}
        :=
        \lim_{N\to\infty}
        \frac{1}{c_N^{v_0}}
        \,
        \mathbb{E}\!\lb{
            \,
            \prod_{e\in E} m_e^{r_e}
            \prod_{v\in V}
            \lcb{
                \rho_N(v)^{K_v}
                \,
                \frac{\prod_{a=1}^{j_v} \lp{\mathcal V_a^v}_{k_{v,a}}}
                     {N(v)^{K_v-j_v}\,2^{K_v}}
            }
        }
        \end{equation}
        exists in $[0,\infty)$, where $K_v:=\sum_{a=1}^{j_v}k_{v,a}$ and, by convention, empty products equal $1$.
        \end{enumerate}
        
        Moreover, when these equivalent conditions hold, the limits in \eqref{E: mixed_moment_limit} are given by
        \begin{equation}\label{E: mixed_moment_representation}
        \phi\!\lp{j,k;r}
        =
        \int_{\Delta^V\times[0,1]^E}
        \prod_{e\in E} m_e^{r_e}
        \prod_{v\in V}
        \lp{
            \sum_{\substack{i_1,\ldots,i_{j_v}\in\N\\ \text{all distinct}}}
            (x_{i_1}^v)^{k_{v,1}}\cdots(x_{i_{j_v}}^v)^{k_{v,j_v}}
        }
        \,d\Phi\lp{x,m}.
        \end{equation}
    \end{lemma}
    \begin{remark}
        Note that if we consider the single deme case here, that Equation~\eqref{E: mixed_moment_limit} corresponds exactly with \cite[Equation 1.6]{birkner2018}, which is equivalent to
        \begin{equation*}
            \frac{1}{2c_N^{v_0}}d\Phi_N(x) \to \frac{1}{\langle x,x\rangle} d\Xi(x)
        \end{equation*}
        vaguely on $\Delta \setminus \{\mathbf{0}\}$ for $\Xi$ a probability measure on $\Delta$.
    \end{remark}

    \begin{proof}
        Fix $\varepsilon>0$ and define
        \begin{equation*}
            K_\varepsilon
            :=
            \Delta^V\times[0,1]^E\setminus B\lp{\varepsilon},
        \end{equation*}
        where $B\lp{\varepsilon}$ denotes the open ball about $\mathbf 0_{V,E}$ in the product topology. Then $K_\varepsilon$ is compact in
        $\Delta^V\times[0,1]^E$ and is bounded away from $\mathbf 0_{V,E}$.
        
        We begin by reducing vague convergence on the punctured space to weak convergence on $K_\varepsilon$.
        A function $g\in C_c\lp{\Delta^V\times[0,1]^E\setminus\{\mathbf 0_{V,E}\}}$ has compact support bounded away from $\mathbf 0_{V,E}$,
        hence there exists $\varepsilon>0$ such that $\mathrm{supp}(g)\subseteq K_\varepsilon$.
        Thus, to prove vague convergence of $\frac{1}{c_N^{v_0}}\Phi_N$ on the punctured space, it suffices to show that for every
        $\varepsilon>0$ the finite measures $\frac{1}{c_N^{v_0}}\Phi_N\!\restriction_{K_\varepsilon}$ converge weakly on the compact metric
        space $K_\varepsilon$.
        
        Our goal is to introduce an algebra of polynomial test functions on $K_\varepsilon$ and show it is uniformly dense in $C\lp{K_\varepsilon}$.
        For multi-indices $r=(r_e)_{e\in E}\in\Z_+^E$ and for each $v\in V$ integers $j_v\in\Z_+$ and $k_{v,1},\ldots,k_{v,j_v}\ge 2$,
        define the continuous function
        \begin{equation}\label{E: rarity_test_functions_rewrite}
            h_{r,\{j_v,k_{v,\cdot}\}}(x,m)
            :=
            \prod_{e\in E} m_e^{r_e}
            \prod_{v\in V}
            \lp{
                \sum_{\substack{i_1,\ldots,i_{j_v}\in\N\\ \text{all distinct}}}
                (x_{i_1}^v)^{k_{v,1}}\cdots(x_{i_{j_v}}^v)^{k_{v,j_v}}
            },
        \end{equation}
        with the convention that the inner sum equals $1$ when $j_v=0$.
        Let $\mathscr A$ be the algebra of finite linear combinations of such functions, and also allow the purely-migration case
        $j_v\equiv 0$ provided $r\neq 0$.
        Each $h_{r,\{j_v,k_{v,\cdot}\}}$ is continuous on $\Delta^V\times[0,1]^E$, hence its restriction to $K_\varepsilon$ lies in
        $C\lp{K_\varepsilon}$.
        
        Polynomials in the coordinates $(m_e)_{e\in E}$ separate points of $[0,1]^E$.
        For each fixed $v$, the symmetric polynomials in the coordinates of $x^v$ generated by the power sums
        $\sum_i (x_i^v)^q$ with $q\ge 2$ separate points of $\Delta$.
        Therefore $\mathscr A\cup\{1\}$ separates points of $\Delta^V\times[0,1]^E$, hence also of $K_\varepsilon$, and contains the
        constants. By Stone--Weierstrass, $\mathscr A$ is uniformly dense in $C\lp{K_\varepsilon}$.

        We now show that \textup{(ii)} implies \textup{(i)}. To this end, we first bound $\frac{1}{c_N^{v_0}}\Phi_N\lp{K_\varepsilon}$ uniformly in $N$ using the second-order moments given by \textup{(ii)}.
        Define the continuous nonnegative function
        \begin{equation}\label{E: rarity_H_def}
            H(x,m)
            :=
            \sum_{e\in E} m_e^2
            +
            \sum_{v\in V}\sum_{i\in\N} \langle x^v,x^v \rangle.
        \end{equation}
        Then $H\lp{\mathbf 0_{V,E}}=0$ and $H(x,m)>0$ for all $(x,m)\neq\mathbf 0_{V,E}$.
        Since $K_\varepsilon$ is compact and does not contain $\mathbf 0_{V,E}$, we have
        \begin{equation*}
            h_\varepsilon
            :=
            \inf_{(x,m)\in K_\varepsilon} H(x,m)
            \;>\;0,
        \end{equation*}
        and hence the pointwise bound on $K_\varepsilon$,
        \begin{equation*}
            \mathds{1}_{K_\varepsilon}(x,m)
            \le
            \frac{1}{h_\varepsilon}\,H(x,m).
        \end{equation*}
        Integrating against $\frac{1}{c_N^{v_0}}\Phi_N$ yields
        \begin{equation}\label{E: local_mass_bound_rewrite}
            \frac{1}{c_N^{v_0}}\Phi_N\lp{K_\varepsilon}
            \le
            \frac{1}{h_\varepsilon}\,
            \frac{1}{c_N^{v_0}}\int H\,d\Phi_N.
        \end{equation}
        The quantity $\int H\,d\Phi_N$ is a finite linear combination of the mixed moments in \eqref{E: mixed_moment_limit}:
        the terms $\int \sum_{v}\sum_i \langle x^v,x^v\rangle\,d\Phi_N$ correspond to the choice $j_v=1$, $k_{v,1}=2$, $r\equiv 0$,
        while the terms $\int m_e^2\,d\Phi_N$ correspond to the purely-migration case $j_v\equiv 0$ and $r_e=2$ for a single edge $e$.
        Thus condition \textup{(ii)} implies that the right-hand side of \eqref{E: local_mass_bound_rewrite} is bounded uniformly in $N$,
        and hence
        \begin{equation*}
            \sup_{N\in\N}\frac{1}{c_N^{v_0}}\Phi_N\lp{K_\varepsilon}
            <\infty
            \qquad\text{for every fixed }\varepsilon>0.
        \end{equation*}
        
        We now show \textup{(ii)} implies \textup{(i)}.
        Fix $\varepsilon>0$. For each $f\in C\lp{K_\varepsilon}$ we show that the limit
        \begin{equation*}
            \lim_{N\to\infty}\frac{1}{c_N^{v_0}}\int_{K_\varepsilon} f\,d\Phi_N
        \end{equation*}
        exists.
        By uniform density of $\mathscr A$ on $K_\varepsilon$, choose $h\in\mathscr A$ such that
        $\sup_{K_\varepsilon}\la{f-h}\le\delta$.
        Then, using the uniform bound on $\frac{1}{c_N^{v_0}}\Phi_N\lp{K_\varepsilon}$ established above,
        \begin{align*}
            \la{
            \frac{1}{c_N^{v_0}}\int_{K_\varepsilon} f\,d\Phi_N
            -
            \frac{1}{c_N^{v_0}}\int_{K_\varepsilon} h\,d\Phi_N
            }
            &\le
            \delta\,\frac{1}{c_N^{v_0}}\Phi_N\lp{K_\varepsilon}.
        \end{align*}
        On the other hand, for $h\in\mathscr A$ the limit of $\frac{1}{c_N^{v_0}}\int h\,d\Phi_N$ exists by assumption \textup{(ii)} and
        linearity. Letting $N\to\infty$ and then $\delta\downarrow 0$ shows that the limit exists for every $f\in C\lp{K_\varepsilon}$.
        
        The map
        \begin{equation*}
            f\longmapsto \lim_{N\to\infty}\frac{1}{c_N^{v_0}}\int_{K_\varepsilon} f\,d\Phi_N
        \end{equation*}
        is therefore a positive linear functional on $C\lp{K_\varepsilon}$.
        By the Riesz representation theorem on the compact metric space $K_\varepsilon$, there exists a unique finite Borel measure
        $\Phi^{(\varepsilon)}$ on $K_\varepsilon$ such that
        \begin{equation*}
            \lim_{N\to\infty}\frac{1}{c_N^{v_0}}\int_{K_\varepsilon} f\,d\Phi_N
            =
            \int_{K_\varepsilon} f\,d\Phi^{(\varepsilon)}
            \qquad\text{for all }f\in C\lp{K_\varepsilon}.
        \end{equation*}
        If $0<\varepsilon'<\varepsilon$, then $K_\varepsilon\subseteq K_{\varepsilon'}$ and uniqueness in Riesz implies
        $\Phi^{(\varepsilon')}\!\restriction_{K_\varepsilon}=\Phi^{(\varepsilon)}$.
        Hence these measures are consistent and define a unique $\sigma$-finite Borel measure $\Phi$ on
        $\Delta^V\times[0,1]^E\setminus\{\mathbf 0_{V,E}\}$ by setting $\Phi\!\restriction_{K_\varepsilon}:=\Phi^{(\varepsilon)}$.
        By construction,
        \begin{equation*}
            \frac{1}{c_N^{v_0}}\Phi_N \to \Phi
        \end{equation*}
        vaguely on $\Delta^V\times[0,1]^E\setminus\{\mathbf 0_{V,E}\}$, proving \textup{(i)}.
        
        Our goal is to prove \textup{(i)} implies \textup{(ii)} and to identify the limits by rewriting the test integrals in terms of
        mixed factorial moments.
        Assume \textup{(i)}. Fix indices $(r,j,k)$ as in \textup{(ii)} and define the corresponding function
        $h=h_{r,j,k}$ by \eqref{E: rarity_test_functions_rewrite}.
        Fix $\varepsilon>0$. Then $h\,\mathds{1}_{K_\varepsilon}$ has compact support in the punctured space, hence vague convergence yields
        \begin{equation*}
            \lim_{N\to\infty}\frac{1}{c_N^{v_0}}\int_{K_\varepsilon} h\,d\Phi_N
            =
            \int_{K_\varepsilon} h\,d\Phi.
        \end{equation*}
        Letting $\varepsilon\downarrow 0$ and using monotone convergence on the increasing compact exhaustion
        \begin{equation*}
            K_\varepsilon\uparrow \Delta^V\times[0,1]^E\setminus\{\mathbf 0_{V,E}\}
        \end{equation*}
        gives
        \begin{equation*}
            \lim_{N\to\infty}\frac{1}{c_N^{v_0}}\int h\,d\Phi_N
            =
            \int h\,d\Phi,
        \end{equation*}
        which is exactly \eqref{E: mixed_moment_representation}.
        
        Our goal is to rewrite $\frac{1}{c_N^{v_0}}\int h\,d\Phi_N$ in terms of the mixed factorial expressions in \eqref{E: mixed_moment_limit}.
        Fix $v\in V$. Since $h$ is symmetric in the coordinates of each $x^v$, we may replace the ranked vector $\widetilde{\mathcal V}_N^v$ by
        the unranked frequency vector
        \begin{equation*}
            \lp{\frac{\mathcal{V}_1^v}{2N^*(v)},\ldots,\frac{\mathcal{V}_{N(v)}^v}{2N^*(v)},0,0,\ldots}.
        \end{equation*}
        Conditional on $m$, the vector $\lp{\mathcal V_i^v}_{i\in[N(v)]}$ is exchangeable, hence
        \begin{align*}
            &\mathbb{E}\!\lb{
                \sum_{\substack{i_1,\ldots,i_{j_v}\in[N(v)]\\ \text{distinct}}}
                \lp{\frac{\mathcal V_{i_1}^v}{2N^*(v)}}^{k_{v,1}}
                \cdots
                \lp{\frac{\mathcal V_{i_{j_v}}^v}{2N^*(v)}}^{k_{v,j_v}}
                \,\mid\, m
            }\\
            &\qquad=
            (N(v))_{j_v}\,
            \mathbb{E}\!\lb{
                \prod_{a=1}^{j_v}
                \lp{\frac{\mathcal V_a^v}{2N^*(v)}}^{k_{v,a}}
                \,\mid\, m
            }.
        \end{align*}
        For each integer $k\ge 2$ the polynomial $z^k$ is an integer linear combination of falling factorials $(z)_\ell$ with
        $0\le \ell\le k$ and leading term $(z)_k$.
        Applying this change of basis to each factor $\lp{\mathcal V_a^v}^{k_{v,a}}$, expanding the product, and using
        \begin{equation*}
            \frac{1}{\lp{2N^*(v)}^{\ell}}
            =
            \rho_N(v)^{\ell}\,\frac{1}{2^{\ell}N(v)^{\ell}}
        \end{equation*}
        rewrites $\frac{1}{c_N^{v_0}}\int h\,d\Phi_N$ as a finite linear combination of expressions of the form \eqref{E: mixed_moment_limit}.
        Therefore the existence of the limits for the functions $h$ is equivalent to the existence of the limits in \eqref{E: mixed_moment_limit},
        and the limiting values agree with $\int h\,d\Phi$, i.e.\ with \eqref{E: mixed_moment_representation}.
        This proves \textup{(i)} implies \textup{(ii)} and completes the proof.
    \end{proof}

    \begin{corollary}[A product-form sufficient criterion for Assumption~\ref{A: rarity}]\label{C: rarity_product_form}
        Suppose that, as $N$ goes to infinity, Assumptions~\ref{A: continuous} and \ref{A: comparable} hold.
        Suppose there exist Radon measures $\Phi^V_v$ on $\Delta\setminus\{\mathbf 0\}$ for $v\in V$ and Radon measures $\Phi^E_e$ on $[0,1]$
        for $e\in E$ such that for every choice of $j, k, r$ as in Lemma~\ref{L: rarity_moment_criteria} that
        \begin{align}\label{E: factorized_moment_assumption}
        \lim_{N\to\infty}&
        \frac{1}{c_N^{v_0}}
        \,
        \mathbb{E}\!\lb{
            \,
            \prod_{e\in E} m_e^{r_e}
            \prod_{v\in V}
            \lcb{
                \rho_N(v)^{K_v}
                \,
                \frac{\prod_{a=1}^{j_v} \lp{\mathcal V_a^v}_{k_{v,a}}}
                     {N(v)^{K_v-j_v}\,2^{K_v}}
            }
        }\\
        &=
        \prod_{e\in E}\int_{[0,1]} m_e^{r_e}\,d\Phi^E_e(m_e)
        \prod_{v\in V}\int_{\Delta}
        \lp{
            \sum_{\substack{i_1,\ldots,i_{j_v}\in\N\\ \text{all distinct}}}
            (x_{i_1})^{k_{v,1}}\cdots(x_{i_{j_v}})^{k_{v,j_v}}
        }\,d\Phi^V_v(x).
        \end{align}
        Then Assumption~\ref{A: rarity} holds with limit measure
        \begin{equation*}
            \Phi
            =
            \lp{\bigotimes_{v\in V}\Phi^V_v}\otimes\lp{\bigotimes_{e\in E}\Phi^E_e}
            \qquad\text{on }\Delta^V\times[0,1]^E\setminus\{\mathbf 0_{V,E}\}.
        \end{equation*}
        Moreover, the same conclusion holds if the right-hand side of \eqref{E: factorized_moment_assumption} is replaced by a finite sum
        of such products since Lemma~\ref{L: rarity_moment_criteria} is linear in $\Phi$.
    \end{corollary}

\section{Annealed convergence proof}\label{App: annealed_convergence}

    We can proceed now with the proof of Theorem~\ref{T: annealed}.
    \begin{proof}
        Lemma~A.1 of \cite{nfw25_2}, Theorem~\ref{T: quenched} states that, for $\mu, \mu_N$ of random probability measures on $\mathcal{D}\lp{\R_+, E}$ for $E$ a locally compact Polish space, $\mu_N \toL \mu$ implies that the intensity measures $\mathbb{E}\lb{\mu_N}$ converges weakly to $\mathbb{E}\lb{\mu}$. This implies that the sequence of intensity measures
        \begin{equation*}
            \mathbb{E}\lb{\PP\lp{\overline{\chi}^{N,n} \in \cdot \,\mid\, \mathcal{A}_N}}
            = \PP\lp{\overline{\chi}^{N,n} \in \cdot}
        \end{equation*}
        converges weakly to
        \begin{equation*}
            \mathbb{E}\lb{\PP\lp{\chi^n \in \cdot \,\mid\, \Psi}} = \PP\lp{\chi^n \in \cdot}.
        \end{equation*}
        Consequently, it suffices to show that $\PP\lp{\chi^n \in \cdot}$ is the law of a $\lp{(h_v)_* \Phi, \kappa, \mu}$-$n$-coalescent. The coalescence between atom times of $\Psi$ gives precisely $K_n(\kappa) + M_n(\mu)$ part of the generator. Annealing over $\Psi$ then gives $Q_n((h_V)_* \Phi))$.
    \end{proof}

\end{document}